\documentclass[a4paper]{article}

\usepackage{authblk}
\usepackage[T1]{fontenc}
\usepackage[latin1]{inputenc}
\usepackage{longtable}
\usepackage{textcomp}
\usepackage{amsmath,amssymb,amsthm}
\usepackage{mathtools}
\usepackage{graphicx}
\usepackage{xcolor}
\usepackage{times}
\usepackage{graphicx}
\usepackage{dsfont}
\usepackage{xspace}
\usepackage{color}
\usepackage[hyphens]{url}
\usepackage{enumitem}
\usepackage[linesnumbered,ruled,vlined]{algorithm2e}
\usepackage{tikz,pgfplots}
\usepackage[left=2.5cm,right=2.5cm,top=2cm,bottom=3cm]{geometry}
\usepackage{thmtools}
\usepackage{arydshln}
\usepackage{multirow}
\usepackage{placeins} \usepackage{gensymb} \usepackage{mathrsfs}
\usepackage{lettrine}
\usepackage{marginnote}
\usepackage[nottoc,numbib]{tocbibind}
\usepackage{multirow}
\usepackage{epstopdf}
\usepackage{float}
\usepackage[normalem]{ulem}
\pgfplotsset{compat=1.9}

\pgfplotsset{
        colormap={parula}{
            rgb255=(53,42,135)
            rgb255=(15,92,221)
            rgb255=(18,125,216)
            rgb255=(7,156,207)
            rgb255=(21,177,180)
            rgb255=(89,189,140)
            rgb255=(165,190,107)
            rgb255=(225,185,82)
            rgb255=(252,206,46)
            rgb255=(249,251,14)
        },
    }

\definecolor{myOrange}{RGB}{255, 169, 87 }
\definecolor{myGreen}{RGB}{180, 255, 162  }
\definecolor{myGrey}{RGB}{187, 187, 187  }
\definecolor{myDarkGrey}{RGB}{120, 120, 120  }

\usepackage{makeidx}

\makeindex
\usepackage[totoc]{idxlayout}

\usepackage[plainpages=false,pdfpagelabels,hidelinks]{hyperref}

\newcommand{\yz}{z}
\newcommand{\yc}{c}
\newcommand{\ypath}{\mathbf{y}}
\newcommand{\w}{\mathbf{w}}

\newcommand{\riemann}{\mathrm{R}}
\newcommand{\Riemann}{\mathbf{R}}

\def\XXint#1#2#3{{\setbox0=\hbox{$#1{#2#3}{\int}$}
\vcenter{\hbox{$#2#3$}}\kern-.5\wd0}}

\makeatletter
\DeclareRobustCommand\onedot{\futurelet\@let@token\@onedot}
\def\@onedot{\ifx\@let@token.\else.\null\fi\xspace}

\def\ie{\emph{i.e}\onedot} 
\def\cf{\emph{cf}\onedot}

\def\etal{\emph{et al}\onedot}

\makeatother

\newcommand{\R}{\mathbb{R}}
\newcommand{\N}{\mathbb{N}}

\newcommand\restr[2]{{\left.\kern-\nulldelimiterspace #1 \vphantom{\big|}\right|_{#2}}}

\newcommand{\tr}{\mathrm{tr}} 

\renewcommand{\d}{\,\mathrm{d}} 

\newcommand{\dist}{\mathrm{dist}}

\newcommand{\Vspace}{\mathbf{V}}

\newcommand{\Yspace}{\mathbf{Y}}

\newcommand{\E}{{\mathbf{E}}}

\newcommand{\energy}{\mathcal{W}}

\newcommand{\pathenergy}{\mathcal{E}}
\newcommand{\Pathenergy}{\mathbf{E}}

\newcommand{\metric}{g}

\newcommand{\manifold}{\mathcal{M}}
\newcommand{\y}{y}

\newcommand{\ptransport}{{\mathrm{P}}} 
\newcommand{\Ptransport}{{\mathbf{P}}} 
\newcommand{\cov}{{\frac{D}{\d t}}} 
\newcommand{\covt}{{\frac{D}{\d t}}} 
\newcommand{\covs}{{\frac{D}{\d s}}} 
\newcommand{\Covt}[1]{{\frac{\mathbf{{D}}^{{#1}}}{\d t}}} 
\newcommand{\Covs}[1]{{\frac{\mathbf{D}^{{#1}}}{\d s}}} 

\newcommand{\shell}{s}

\newcommand{\mem}{\mbox{{\tiny mem}}}
\newcommand{\bend}{\mbox{{\tiny bend}}}

\makeatletter
\DeclareRobustCommand\onedot{\futurelet\@let@token\@onedot}
\def\@onedot{\ifx\@let@token.\else.\null\fi\xspace}

\def\ie{\emph{i.e}\onedot} 
\def\cf{\emph{cf}\onedot}

\def\etal{\emph{et al}\onedot}

\makeatother

\makeatletter
\def\namedlabel#1#2{\begingroup#2\def\@currentlabel{#2}\phantomsection\label{#1}\endgroup}
\makeatother

\definecolor{gnuplotbrown}{RGB}{255,76,0}
\definecolor{gnuplotblue}{RGB}{0,0,255}
\definecolor{gnuplotred}{RGB}{255,0,0}
\definecolor{gnuplotpurple}{RGB}{255,2,255}
\definecolor{gnuplotgray}{RGB}{131,128,126}
\definecolor{uniblau}{HTML}{004291}
\definecolor{bigsblau}{HTML}{365079}
\definecolor{bigsblau50}{HTML}{9CA7BC}
\definecolor{bigsblau25}{HTML}{CDD3DD}
\definecolor{uniorangedark}{HTML}{E6B400}
\definecolor{uniorange}{HTML}{FFCB0E}
\definecolor{uniorange!50}{HTML}{FFE586}
\definecolor{uniwhite}{HTML}{FFF2C2}
\definecolor{hcmgruen}{HTML}{567877}
\definecolor{hcmgruen50}{HTML}{AFBDBE}
\definecolor{hcmgruen25}{HTML}{D7DEDE}
\definecolor{himgrau}{HTML}{626566}
\definecolor{himgrau75}{HTML}{949592}
\definecolor{himgrau50}{HTML}{C5C4BE}
\definecolor{himgrau25}{HTML}{F5F5F5}
\definecolor{textgrau}{HTML}{000000}
\definecolor{black}{HTML}{000000}
\definecolor{white}{HTML}{FFFFFF}
\definecolor{lightgray}{gray}{0.9}
\definecolor{lila}{HTML}{FF34B3}
\definecolor{lightblue}{HTML}{99D6FF}
\colorlet{blue}{uniblau}
\colorlet{greyblue}{bigsblau}
\colorlet{hcmgelb}{uniorange}
\colorlet{yellow}{uniorange}
\colorlet{tafelgruen}{hcmgruen}
\colorlet{green}{hcmgruen}

\newsavebox\MBox

\theoremstyle{plain}\newtheorem{theorem}{Theorem}[section]

\newtheorem{corollary}[theorem]{Corollary}

\theoremstyle{remark}

\theoremstyle{definition}

\numberwithin{equation}{section}

\newcommand{\notinclude}[1]{}
\newcommand{\myArg}[2]{[#1, #2]}

\begin{document}

\title{Consistent Curvature Approximation\\ on Riemannian Shape Spaces}

\author[$\dagger$]{Alexander Effland}
\author[$\star$]{Behrend Heeren}
\author[$\star$]{Martin Rumpf}
\author[$\ddag$]{Benedikt Wirth}

\affil[$\dagger$]{Institute of Computer Graphics and Vision, Graz University of Technology, Graz, Austria}
\affil[$\star$]{Institute for Numerical Simulation, University of Bonn, Bonn, Germany}
\affil[$\ddag$]{Applied Mathematics, University of M\"unster, M\"unster, Germany}
\affil[ \ ]{Email: benedikt.wirth@uni-muenster.de}

\date{}

\maketitle

\begin{abstract}
We describe how to approximate the Riemann curvature tensor as well as sectional curvatures on possibly infinite-dimensional shape spaces that can be thought of as Riemannian manifolds. 
To this end, we extend the variational time discretization of geodesic calculus presented in \cite{RuWi12b}, 
which just requires an approximation of the squared Riemannian distance that is typically easy to compute. 
First we obtain first order discrete covariant derivatives via a Schild's ladder type discretization of parallel transport. 
Second order discrete covariant derivatives are then computed as nested first order discrete covariant derivatives.
These finally give rise to an approximation of the curvature tensor.
First and second order consistency are proven for the approximations of the covariant derivative and the curvature tensor.
The findings are experimentally validated on two-dimensional surfaces embedded in $\R^3$.
Furthermore, as a proof of concept the method is applied to the shape space of triangular meshes, and discrete sectional curvature indicatrices are computed on low-dimensional vector bundles. 
\end{abstract}

\section{Introduction}\label{sec:intro}
Over the last two decades there has been a growing interest in modeling general shape spaces as Riemannian manifolds. 
Applications range from medical imaging and computer vision to geometry processing in computer graphics. 
Besides the computation of a rigorous distance between shapes defined as the length of a geodesic curve, 
other tools from Riemannian geometry proved to be practically useful as well.
The geometric logarithm offers a representation of large scale shape variability in a linear tangent space,
the exponential map is a tool for shape extrapolation, and 
parallel transport allows to transfer edited details along animation sequences.

The Riemann curvature tensor represents a higher order characterization of the local geometry of the manifold.
In this paper, we  discuss how to numerically approximate the curvature tensor as well as the sectional curvature, \ie a generalization of the Gau\ss~curvature, on Riemannian manifolds. 
The method is based on a variational time discretization of the Riemannian path energy which has been developed in \cite{RuWi12b} 
and has previously been used to deduce a corresponding discretization of exponential map, logarithm, and  parallel transport. 
It requires a local approximation $\energy$ of the squared Riemannian distance $\dist$ such that
$\energy\myArg{\y}{\tilde\y}=\dist^2(\y,\tilde\y)+O(\dist^3(\y,\tilde\y))$ for $\y$ and $\tilde \y$ being points on the manifold.
In the case of shape spaces such approximations are usually substantially easier to compute 
than the underlying Riemannian metric and the squared distance itself. Our approximation of the Riemann curvature tensor is based on replacing the second order covariant derivatives in its definition by corresponding \emph{covariant difference quotients}.  
We prove the consistency of these difference quotients as well as the resulting approximation of the Riemann curvature tensor.
Furthermore, we experimentally validate the presented approach on embedded surfaces and show its applicability in a concrete example of a high-dimensional shape space. 
More precisely, we consider the space of triangular surfaces along with a physical deformation model. 
In the computer graphics community this type of model is known as \emph{Discrete Shells} and was first introduced by Grinspun \etal \cite{GrHiDeSc03}. 
This approach turned out to be very effective as well as efficient and has been used and modified by various authors. In particular, it has also been investigated in the context of Riemannian shape space theory in \cite{HeRuWa12,HeRuSc14}.
\bigskip

The paper is organized as follows: 
In Section \ref{sec:related} we discuss some examples of shape spaces and different approaches for the computation of curvature with special emphasis on shape spaces.
Section \ref{sec:review} reviews the discrete geodesic calculus from \cite{RuWi12b} which is the starting point of our computational scheme for the Riemann curvature tensor. 
Since the computation of the discrete curvature tensor is based on discrete covariant derivatives, 
we first investigate first and second order approximations of the covariant derivative in Section \ref{sec:cov}. 
Then, in Section \ref{sec:curvatureTensor} we prove consistency estimates for a corresponding first and second order approximation of the  curvature tensor.
Finally, Section \ref{sec:ShellSpace} discusses the application of our approach to the study of curvature on the space of discrete shells as a proof of concept.

\section{Related work} \label{sec:related}
Over the last two decades Riemannian calculus on shape spaces has attracted a lot of attention.  
In particular, it allows to transfer many important concepts from classical geometry to these usually high or even infinite-dimensional spaces.
\paragraph{Some examples of shape spaces.}
Prominent examples with a full-fledged geometric theory are spaces of planar curves with a curvature-based metric \cite{MiMu07}, 
an elastic metric \cite{SrJaJo06}, Sobolev-type metrics \cite{ChKePo05,MiMu07,SuYeMe07,MiMu07, Br16, BaBrHa17} or a Riemannian distance computed 
based on the square root velocity function on the curves \cite{SrKlJo11,Br16}.
Spaces of surfaces have also been considered as Riemannian manifolds.
Sobolev metrics on the space of surfaces have been presented by
Bauer \etal \cite{BaHaMi11}. In \cite{BaHaMi12} they extended the Riemannian approach for the space of curves from \cite{MiMu07} 
introducing a suitable metric on tangent spaces to the space of embeddings or immersions.
In the context of geometry processing Kilian \etal \cite{KiMiPo07}
studied geodesics in the space of triangulated surfaces, where the metric is derived from the in-plane membrane
distortion.  Since then, a variety of other Riemannian metrics have been investigated
on the space of surfaces \cite{LiShDi10,KuKlDi10,BaBr11}.
In~\cite{HeRuWa12,HeRuSc14} a metric was proposed that measures membrane distortion as well as bending. 
In image processing the large deformation diffeomorphic metric mapping (LDDMM) framework is based on the theory of diffeomorphic flows:
Dupuis \etal \cite{DuGrMi98} showed that the associated flow is actually a flow of diffeomorphisms.
In \cite{HaZaNi09}, Hart \etal exploited the optimal control perspective on the LDDMM model with the motion field as the underlying control.
Vialard \etal \cite{ViRiRu12a,ViRiRu12} studied methods from optimal control theory
to accurately estimate the initial momentum of the flow and to relate it to the Hamiltonian formulation of geodesics. 
Lorenzi and Pennec \cite{LoPe13} applied the \mbox{LDDMM} framework to compute geodesics and parallel transport using Lie group methods.
The metamorphosis model \cite{MiYo01,TrYo05} generalizes the flow of diffeomorphism approach by allowing additional intensity variations along the flow of diffeomorphically deformed images.
\medskip

In this paper we aim at the numerical computation of the Riemann curvature tensor and sectional curvatures on Riemannian manifolds. 
Our focus is on high-dimensional shape manifolds and not on low-dimensional embedded 
surfaces even though we show numerical experiments for a two-dimensional embedded surface to confirm our theoretical findings. 
Nevertheless, let us briefly review some of the numerous approaches for the approximation of curvature on two-dimensional discrete surfaces.

\paragraph{Numerical approximation of curvature on discrete 2D surfaces.}
One of the earliest approaches to numerically estimate the curvature tensor at the vertices of a polyhedral approximation of a surface was proposed by Taubin \cite{Ta95a}.
In detail, this ansatz exclusively works for embedded two-dimensional manifolds in $\R^3$ and relies on the observation that the directional curvature function
corresponds to a quadratic form, which can be estimated with linear complexity and from which the curvature tensor can be retrieved.
Starting from the triangulation of a two-dimensional manifold embedded in $\R^3$, Meyer~\etal \cite{MeDeScBa} defined discrete operators
representing, for instance, the mean and the Gaussian curvature by spatial averaging of suitably rescaled geometric quantities.
Using a mixed finite element/finite volume discretization as well as a Voronoi decomposition principal curvatures can be robustly estimated.
Cohen-Steiner and Morvan \cite{CoMo03} proposed an integral approximation of the curvature tensor on smooth or polyhedral surfaces
using normal cycles and proved its linear convergence if the polyhedral triangulation is Delaunay.
Hildebrandt~\etal \cite{HiPoWa06} showed that if a sequence of polyhedral surfaces isometrically embedded in $\R^3$ converges
to a differentiable manifold~$\manifold$ with respect to the Hausdorff distance, then the convergence of the normal field is equivalent to the convergence of the metric tensor, which
itself is equivalent to the convergence of the surface area. In addition, under suitable assumptions they established the convergence of geodesic curves
and the mean curvature functionals on discrete surfaces to their limit counterparts on $\manifold$.
In \cite{KaSiNo07}, Kalogerakis~\etal approximated the second fundamental form of a possibly noisy surface using M-estimation, which amounts to
an efficient data fitting approach using the method of iteratively reweighted least squares. Here, the surface is either represented by
polygonal meshes or by point clouds.
Hildebrandt and Polthier \cite{HiPo11} introduced generalized shape operators for smooth and polyhedral surfaces as linear operators
on Sobolev spaces and provided error estimates to approximate the generalized shape operator on smooth surfaces by that on polyhedral surfaces,
from which several geometric quantities can be recovered.
Starting from a weak formulation of the Ricci curvature, Fritz \cite{Fr13} employed a surface finite element method to approximate
the Ricci curvature on isometrically embedded hypersurfaces $\manifold$ and
proved that the rate of convergence is $\frac{2}{3}$ and $\frac{1}{3}$ with respect to the $L^2(\manifold)$- and $H^1(\manifold)$-norm, respectively.

\paragraph{Explicit computation of curvature on shape spaces.}
Michor and Mumford \cite{MiMu04} studied the Riemann curvature tensor 
on the space of smooth planar curves and gave an explicit formula for the sectional curvature. 
They could show that for large, smooth curves all sectional curvatures turn out to be nonnegative, 
whereas for curves with high frequency perturbation sectional curvatures are nonpositive.
Younes \etal \cite{YoMiSh08} investigated a metric on the space of plane curves derived as the limit case of 
a scale invariant metric of Sobolev order 1 from \cite{MiMu07} which allows the explicit computation of geodesics
and sectional curvature.
Micheli \etal \cite{MiMiMu12} showed how to compute sectional curvature on the space of landmarks with a metric induced by the flow of diffeomorphisms.
They were able to evaluate the sectional curvature for pairs of tangent directions to special geodesics along which only two landmarks actually move.
A formula for the derivatives of the inverse of the metric is at the core of this approach. 
In \cite{MiMiMu13} the formula for the sectional curvature on the landmark space was generalized to special infinite-dimensional weak Riemannian manifolds.
A brief overview over the geometry of shape spaces with a particular emphasis on sectional curvature in different shape spaces 
can be found in the contribution by Mumford \cite{Mu12}.

\section{A brief review of continuous and discrete geodesic calculus} \label{sec:review}
To keep this paper self-contained, we first collect all required ingredients of the discrete geodesic calculus introduced in \cite{RuWi12b}
and state our overall assumptions. Let us remark that some results reviewed here hold also under weaker assumptions.
For a comprehensive discussion we refer to \cite{RuWi12b}.
Let $\Vspace$ be a separable, reflexive Banach
space that is compactly embedded in a Banach space $\Yspace$. Let $\manifold$ be the weak closure of an open path-connected subset
of $\Vspace$, potentially with a smooth boundary. We consider $(\manifold, g)$ as a Riemannian manifold with metric
 $g : \manifold \times \Vspace\times\Vspace\to\R$, which is uniformly bounded and 
$\Vspace$-coercive in the sense $c^\ast \|v\|^2_\Vspace \leq g_\y(v,v) \leq C^\ast \|v\|^2_\Vspace$. Furthermore, $g$ is $C^3(\Yspace\times\Vspace\times\Vspace;\R)$-smooth.
Here, the tangent space $T_\y \manifold$ is identified with $\Vspace$.

\paragraph{Geodesic paths.} A geodesic path $\ypath=(\ypath(t))_{t\in [0,1]}$ on $\manifold$ is defined as a local minimizer of the path energy
\begin{align}\label{eq:contenergy}
\pathenergy[\ypath] = \int_0^1 g_{\ypath(t)}(\dot\ypath(t),\dot\ypath(t)) \d t
\end{align}
for fixed end positions $\ypath(0)=\y_A$ and $\ypath(1) = \y_B$ with $\y_A,\y_B\in\manifold$.  Then the Riemannian distance $\dist(\y_A,\y_B)$ is given as the square root of the minimal path energy.
Given $v\in \Vspace$, one may ask for the end point $\ypath(1)$ of a geodesic $(\ypath(t))_{t\in[0,1]}$ with $\ypath(0) = y$ and $\dot \ypath(0) = v$.
This map is called the exponential map with $\exp_y(v):= \ypath(1)$.

The discrete geodesic calculus is based on a local, usually easily computable approximation of the squared Riemannian distance $\dist^2$
by a smooth functional $\energy: \manifold\times\manifold \to \R$ with an extension of $\energy$ to $\Yspace\times\Yspace$ by $\infty$. 
$\energy$ is assumed to be weakly lower semi-continuous  and  locally consistent with the squared Riemannian distance in the sense 
\begin{align} 
\energy\myArg{\y}{\tilde\y}=\dist^2(\y,\tilde\y)+O(\dist^3(\y,\tilde\y))\,. \label{eq:consW}
\end{align}
The consistency implies
$\energy\myArg{\y}{\y} = 0\,,\quad \energy_{,2}\myArg{\y}{\y}(v) = 0\,,\quad \energy_{,22}\myArg{\y}{\y}(v,w) = 2 g_\y(v,w)$ for any $v,w\in\Vspace$.
Here, $\energy_{,j}$ denotes the derivative with respect to the $j$th component, and in analogy $\energy_{,ij}$ the second order derivative in the $i$th and $j$th component.
Furthermore, we also observe~\cite[Lemma 4.6]{RuWi12b} that $\energy_{,1}\myArg{\y}{\y}(v) = 0$ and
\begin{align}\label{eq:secondDerivIdentities}
 \energy_{,11}\myArg{\y}{\y}(v,w) = -\energy_{,12}\myArg{\y}{\y}(v,w) = -\energy_{,21}\myArg{\y}{\y}(v,w) = \energy_{,22}\myArg{\y}{\y}(v,w)\,.
\end{align}
Differentiating \eqref{eq:secondDerivIdentities} once again we achieve for any $u,v,w\in\Vspace$ that 
\begin{align}\label{eq:thirdDerivSymm}
\begin{split}
 &\energy_{,221}\myArg{\y}{\y}(u,v,w)+\energy_{,222}\myArg{\y}{\y}(u,v,w)=\energy_{,111}\myArg{\y}{\y}(u,v,w)+\energy_{,112}\myArg{\y}{\y}(u,v,w)\\
=&-\energy_{,121}\myArg{\y}{\y}(u,v,w)-\energy_{,122}\myArg{\y}{\y}(u,v,w)=-\energy_{,211}\myArg{\y}{\y}(u,v,w)-\energy_{,212}\myArg{\y}{\y}(u,v,w)\,.
\end{split}
\end{align}
With the function $\energy$ at hand one can define the discrete path energy 
\begin{align}
\E[(\y_0,\ldots,\y_K)] =K\sum_{k=1}^K \energy\myArg{\y_{k-1}}{\y_k} \label{eq:discretePathEnergy}
\end{align}
on a discrete $K$-path $(\y_0,\ldots,\y_K)\in\manifold^{K+1}$.
A \emph{discrete geodesic} of order $K$ (or $K$-geodesic) is then defined as a local minimizer of $\Pathenergy[(\y_0,\ldots,\y_K)]$ for fixed end points
$\y_0 = \y_A$ and $\y_K = \y_B$. The Euler--Lagrange conditions of a discrete geodesic are 
\begin{align}\label{eq:EL}
\energy_{,2}[\y_{k-1},\y_{k}] + \energy_{,1}[\y_{k},\y_{k+1}] =0
\end{align}
for $k=1,\ldots, K-1$.
It is shown in \cite[Theorems~4.1 \& 4.2]{RuWi12b} that under the above assumptions 
continuous geodesics exists and are locally unique (for $\dist(\y_A,\y_B)$ sufficiently small).
Furthermore, \cite[Theorems~4.3 \& 4.7]{RuWi12b} ensure the existence of discrete geodesics and their local uniqueness.  
Finally, using a suitable extension via piecewise geodesic paths it is shown in \cite[Theorem~4.8]{RuWi12b} that the discrete 
path energy $\Gamma$--converges to the continuous path energy for $K\to \infty$ 
and that due to equi-coerciveness discrete (extended) minimizers converge to a continuous minimizer. 
If $(\y_0, \y_1, \ldots, y_K)$ is a discrete $K$-geodesic connecting  $\y_0 = \y_A$ and $\y_K = \y_B$, the displacement $K (\y_1-\y_0) \in \Vspace$ is considered as the discrete logarithm of $\y_B$ at $\y_A$.
Vice versa, if $\y_0$ and $\y_1 =\y_0+v$ are given for $v\in\Vspace$ one can iteratively solve  \eqref{eq:EL} for $\y_{k+1}$ given $\y_k$ and $\y_{k-1}$ 
and compute a discrete exponential map which is consistent with the discrete logarithm by construction.
For the analysis of convergence to the corresponding continuous counterparts we refer the reader to \cite[Theorems~5.1 \& 5.10]{RuWi12b}.

\paragraph{Covariant derivative.} 
The differentiation of vector fields on manifolds leads to the notion of the covariant derivative. 
Before we proceed, let us fix the notation and introduce some abbreviations. 
In the following, let $\y \in \manifold$ be an arbitrary but fixed point and $v\in \Vspace$ an arbitrary but fixed tangent vector. 
For some $\epsilon > 0$, let $\ypath = (\ypath(t))_{t\in (-\epsilon,\epsilon)}$ be a path with $\ypath(0) = \y$ and $\dot \ypath(0) = v$. 
Furthermore, we consider a vector field $w$ along the path, \ie $t \mapsto w(\ypath(t))$ for all $|t| < \epsilon$, 
and the corresponding covariant derivative $t \mapsto \cov w(\ypath(t))$, which is again a vector field along the path for all $|t| < \epsilon$. 
The covariant derivative $\cov w$ of $w$ along the path $\ypath$  is uniquely defined -- due to the coercivity of the metric -- by 
 \begin{align}\label{eq:defCovDeriv}
 g_{\ypath(t)}\left(\cov (w \circ \ypath)(t) , z\right) = g_{\ypath(t)}\left(\tfrac{\mathrm{d}}{\mathrm{d} t}(w \circ \ypath)(t), z\right)+ g_{\ypath(t)}(\Gamma_{\ypath(t)}(w\circ \ypath(t),\dot \ypath(t)),z) 
 \end{align}
for all $z\in \Vspace$ and all $t\in (-\epsilon,\epsilon)$, where the dot expresses differentiation with respect to $t$.
Here, the Christoffel operator  $\Gamma: \manifold \times \Vspace\times \Vspace \to \Vspace;\; (\y,v,w) \mapsto \Gamma_\y(v,w)$ is defined by
\begin{align}
2 g_\y(\Gamma_{\y}(v,w),z) = \left(D_\y g_\y\right)(w)(v,z) - \left(D_\y g_\y\right)(z)(v,w) + \left(D_\y g_\y\right)(v)(w,z)
\end{align}
for all tangent vectors $z \in \Vspace$. 
Note that $\cov (w\circ \ypath)(0)$ does depend on the point $\y=\ypath(0)$ and the direction $v \in \Vspace$ but not on the specific choice of the path,
for which reason we will also use the short notation\smallskip

\centerline{
$
\cov w(\y)
\quad\text{instead of}\quad
\cov(w\circ\ypath)(0)
$
}
\smallskip

\noindent whenever the tangent vector $v$ is clear from the context.
A vector field $w$ is parallel along a curve $\ypath$ if $ \cov  (w \circ \ypath)(t)=0$ for all $t$. 
For given $w(\ypath(0))$ a parallel vector field along $\ypath$ can be generated solving the differential equation $\frac\d{\d t}(w\circ \ypath)(t) = - \Gamma_{\ypath(t)}(w\circ \ypath (t),\dot \ypath(t))$ with initial data $w(\ypath(0))$. 
The associated parallel transport map $\ptransport_{\ypath(\tau \leftarrow 0)} w(\ypath(0)) = w(\ypath(\tau))$, which maps initial data $w(\ypath(0)) \in \Vspace$ to output data $w(\ypath(\tau))\in \Vspace$ for $\tau\in \R$ and 
$w$ parallel along $\ypath$, is a linear isomorphism from $\Vspace$ to $\Vspace$.
Here, we use the notation $\ypath(\tau \leftarrow 0)$ for the curve segment $s \mapsto \ypath(s\tau)$ for $s \in [0,1]$.
The covariant derivative at $\y = \ypath(0)$ can be written as the limit of difference quotients, \ie
\begin{align}
\cov (w\circ \ypath)(0) = \lim_{\tau \to 0} \frac{ \ptransport_{\ypath(\tau  \leftarrow 0)}^{-1} \tau w(\ypath(\tau)) - \tau w(\ypath(0))}{\tau^2}\,,
\end{align}
where $\ptransport_{\ypath(\tau  \leftarrow 0)}^{-1} = \ptransport_{\ypath(0 \leftarrow \tau)}$
(the above notation is chosen for consistency with the discrete approximation to be introduced further below).

A discrete parallel transport can be defined via an iterative construction of (discrete) geodesic parallelograms called Schild's ladder. 
However, for the definition of discrete curvature it suffices to introduce only a single step of this scheme. 
For $\y \in \mathring{\manifold}$ and sufficiently small  vectors $w,\, v\in \Vspace$ such a single step corresponds to the discrete transport of $w$ along the line segment from $\y$ to $\y+v$.

The single step of Schild's ladder is achieved by constructing a discrete geodesic parallelogram $\y,\y+v,z,\y+w$
(where $(\y,\y+v)$, $(\y,\y+w)$ and $(\y+v,z)$ are viewed as discrete $1$-geodesics)
and taking $z-(\y+v)\in\Vspace$ as the result of the discrete transport.
In the following we explain this procedure in detail. 
First, we compute a discrete $2$-geodesic $(\y+w, \yc, \y+v)$, where existence and uniqueness of $\yc$ follows from  \cite[Theorems~4.3 \& 4.7]{RuWi12b} for $\|v\|_\Vspace$ and $\|w\|_\Vspace$ sufficiently small.
Then, we determine $z \in \manifold$ such that $(\y,\yc, z)$ is a discrete $2$-geodesic.
This time, existence and uniqueness is ensured by \cite[Lemma 5.6]{RuWi12b} for $\|v\|_\Vspace,\,\|w\|_\Vspace$ sufficiently small. 
Note that the constructed $2$-geodesics form the diagonals of the parallelogram with $c$ being the center point. 
Finally, the discrete transport of $w$ along the line segment from $y$ to $y+v$ is given by
\begin{align}\label{eq:Ptransport}
\Ptransport_{\y+v,\y} w \coloneqq z - (\y+v)\,.
\end{align}
The system of Euler--Lagrange equations associated with \eqref{eq:Ptransport} reads
\begin{align}
\label{eq:ELeqTransp}
\begin{split}
\energy_{,2}[\y+w,\yc] + \energy_{,1}[\yc,\y+v] &=0,\\
\energy_{,2}[\y,\yc] + \energy_{,1}[\yc,z] &=0
\end{split}
\end{align}
for given $\y,\,v,\, w$ and unknown $\yc$ and $z$. 
Iterating this scheme along a polygonal curve on $\mathring{\manifold}$ one can define a consistent discrete transport of a vector along that polygonal curve.
Let us remark that we actually need to introduce a scaling of the vector field $w$ and $v$ by $\tau$ before being transported and the corresponding rescaling afterwards to ensure consistency (\cf \cite[Theorem 5.11]{RuWi12b}). 
The transport map \eqref{eq:Ptransport} is invertible for small~$w$ and~$v$.
Note that in general $\Ptransport_{\y+v,\y}^{-1}$ does not coincide with $\Ptransport_{\y,\y +v}$ 
because $\energy$ is not assumed to be symmetric.
With this discrete parallel transport at hand we can finally define a one-sided covariant difference quotient of a general vector field $w$ along the (polygonal) path $\ypath(t)=\y+tv$ by
\begin{align} \label{eq:CovPlus}
\Covt{\tau} w(\y)\coloneqq\frac{ \Ptransport_{\y+\tau v,\y}^{-1} \tau w(\y+\tau v) - \tau w(\y)}{\tau^2}
\end{align}
(again we suppress the vector $v\in\Vspace$ in the notation since it will always be clear from the context).
Depending on the sign of $\tau$ this is a forward ($\tau >0$) or a backward ($\tau <0$) difference quotient.
Computing \eqref{eq:CovPlus} is based on solving
the following system of Euler--Lagrange equations associated with the \emph{inverse} discrete parallel transport $\Ptransport_{\y+\tau v,\y}^{-1}(\tau w) = z -y$ for given $\y,\,v,\, w$ and unknowns $\yc,z$:
\begin{align}\label{eq:ELeqInverseTransp}
\begin{split}
 \energy_{,2}[z,\yc] + \energy_{,1}[\yc,\y+\tau v] &=0,\\
\energy_{,2}[\y,\yc] + \energy_{,1}[\yc,\y+\tau (v+w)] &=0.
\end{split}
\end{align} 

\section{Consistency of  discrete covariant derivatives}\label{sec:cov}
Instead of the one-sided covariant difference quotient defined in \eqref{eq:CovPlus} one can also consider  a central  covariant difference quotient of a vector field $w$ along the (polygonal) path $\ypath(t)=\y+tv$ by
\begin{align}\label{eq:CovCentral}
\Covt{\pm\tau} w(\y) = \frac{ \Ptransport_{\y+\tau v,\y}^{-1} \tau w(\y+\tau v) + \Ptransport_{\y-\tau v,\y}^{-1} \left(- \tau w(\y-\tau v) \right)}{2\tau^2}\, .
\end{align}
Above, $\Ptransport_{\y-\tau v,\y}^{-1} \left(- \tau w(\y-\tau v)\right)$ is the appropriately reflected construction of $\Ptransport_{\y+\tau v,\y}^{-1} \tau w(\y+\tau v)$.
Note that it does in general not equal $- \Ptransport_{\y-\tau v,\y}^{-1} \left( \tau w(\y-\tau v)\right)$.
Theorem \ref{lemma:consistency} below states  the consistency of both the one-sided
covariant difference quotient \eqref{eq:CovPlus} and the central  covariant difference quotient \eqref{eq:CovCentral}. The associated estimates will then be used 
to prove the consistency of different approximations of the continuous Riemann curvature tensor.
In fact, a proof for the first order consistency of the  one-sided covariant difference quotient
can already be found in~\cite[Theorem~5.13]{RuWi12b} as a corollary of the convergence of the discrete parallel transport.
Here, we use a different approach based on the implicit function theorem, which is substantially simpler and also allows to prove the second order consistency of 
the central covariant difference quotient with little extra effort.

\begin{theorem}[Consistency of the covariant difference quotients]  \label{lemma:consistency}
Let \eqref{eq:consW} and the assumptions from the beginning of Section\,\ref{sec:review} hold true, and let $m\in\N$, $\y \in \mathring\manifold$, 
$w\in C^{m-1}(\manifold;\Vspace)$, and $\energy \in C^m(\manifold\times\manifold;\R)$
(which by \eqref{eq:consW} automatically implies $\metric \in C^{m-2}(\manifold\times\Vspace\times\Vspace;\R)$).
We consider the covariant derivative of $w$ at $\y$ in a fixed given direction $v\in \Vspace$.
If $m\geq4$ we have
\begin{align}
\left\|\Covt{\tau}  w (\y) - \cov w (\y) \right\|_\Vspace \leq C \tau
\end{align}
for  $\tau$ sufficiently small.
Furthermore, if $m\geq5$, then we have
\begin{align}
\left\| \Covt{\pm \tau}  w (\y) - \cov w (\y) \right\|_\Vspace \leq C \tau^2
\end{align}
for $\y \in \mathring\manifold$ and $\tau$ sufficiently small. In both estimates the constant $C>0$ is independent of $\tau$ and within sufficiently small balls $B\subset\mathring\manifold$ also independent of $\y\in B$.
\end{theorem}
\begin{proof}
Without loss of generality we may assume $\y=0$, and
we consider $\tau \mapsto \w(\tau)= w\circ\ypath(\tau)$ to be the vector field along $\ypath(\tau) = 0 + \tau v$.
By \eqref{eq:CovPlus} and the definition of $\Ptransport_{\tau v,0}^{-1} \tau w(\tau v)$ we have $\left( \Covt{\tau}  w \right)(0) = \tau^{-2} ( \yz(\tau)-\tau \w(0))$,
where $\yz(\tau)$ is the fourth corner of the discrete geodesic parallelogram $0,\tau v,\tau v+\tau \w(\tau),\yz(\tau)$ with center $\yc(\tau)$
and where $\yz$ and $\yc$ satisfy \eqref{eq:ELeqInverseTransp}.

Introducing $F: \R \times \Vspace \times \Vspace \to \Vspace'\times\Vspace'$ with
\begin{align}\label{eq:F}
F(\tau, \yz, \yc)(r) = 
   \begin{pmatrix} 
      \energy_{,2}\myArg{\yz}{\yc}(r) + \energy_{,1}\myArg{\yc}{\tau v}(r) \\
       \energy_{,2}\myArg{0}{\yc}(r) + \energy_{,1}[\yc,\tau v+ \tau \w(\tau)](r)
   \end{pmatrix}
\end{align}
for $r\in \Vspace$, the condition \eqref{eq:ELeqInverseTransp} on $\yz(\tau)$ and $\yc(\tau)$ can be rewritten as $F(\tau, \yz(\tau), \yc(\tau))(r) = 0$ for all $r\in\Vspace$.
In what follows, we skip the tangent vector $r$ to which $F$ is applied, remembering that in all derivatives of $\energy$ the first differentiation is in direction $r$. 
Now the implicit function theorem tells us that there exists a neighbourhood $U$ of the origin in $\R\times \Vspace \times \Vspace$ such that the set
$\{(\tau, \yz, \yc)\in U\,|\,F(\tau, \yz, \yc) =0\}$ is the graph of a map
$\tau \mapsto (\yz, \yc)(\tau)$, \ie
\[
F(\tau, \yz(\tau), \yc(\tau)) = 0\,.
\]
Furthermore, we can differentiate with respect to $\tau$ and obtain
\begin{align}\label{eq:dTauFequalsZero}
0= \frac{\mathrm{d}}{\mathrm{d} \tau} F(\tau, \yz(\tau), \yc(\tau)) =  \partial_\tau F(\tau, \yz(\tau), \yc(\tau))  +  \partial_{(\yz,\yc)} F(\tau, \yz(\tau), \yc(\tau)) 
 \begin{pmatrix} \dot \yz(\tau)\\ \dot \yc(\tau) \end{pmatrix}\,,
\end{align}
where the dot expresses differentiation with respect to $\tau$. 
Inserting \eqref{eq:F} into \eqref{eq:dTauFequalsZero} we obtain the block operator equation 
\begin{align} \label{eq:IFTderivative}
\begin{pmatrix} 
\energy_{,21}\myArg{\yz}{\yc} & \energy_{,22}\myArg{\yz}{\yc} + \energy_{,11}\myArg{\yc}{\tau v} \\
0 & \energy_{,22}\myArg{0}{\yc} + \energy_{,11}\myArg{\yc}{\tau v+ \tau \w} 
\end{pmatrix}
\begin{pmatrix} \dot \yz\\ \dot \yc \end{pmatrix}
= - 
\begin{pmatrix} 
\energy_{,12}\myArg{\yc}{\tau v}(v)  \\
\energy_{,12}\myArg{\yc}{\tau v+ \tau \w} (v+ \w + \tau \dot \w)
\end{pmatrix}
\,,
\end{align}
where we abbreviated $\yc = \yc(\tau)$, $\yz = \yz(\tau)$ and $\w = \w(\tau)$. 
Obviously, we have $(\yz(0), \yc(0))=(0,0)$. 
Using $2g = \energy_{,11}\myArg{0}{0} = \energy_{,22}\myArg{0}{0} = - \energy_{,12}\myArg{0}{0}$  for $g=g_0$ we can evaluate this operator equation as
\begin{align} \label{eq:IFTderivativeZero}
\begin{pmatrix} 
-2g & 4g \\
0 & 4g
\end{pmatrix}
\begin{pmatrix} \dot \yz(0)\\ \dot \yc(0) \end{pmatrix}
= 
\begin{pmatrix} 
2g \,v \\
2g \,v + 2g \,\w(0) 
\end{pmatrix}
\,,
\end{align}
which leads to $\dot \yc(0)=\tfrac12 (v+\w(0))$ and 
$\dot \yz(0)= \w(0)$.
Next, we differentiate \eqref{eq:IFTderivative} once more in $\tau$ and obtain
\begin{align}  \label{eq:IFTsecondDerivative}
&\begin{pmatrix} 
\energy_{,21}\myArg{\yz}{\yc} & \energy_{,22}\myArg{\yz}{\yc} + \energy_{,11}\myArg{\yc}{\tau v} \\
0 & \energy_{,22}\myArg{0}{\yc} + \energy_{,11}\myArg{\yc}{\tau v+ \tau \w} 
\end{pmatrix}
\begin{pmatrix} \ddot \yz\\ \ddot \yc \end{pmatrix}
= 
\begin{pmatrix} 
R_1(\tau) \\
R_2(\tau) 
\end{pmatrix}\, ,
\end{align}
where  $R_1$ and $R_2$ are functionals mapping
$\tau$
into $\Vspace'$ and act on $r \in \Vspace$ with
\begin{align*}
R_1(\tau)(r) =& - \energy_{,211}\myArg{z}{c}(r,\dot \yz,\dot \yz) -  \energy_{,212}\myArg{z}{c}(r,\dot \yz,\dot \yc)- \energy_{,221}\myArg{z}{c}(r,\dot \yc,\dot \yz) - \energy_{,222}\myArg{z}{c}(r,\dot \yc,\dot \yc)\\
& - \energy_{,111}\myArg{c}{\tau v}(r,\dot \yc,\dot \yc) - \energy_{,112}\myArg{c}{\tau v}(r,\dot \yc,v) - \energy_{,121}\myArg{c}{\tau v}(r,v,\dot \yc) - \energy_{,122}\myArg{c}{\tau v}(r,v,v)\\
=& - \energy_{,211}\myArg{z}{c}(r,\dot \yz,\dot \yz) -  2\energy_{,212}\myArg{z}{c}(r,\dot \yz,\dot \yc) - \energy_{,222}\myArg{z}{c}(r,\dot \yc,\dot \yc)\\
& - \energy_{,111}\myArg{c}{\tau v}(r,\dot \yc,\dot \yc) - 2\energy_{,112}\myArg{c}{\tau v}(r,\dot \yc,v) - \energy_{,122}\myArg{c}{\tau v}(r,v,v)\,, \\
R_2(\tau)(r) =& - \energy_{,222}\myArg{0}{c}(r,\dot \yc,\dot \yc)  - \energy_{,111}\myArg{c}{\tau (v+\w)}(r,\dot \yc,\dot \yc) - \energy_{,112}\myArg{c}{\tau (v+\w)}(r,\dot \yc,v+\w+\tau \dot \w)\\
& - \energy_{,121}\myArg{c}{\tau (v+\w)}(r,v+\w+\tau \dot \w,\dot \yc) - \energy_{,122}\myArg{c}{\tau (v+\w)}(r,v+\w+\tau \dot \w,v+\w+\tau \dot \w)\\
& - \energy_{,12}\myArg{c}{\tau (v+\w)} (r, 2 \dot \w + \tau \ddot \w) \\
=& - \energy_{,222}\myArg{0}{c}(r,\dot \yc,\dot \yc)  - \energy_{,111}\myArg{c}{\tau (v+\w)}(r,\dot \yc,\dot \yc) - 2\energy_{,112}\myArg{c}{\tau (v+\w)}(r,\dot \yc,v+\w+\tau \dot \w)\\
&  - \energy_{,122}\myArg{c}{\tau (v+\w)}(r,v+\w+\tau \dot \w,v+\w+\tau \dot \w)- \energy_{,12}\myArg{c}{\tau (v+\w)} (r, 2 \dot \w + \tau \ddot \w) \, .
\end{align*}
Evaluating \eqref{eq:IFTsecondDerivative} at $\tau=0$ gives
\begin{align}  \label{eq:2nd}
&\begin{pmatrix} 
-2g &4g \\
0 & 4g
\end{pmatrix}
\begin{pmatrix} \ddot \yz(0)\\ \ddot \yc(0) \end{pmatrix}
=
\begin{pmatrix} 
R_1(0) \\
R_2(0) 
\end{pmatrix}
\end{align}
with
\begin{align*}
R_1(0)(r) =& -\energy_{,211}\myArg{0}{0}(r,\w(0),\w(0)) - \energy_{,212}\myArg{0}{0}(r,\w(0),v+\w(0)) \\
& -\tfrac14 \energy_{,222}\myArg{0}{0}(r,v+\w(0),v+\w(0)) -\tfrac14  \energy_{,111}\myArg{0}{0}(r,v+\w(0),v+\w(0))\\
& -\energy_{,112}\myArg{0}{0}(r,v+\w(0),v) -\energy_{,122}\myArg{0}{0}(r,v,v)\,, \\
R_2(0)(r) =& -\tfrac14 \energy_{,222}\myArg{0}{0}(r,v+\w(0),v+\w(0))  -\tfrac14 \energy_{,111}\myArg{0}{0}(r,v+\w(0),v+\w(0)) \\
& -\energy_{,112}\myArg{0}{0}(r,v+\w(0),v+\w(0)) -\energy_{,122}\myArg{0}{0}(r,v+\w(0),v+\w(0))  - \energy_{,12}\myArg{0}{0} (r, 2\dot \w(0)) \,,
\end{align*}
where we used  $\dot \yc(0)=\tfrac12 (v+\w(0))$ and $\dot \yz(0) = \w(0)$ from above. 
Solving \eqref{eq:2nd} yields  $g(\ddot c(0),r) =  \tfrac{R_2(0)(r)}4$ and
\begin{align}
g(\ddot z(0),r)  = &\frac12 R_2(0)(r) - \frac12 R_1(0)(r) \notag\\ 
=& \tfrac12\big(
[-\energy_{,121}-\energy_{,122}]\myArg{0}{0}(r,v+\w(0),v+\w(0))
+\energy_{,212}\myArg{0}{0}(r,\w(0),v+\w(0))\notag\\
&+\energy_{,112}\myArg{0}{0}(r,v+\w(0),v)
+\energy_{,122}\myArg{0}{0}(r,v,v)
+\energy_{,211}\myArg{0}{0}(r,\w(0),\w(0))
\big)\notag\\
&- \energy_{,12}\myArg{0}{0} (r, \dot \w(0)) \notag\\
=&\tfrac12\big(
[-\energy_{,121}-\energy_{,122}+\energy_{,212}+\energy_{,211}]\myArg{0}{0}(r,\w(0),\w(0))\notag\\
&+[-\energy_{,121}-\energy_{,122}+\energy_{,112}+\energy_{,122}]\myArg{0}{0}(r,v,v)\notag\\
&+[-\energy_{,121}-\energy_{,112}-2\energy_{,122}+\energy_{,212}+\energy_{,112}]\myArg{0}{0}(r,\w(0),v)
\big)- \energy_{,12}\myArg{0}{0} (r, \dot \w(0)) \notag\\
=&-\tfrac12[\energy_{,211}+\energy_{,122}]\myArg{0}{0}(r,v,\w(0)) - \energy_{,12}\myArg{0}{0} (r, \dot \w(0))\,.\label{eq:metricZConsistency}
\end{align}
Note that \eqref{eq:thirdDerivSymm} implies $(-\energy_{,121}-\energy_{,122}+\energy_{,212}+\energy_{,211})\myArg{0}{0} = 0$
and $(-\energy_{,121}-\energy_{,122}+\energy_{,212})\myArg{0}{0} = -\energy_{,211}\myArg{0}{0}$ which was used in the last step.

From $g(v,w) = \tfrac12 \energy_{,11}\myArg{0}{0}(v,w)$ we deduce
\[D_\y g(r)(v,w) = \tfrac12 \left(\energy_{,111}\myArg{0}{0}(v,w,r) + \energy_{,112}\myArg{0}{0}(v,w,r) \right)\,.\]
Hence, using the definition of the covariant derivative in \eqref{eq:defCovDeriv}, we obtain at $\ypath(0) =\y= 0$
\begin{align*}
&g\left(\cov w(\y),r\right) =
g(\dot \w(0),r) + g(\Gamma_0(\w(0),v),r)\\
=& - \frac12 \energy_{,12}\myArg{0}{0} (r, \dot \w(0)) + \frac12 \left(D_\y g(\w(0))(r,v) +D_\y g(v)(r,\w(0))-D_\y g(r)(v,\w(0))\right)\\
=& \frac14 \big( \energy_{,111}\myArg{0}{0}(r,v,\w(0)) + \energy_{,112}\myArg{0}{0}(r,v,\w(0)) + \energy_{,111}\myArg{0}{0}(r,\w(0),v) \\
&+ \energy_{,112}\myArg{0}{0}(r,\w(0),v) - \energy_{,111}\myArg{0}{0}(v,\w(0),r) - \energy_{,112}\myArg{0}{0}(v,\w(0),r) \big) - \frac12 \energy_{,12}\myArg{0}{0} (r, \dot \w(0))\\
=& \frac14 \big( \energy_{,111} + \energy_{,112}  + \energy_{,121}  - \energy_{,211}\big)\myArg{0}{0}(r,v,\w(0)) - \frac12 \energy_{,12}\myArg{0}{0} (r, \dot \w(0)) \\
=& -\frac14 \big( \energy_{,122}  + \energy_{,211}\big)\myArg{0}{0}(r,v,\w(0)) - \frac12 \energy_{,12}\myArg{0}{0} (r, \dot \w(0))\,,
\end{align*}
where \eqref{eq:thirdDerivSymm} implied $(\energy_{,111} +   \energy_{,112})\myArg{0}{0}  = -(\energy_{,121} + \energy_{,122})\myArg{0}{0}$ in the last step.
Then, taking into account~\eqref{eq:metricZConsistency} we obtain
\begin{align*}
 g\left( \frac12\ddot z(0),r\right) = g\left(\cov w(0),r\right)\, .
\end{align*}
Next, we consider the Taylor expansion of $\tau \mapsto \yz(\tau)$ at $\tau =0$, \ie
\begin{align*}
\yz(\tau) =& \yz(0) + \dot z(0) \tau + \frac12 \ddot \yz(0) \tau^2 + \frac16 \dddot \yz(0) \tau^3 + R(\tau)\\
=& 0 + w(0) \tau + \cov w(0) \,\tau^2 + \frac16 \dddot \yz(0) \tau^3 + R(\tau)
\end{align*}
with the remainder term $R(\tau)=o(\tau^3)$ if $\energy \in C^4(\manifold\times\manifold;\R)$ and $R(\tau)=O(\tau^4)$ if $\energy \in C^5(\manifold\times\manifold;\R)$.
This implies
\begin{align*}
 \cov w(0) =& \frac1{\tau^2} \left( \yz(\tau)-\tau w(0)\right) + O(\tau)
=\Covt{\tau}  w(0) + O(\tau)\,,
\end{align*}
which establishes first order consistency of the one-sided covariant difference quotient.
Similarly, for the central covariant difference quotient we obtain 
\begin{align*}
\Covt{\pm \tau}  w(0)
&=\frac12\left( \frac1{\tau^2} \left( \yz(\tau)-\tau w(0)\right) + \frac1{\tau^2} \left( \yz(-\tau)+\tau w(0)\right)\right)
= \frac{\yz(\tau) + \yz(-\tau)}{2\tau^2} \\
&= \frac1{2\tau^2}\Big(w(0) \tau + \cov w(0)\, \tau^2 +\frac{1}{6}\dddot \yz(0) \tau^3 +  O(\tau^4)
- w(0) \tau +\cov w(0)\, \tau^2 -\frac{1}{6}\dddot \yz(0) \tau^3 +  O(\tau^4)\Big) \\
&= \cov w (0) + O(\tau^2)\,.
\end{align*}
The differentiability requirements in the statement of the theorem directly follow by inspection of \eqref{eq:IFTsecondDerivative}.
\end{proof}

\section{Discrete Approximation of the Riemann Curvature Tensor}\label{sec:curvatureTensor}
For $\y\in \manifold$ and a pair of tangent vectors $v,\,w \in \Vspace$ the Riemannian curvature tensor is a linear mapping $\riemann_\y(v,w) :\Vspace \to \Vspace$
with
\begin{equation}\label{eqn:RiemannCurvature}
\left(\riemann_\y(v,w)z\right)= \left(\covt \covs  z - \covs \covt z \right)(\y)\,,
\end{equation}
where $z$ is an extension of the tangent vector $z$ by a smooth vector field (in our context, for instance, by a vector field constant on $\manifold$) and $\ypath(t,s)$ a smooth mapping $\ypath: [-\epsilon, \epsilon]^2 \to \manifold$ with $\ypath(0,0)=\y$, $\partial_t \ypath(0,0) = v$ and 
$\partial_s \ypath(0,0) = w$ for some $\epsilon >0$. Examples of such mappings are $\ypath(t,s) =\exp_\y(sw+tv)$ on general manifolds.
In the case $\manifold \subset \Vspace$ we simply choose $\ypath(t,s)=\y+sw+tv$.
To define a discrete counterpart we consider nested covariant difference quotients with a smaller  step size in the inner difference quotient.
In detail, using the first order one-sided covariant difference quotient we define
\begin{equation}\label{eq:RTensor}
\left(\Riemann^\tau_\y(v,w)z\right)= \left(\Covt{\tau} \Covs{\tau^\beta}  z - \Covs{\tau} \Covt{\tau^\beta} z \right)(\y)
\end{equation}
with $\beta\geq2$.
Using the central  covariant difference quotient
we obtain the corresponding discrete Riemannian curvature tensor 
\begin{align}
\left(\Riemann^{\pm\tau}_\y(v,w)z\right)= \left(\Covt{\pm\tau} \Covs{\pm\tau^\beta}  z - \Covs{\pm\tau} \Covt{\pm\tau^\beta} z \right)(\y)
\end{align}
with $\beta \geq \frac32$.
The consistency results for the different discrete Riemannian curvature tensors are given in the following theorem.

\begin{theorem}[Consistency of the discrete Riemannian curvature tensor] \label{thm:consRiemCurvTensor}
Let \eqref{eq:consW} and the assumptions from the beginning of Section\,\ref{sec:review} hold, and let $m\in\N$ as well as $\energy \in C^{m}(\manifold\times\manifold;\R)$ and $\metric \in C^{m}(\manifold\times\Vspace\times\Vspace;\R)$.
For $\y \in \mathring\manifold$, $v,w,z \in \Vspace$ and $\tau$ sufficiently small we have
\begin{align}
\left(\Riemann^{\tau}_\y(v,w)z\right) = \left(\riemann_\y(v,w)z\right) + O(\tau)
\end{align}
with respect to the $\Vspace$ norm for $m\geq4$ and $\beta\geq2$, and we have
\begin{align}
\left(\Riemann^{\pm\tau}_\y(v,w)z\right) = \left(\riemann_\y(v,w)z\right) + O(\tau^2)
\end{align}
with respect to the $\Vspace$ norm for $m\geq5$ and $\beta \geq \frac32$.
\end{theorem}
\begin{proof}
We first prove the estimate for the discrete Riemannian curvature tensor based on 
one-sided covariant difference quotients.
It follows from the stability of the midpoint and the endpoint in $2$-geodesics \cite[Lemmas 5.5 \& 5.8]{RuWi12b} that
\begin{align}
\| \Ptransport_{\y+\tau v,\y}^{-1} w \|_\Vspace  \leq C  \| w \|_\Vspace
\end{align}
for some constant $C>0$ and all $\tau$ and $\| w \|_\Vspace$ sufficiently small. In fact, from the local differentiability of discrete geodesics with respect to their end points \cite[Theorem 4.7]{RuWi12b}
and the local differentiability of the discrete exponential map with respect to base point and direction \cite[Lemma 5.6 \& Remark 5.7]{RuWi12b}
the discrete parallel transport is differentiable.
Thus, the difference between the first term of \eqref{eq:RTensor} and of \eqref{eqn:RiemannCurvature} can be estimated as
\begin{align*}
&\left\|\left(\Covt{\tau} \Covs{\tau^\beta}  z - \covt \covs z \right)(\y)  \right\|_{\Vspace}
\leq  \left\|\Covt{\tau} \left( \Covs{\tau^\beta}   z -  \covs  \right)(\y)\right\|_{\Vspace}  + \left\|\left(\Covt{\tau} \covs  z - \covt \covs  z \right)(\y)\right\|_{\Vspace} \\
& \leq  \frac1{\tau^2} \left\| \left( \Ptransport_{\y+\tau v,\y}^{-1} \tau  \left( \Covs{\tau^\beta}  z -  \covs  z \right) (\y+\tau v) \right)
- \tau  \left( \Covs{\tau^\beta}  z -  \covs z \right) (\y) \right\|_{\Vspace}+ O(\tau) \\
&\leq   \frac1{\tau^2} (1+C) \tau \max_{\sigma \in \{0,1\}}\left\| \left( \Covs{\tau^\beta}  z -  \covs z \right) (\y+\tau \sigma v) \right\|_{\Vspace} + O(\tau) 
=  \frac{1+C}{\tau}  O(\tau^\beta) + O(\tau) = O(\max\{\tau,\tau^{\beta-1}\})\,.
\end{align*}
The second term $\left(\Covs{\tau} \Covt{\tau^\beta}  z - \covs \covt z \right)(\y)$ can be treated in analogy and thus the claim follows for the choice $\beta\geq2$.
Let us remark that the above application of Theorem~\ref{lemma:consistency} requires $\energy \in C^4(\manifold\times\manifold,\R)$ and $\covs  z \in C^3(\manifold;\Vspace)$,
which is fulfilled if $g\in C^4(\manifold\times\Vspace\times \Vspace)$.
In the case of the discrete Riemannian curvature tensor based on central covariant difference quotients we proceed along the same lines
and obtain 
\begin{align*}
\left\|\left(\Covt{\pm \tau} \Covs{\pm \tau^\beta}  z - \covt \covs z \right)(\y)  \right\|_{\Vspace}
\leq&  \left\|\Covt{\pm \tau} \left( \Covs{\pm \tau^\beta}   z -  \covs  \right)(\y)\right\|_{\Vspace}  + \left\|\left(\Covt{\pm \tau} \covs  z - \covt \covs  z \right)(\y)\right\|_{\Vspace} \\
\leq&   \frac{2C \tau}{\tau^2}  \max_{\sigma \in \{-1,1\}} \left\| \left(\Covs{\pm \tau^\beta}  z -  \covs z \right) (\y+\tau \sigma v)\right\|_{\Vspace} + O(\tau^2)\\
= &  \frac{2C}{2\tau}  O(\tau^{2\beta}) + O(\tau^2) = O(\max\{\tau^{2\beta-1},\tau^2\})
\end{align*}
which for $\beta\geq\frac32$ implies the second order consistency of the discrete Riemannian curvature tensor.
This time, the application of Theorem~\ref{lemma:consistency} requires $\energy \in C^5(\manifold\times\manifold,\R)$ and $\covs  z \in C^4(\manifold;\Vspace)$,
which is fulfilled if $g\in C^5(\manifold\times\Vspace\times \Vspace)$.
\end{proof}
Finally, for $\y\in\mathring\manifold$ and a pair of tangent vectors $v,\,w \in \Vspace$ the sectional curvature 
\begin{align}
\kappa_\y(v,w) := \frac{ \metric_\y \left( v, \riemann_\y(v,w)w\right) } { \metric_\y(v,v)\metric_\y(w,w) - \metric_\y^2(v,w)}
\end{align}
is the Gaussian curvature of the surface $\{ \exp_\y(sv + tw) : s,t\in (-\epsilon, \epsilon) \}$ at $\y$. 
Having the notion of a consistent discrete Riemann curvature tensor, we define a corresponding discrete sectional curvature as
\begin{align}
\mathbf{\kappa}_\y^{\tau}(v,w) := \frac{ \metric_\y \left( v, \Riemann^\tau_\y(v,w)w\right) } { \metric_\y(v,v)\metric_\y(w,w) - \metric_\y^2(v,w)}
\end{align}
and likewise $\mathbf{\kappa}^{\pm\tau}(v,w)$ using $\Riemann^{\pm\tau}$. 
Under the assumptions of Theorem~\ref{thm:consRiemCurvTensor} we then have
\begin{align*}
 \left| \kappa_\y - \mathbf{\kappa}_\y^{\tau} \right| (v,w) &= \frac{ \left|\metric_\y \left( v, \Big[\riemann_\y(v,w) - \Riemann^\tau_\y(v,w)\Big] w\right) \right|} { \left|\metric_\y(v,v)\metric_\y(w,w) - \metric_\y^2(v,w) \right|} \leq C \left\| v\right\|_\Vspace \, \left\| \Big[\riemann_\y(v,w) - \Riemann^\tau_\y(v,w)\Big] w \right\|_\Vspace = O(\tau)\, , 
\end{align*}
where a second order consistency of $\mathbf{\kappa}_\y^{\pm\tau}$ follows analogously. 
Altogether, we have proven the following:
\begin{corollary}[Consistency of discrete sectional curvature] \label{thm:consSecCurv}
Let $\y\in\mathring\manifold$ and $v,\,w \in \Vspace$. Under the assumptions of Theorem~\ref{thm:consRiemCurvTensor} we have 
$\left| \kappa_\y - \mathbf{\kappa}_\y^{\tau} \right| (v,w) = O(\tau)$ 
and
$|\kappa_\y - \mathbf{\kappa}_\y^{\pm\tau}|(v,w) = O(\tau^2)\, .$
\end{corollary}

Figures~\ref{fig:torus} and \ref{fig:torusConvergence} illustrate the numerical computation and convergence of the sectional curvature on a two-dimensional embedded surface.

\begin{figure}
\setlength\unitlength{.21\linewidth}
\strut\hfill
\includegraphics[width=\unitlength]{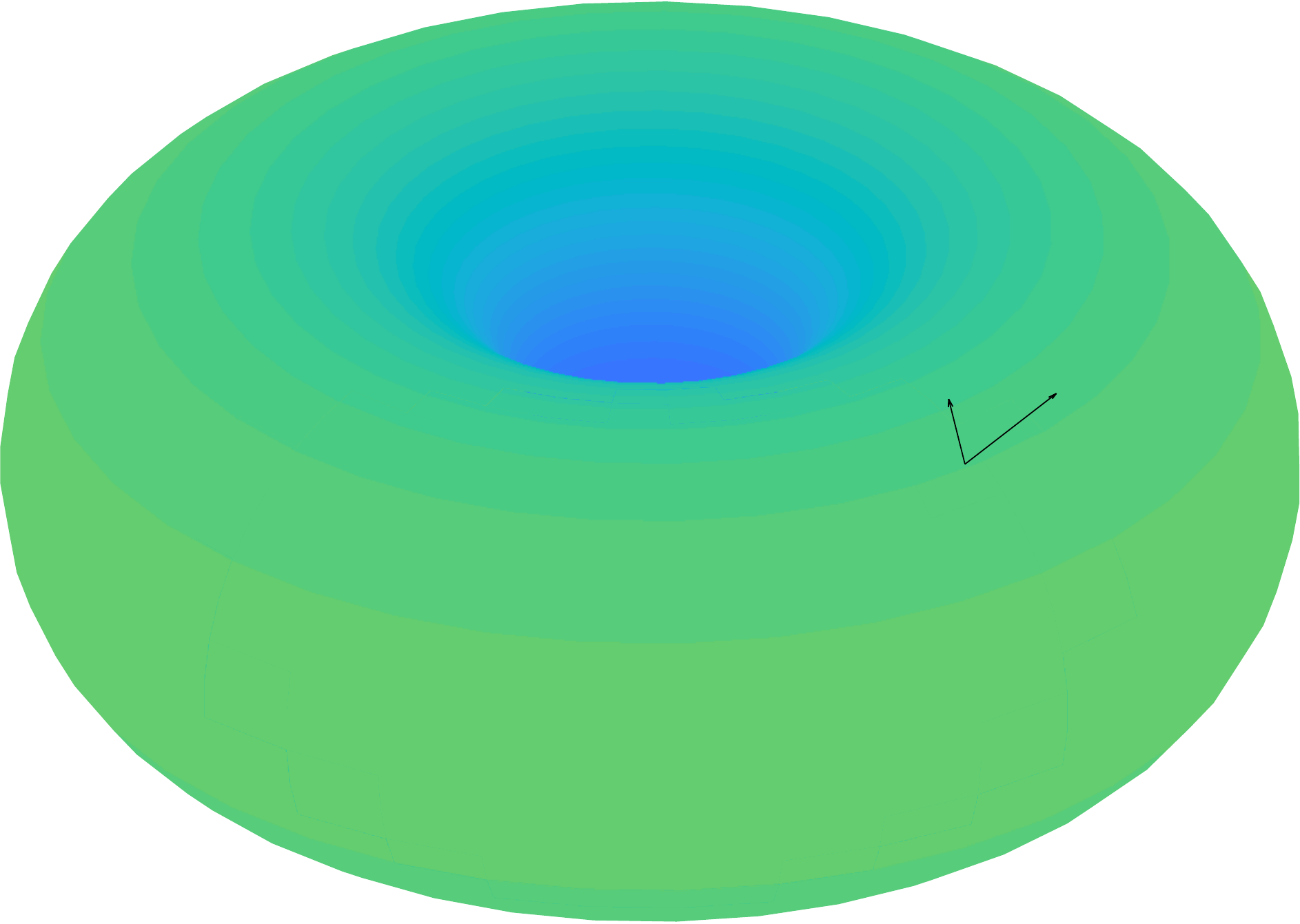}%
\includegraphics[width=\unitlength]{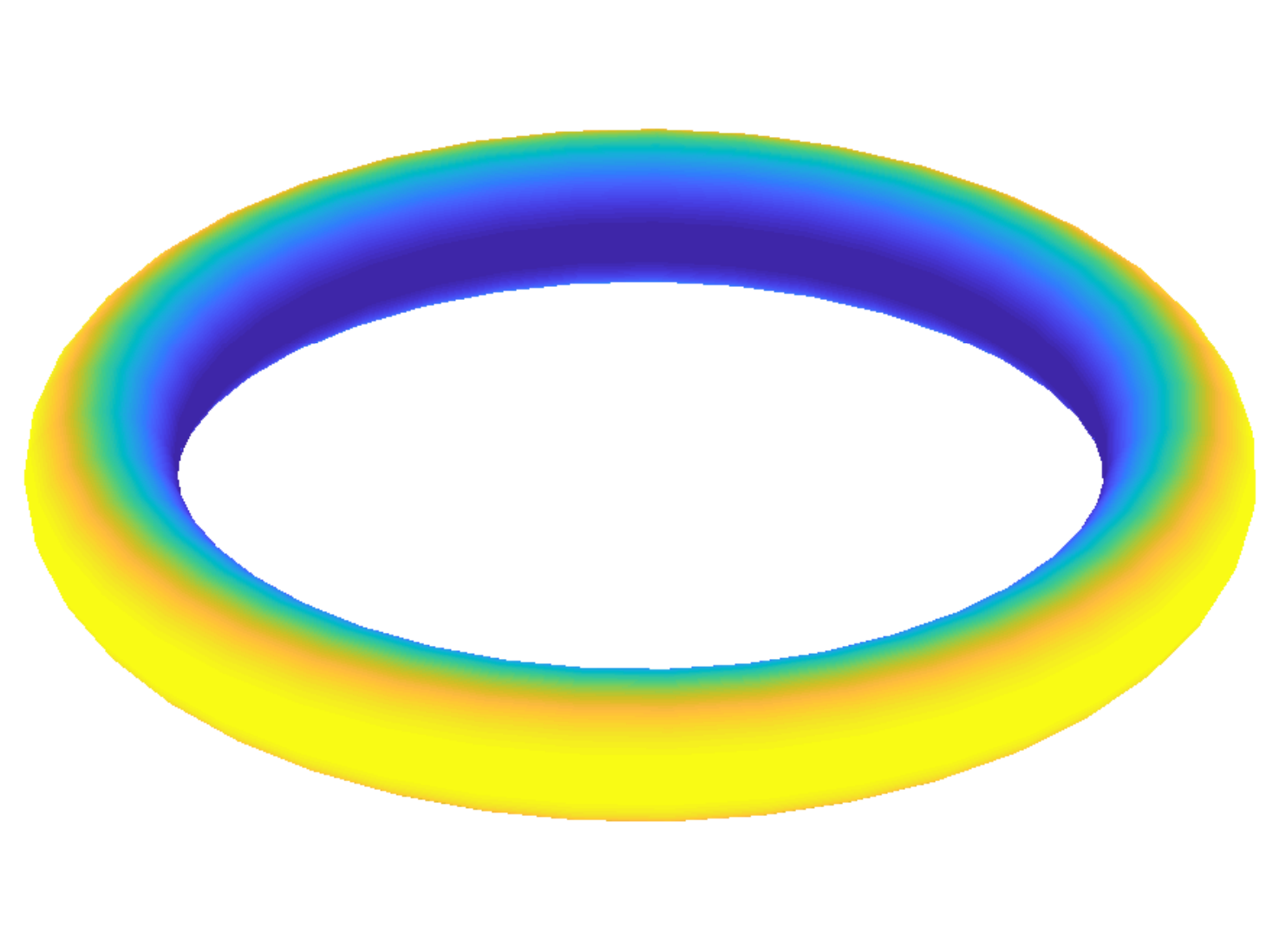}\hfill\hfill%
\includegraphics[width=\unitlength]{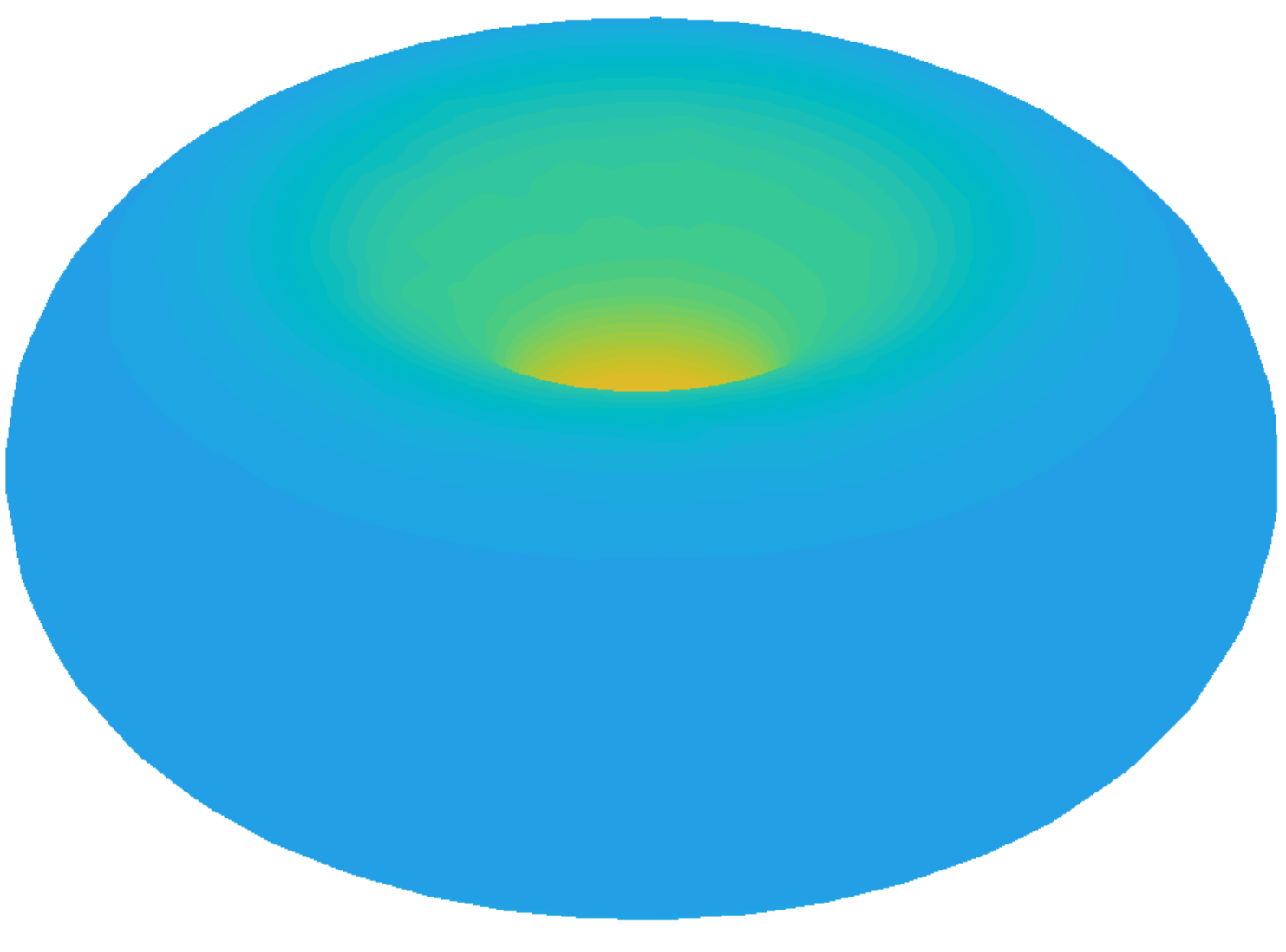}%
\includegraphics[width=\unitlength]{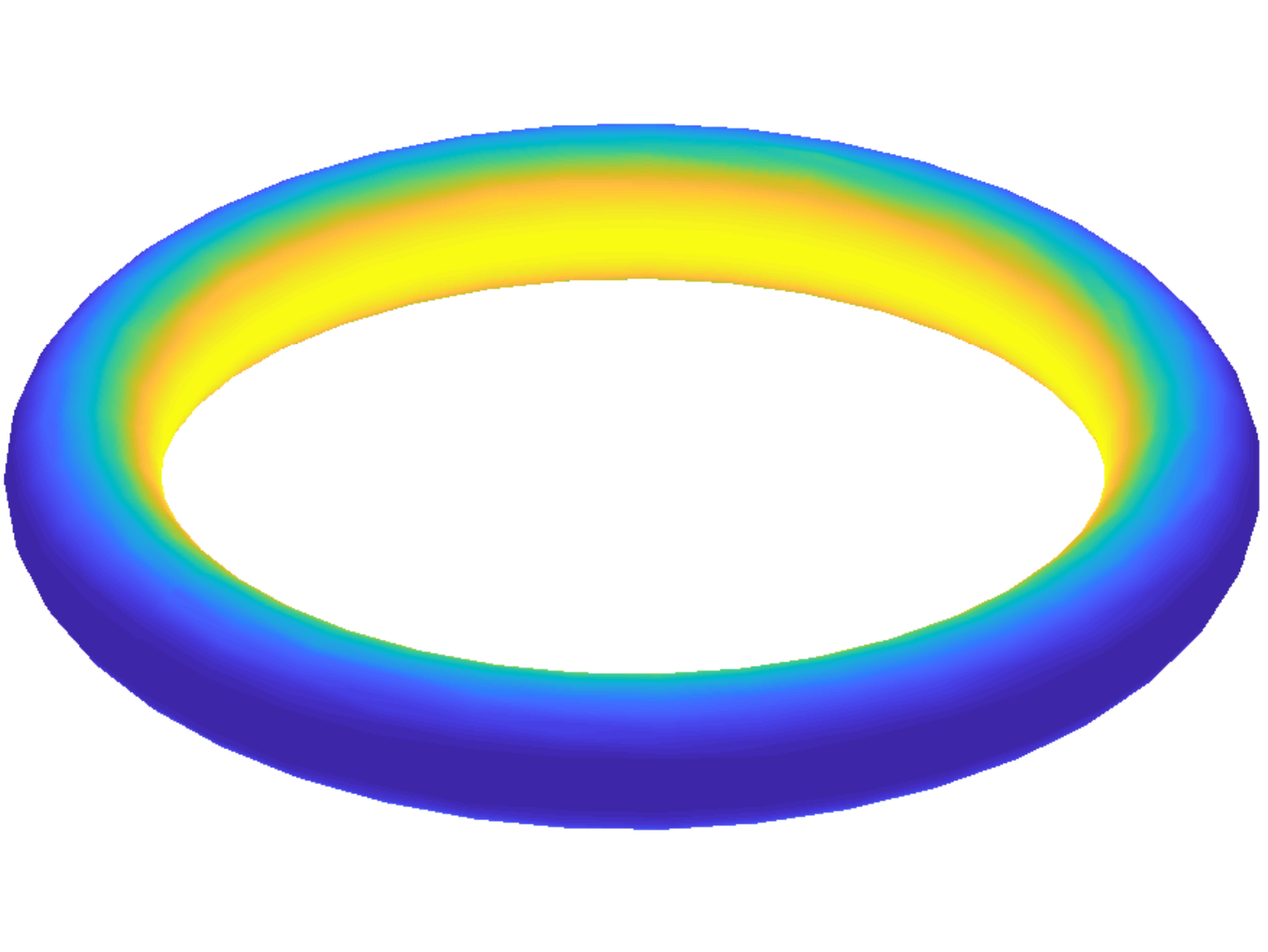}\hfill\ \\
\strut\hfill-4 \includegraphics[width=\unitlength]{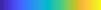} 3\hfill\hfill$-5\cdot10^{-5}$ \includegraphics[width=\unitlength]{images/colorbar} $7\cdot10^{-5}$\hfill\ 
\caption{Left: Gaussian curvature on two tori (each with center line radius of $\sqrt2$) numerically approximated via $\kappa^{\pm\tau}(v,w)$ for $\tau=10^{-2}$ and $\beta=\frac32$.
The tori were parameterized by toroidal and poloidal angles, $v=(1,0)$, $w=(0,1)$, and $\energy$ is taken as the squared Euclidean distance in the embedding space $\R^3$.
The black arrows on the left indicate the position and tangent vectors for which Figure~\ref{fig:torusConvergence} shows a convergence study of the approximated sectional curvature.
Right: Corresponding approximation error $(\kappa^{\pm\tau}-\kappa)(v,w)$.
}
\label{fig:torus}
\end{figure}

\begin{figure}[H]
\centering
\begin{tikzpicture}
\begin{loglogaxis}[xmin=0.005, xmax=1e-0, ymin=1e-6, ymax=3.0, legend pos=north west, ylabel = relative error, xlabel = $\tau$
]
\addplot[color=black, thick, dotted] plot coordinates {  (1.00000000000000000 , 2.60133582037034916)  (0.55519359143862090 , 0.70514245722938651)  (   0.30823992397451433 , 0.31337773154819404)  (   0.17113283041617811 , 0.16090547679432438)  (   0.09501185073181440 , 0.08630403877328301)  (   0.05274997063702620 , 0.04721138033466334 )  (   0.02928644564625237 , 0.02600080861945030 )  (   0.01625964693881482 , 0.01438013437291902)  (   0.00902725177948458 , 0.00797661969403041  )  (   0.00501187233627272,   0.00453021968821463 ) };
\addlegendentry{$\kappa^\tau$}
\addplot[color=black, thick, dashed] plot coordinates {  (1.00000000000000000 , 5.89271033935994e-01)  (0.55519359143862090 , 2.16942609731176e-01)  (   0.30823992397451433 , 6.72739164352992e-02 )  (   0.17113283041617811 , 2.13457594690691e-02)  (   0.09501185073181440 ,  6.61950532543652e-03 )  (   0.05274997063702620 , 2.03682906063085e-03)  (   0.02928644564625237 , 6.26160959342312e-04 )  (   0.01625964693881482 , 1.91844421654184e-04)  (   0.00902725177948458 , 5.94380564209744e-05   )  (   0.00501187233627272,  1.63140029911033e-05 ) };
\addlegendentry{$\kappa^{\pm\tau}$}
\addplot[thin,color=black] plot coordinates { (0.08, 0.06) ( 0.008,  0.006 ) ( 0.08, 0.006 ) (0.08, 0.06) };
\addplot[thin,color=black] plot coordinates { (0.2, 0.02) ( 0.02,  0.0002 ) ( 0.2, 0.0002 ) (0.2, 0.02) };
\end{loglogaxis}
\end{tikzpicture}
\caption{Relative error for approximating the Gaussian curvature at the marked position of the left torus from the previous figure, using one-sided and central differences.
The triangles indicate the slopes of quadratic and linear convergence.}
\label{fig:torusConvergence}
\end{figure}
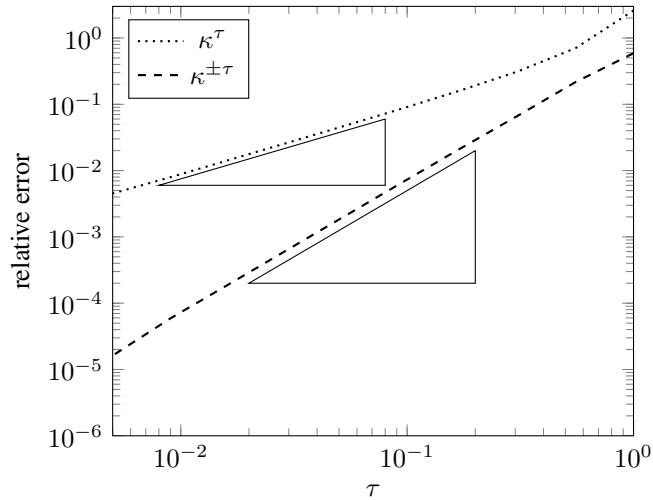

\section{Discrete shell space}\label{sec:ShellSpace}
\definecolor{myOrange}{RGB}{255, 169, 87 }
\definecolor{myGreen}{RGB}{180, 255, 162  }
\definecolor{myGrey}{RGB}{187, 187, 187  }

\newcommand{\dW}{{\energy}}
\renewcommand{\shell}{{y}}

In this section, we consider the manifold $\manifold$ of \emph{discrete shells}, \ie triangle meshes with fixed mesh connectivity whose deformations are modelled based on a physical shell model. 
The concept of discrete shells has been introduced by Grinspun et al. \cite{GrHiDeSc03} as a tool for efficient and reliable physical simulation of shell deformations.
Fr\"ohlich and Botsch \cite{FrBo11} studied the effective discretization of membrane and bending distortion. 
In   \cite{HeRuWa12,HeRuSc14} this approach was investigated in the context of shape space theory with a metric representing a physical dissipation due to membrane and bending distortion.  
In fact, a suitable Riemannian metric on  $\manifold$ is induced by the Hessian of an appropriate elastic energy functional. 
In the following, we briefly summarize the underlying physical model and discuss its discretization, for details we refer the reader to \cite{HeRuWa12, HeRuSc14}.

Before we introduce a discrete surface model based on triangle meshes, let us briefly investigate the case of smooth surfaces $\shell \subset \R^3$ representing the middle layer of a thin sheet of some material.
For a homogeneous, isotropic and elastic material, the deformation energy $\energy_\shell[\phi]$ of an elastic deformation $\phi: \shell \to \R^3$ can be captured by membrane and bending contributions, 
reflected by the change of first (\ie metric) and second (\ie curvature) fundamental forms, respectively. 
In the following, we distinguish between the undeformed reference surface $\shell$ and the deformed surface $\shell^\phi = \phi(\shell)$.
For parametric surfaces $\shell$ we denote by $g$ and $h$ the matrix representation of the first and second form, respectively, in a parameter domain $\omega \subset \R^2$. 
To this end, the membrane distortion induced by $\phi$ can be described pointwise by a (reduced) Cauchy-Green strain tensor $A[\phi] = g^{-1}g^\phi \in \R^{2,2}$,
and the bending distortion is captured pointwise by the relative shape operator $B[\phi] = g^{-1}(h - h^\phi) \in \R^{2,2}$.
The total elastic energy is then given by
\begin{align}\label{eq:smoothShellEnergy}
 \energy_\shell[\phi]= \int_\omega W_{\mem}(A[\phi])\, \sqrt{\det g} \d \xi +\delta^2 \int_\omega W_{\bend}(B[\phi]) \,\sqrt{\det g} \d \xi ,
\end{align}
where $\delta >0$ represents the physical thickness of the thin elastic material. 
For $W_{\mem}$, we make use of a classical hyperelastic energy density as in equation~(8) in \cite{HeRuWa12}. 
Additionally, we set $W_{\bend}(B[\phi]) = (\tr B[\phi])^2$, which measures the squared error of the mean curvature. 
It has been shown in \cite{HeRuSc14} that the Hessian of $\energy_\shell$ induces a Riemannian metric on the space of surfaces modulo rigid body motions. 

Now we consider an approximation of the smooth surface by a triangle mesh -- which we shall also denote by $\shell$. 
Let $\mathcal{V} = \{1, \ldots, n\}$, $\mathcal{T} \subset \mathcal{V}\!\times\!\mathcal{V}\!\times\!\mathcal{V}$ and $\mathcal{E} \subset \mathcal{V}\!\times\!\mathcal{V}$ be the set of vertices, triangles and edges of $\shell$, respectively. 
A (discrete) deformation of $\shell$ can be identified with a mapping $\phi:\mathcal{V}\to\R^3$ which is interpolated piecewise linearly over triangles. 
Obviously, $D\phi$ is constant on triangles $t \in \mathcal{T}$ which leads to an elementwise constant, discrete first fundamental form $g = g_t$.
In particular, we can make use of the same energy density $W_{\mem}$ as above now measuring changes in the discrete first fundamental form.
For the discretization of the bending term, which requires second derivatives of surface positions in the smooth case, 
we consider the widely used \emph{discrete shell} energy \cite{GrHiDeSc03}, where bending is quantified by changes of dihedral angles $\theta_e$ living on edges $e \in \mathcal{E}$. 
The discrete version of the elastic deformation energy \eqref{eq:smoothShellEnergy} reads
\begin{align}\label{eq:discreteShellsEnergy}
 \dW_\shell[\phi] = \sum_{t \in \mathcal{T}} a_t \, W_{\mem}(g^{-1}_t g^\phi_t ) + \delta^2 \sum_{e \in \mathcal{E}} \frac{(\theta_e - \theta^\phi_e)^2}{d_e}l_e^2\, ,
\end{align}
where $a_t$ is the volume of triangle $t$, $l_e$ is the length of edge $e$ and $d_e = \tfrac13(a_t + a_{t'})$ if $e = t \cap t'$. 
In fact, it can be shown that the simplified bending term in \eqref{eq:discreteShellsEnergy} can be derived formally from $W_{\bend}(B[\phi])$ by an appropriate linearization \cite{He16}. 
For a more rigorous discussion of the discretization of the bending term we refer to \cite{CoMo03} and \cite{GlOl19}. 
Again, it has been shown in \cite{HeRuSc14} that the Hessian of \eqref{eq:discreteShellsEnergy} induces a Riemannian metric on the space of triangle meshes modulo rigid body motions.

Now we define the shape space of discrete shells $\manifold$ as the set of all triangle meshes with the same connectivity, 
\ie for two discrete shells $\shell,\tilde\shell\in\manifold$ there is a 1-to-1 correspondence between all vertices, faces and edges. 
In particular, a discrete shell $\shell$ is uniquely described by its vector of vertex positions living in $\R^{3n}$. 
To this end, the functional $(\y,\tilde \y) \mapsto \energy\myArg{\y}{\tilde \y}$ required in the discrete Riemannian calculus is simply defined as $\energy\myArg{\shell}{\tilde \shell} := \dW_\shell[\phi]$ 
with $\dW_\shell$ given by \eqref{eq:discreteShellsEnergy} and $\phi$ is the \emph{unique} discrete deformation such that $\tilde\shell = \phi(\shell)$. 

In the following, we will first consider a toy example (a deformation of a triangulated sphere) 
to investigate the convergence behaviour of the discrete sectional curvature in the space of discrete shells when increasing the temporal and/or spatial resolution 
as well as the dependence on the (squared) thickness $\mu := \delta^{2}$ of the elastic material, which we shall also denote as \emph{bending weight}. 
Finally, we consider a family of almost isometric deformations of a more complex cactus shape and compute sectional curvature patterns.

\paragraph{Sectional curvature near spherical discrete shell.} Let $\shell\in\R^{3n}$  be the vector of vertex positions of a triangluation of the unit sphere $S^2$ with vertices lying on $S^2$.
Furthermore, we consider tangent vectors $v$ and $w$ corresponding to a stretching by a factor $\tfrac32$ in the first and second component of the vertex positions, respectively. 
Figure~\ref{fig:secCurv_DiscreteShells_sphere_convergence} displays the decay of the relative error 
$\left| \frac{\kappa - \kappa^\tau}{\kappa}\right|(v,w)$ and $\left| \frac{\kappa - \kappa^{\pm \tau}}{\kappa}\right|(v,w)$ as a function of the time step size $\tau$
for different spatial resolutions and a bending weight $\mu = 10^{-3}$.
Here, the value of the discrete sectional curvature for the smallest time step $\tau = \tau_{\min}$ was used as ground truth. 
As predicted by our analytical results, we observe a quadratic convergence when using the central difference quotient 
and a convergence which is experimentally still better than linear but not yet quadratic when using the one-sided difference quotient.
Figure~\ref{fig:secCurv_DiscreteShells_sphere_convergenceInH_differentMus} 
shows the experimental convergence of $\kappa^{\pm\tau}$ for fixed $\tau = 10^{-2}$ and $\mu = 10^{-3}$ for decreasing 
grid size $h$ of the underlying triangular mesh as well as the dependence on the thickness parameter $\mu$ in the discrete shell model for fixed $\tau$ and $h$. 
In fact, one observes second order convergence with respect to a uniform refinement of triangles as well as a decay of the sectional curvature for decreasing values of $\mu$.

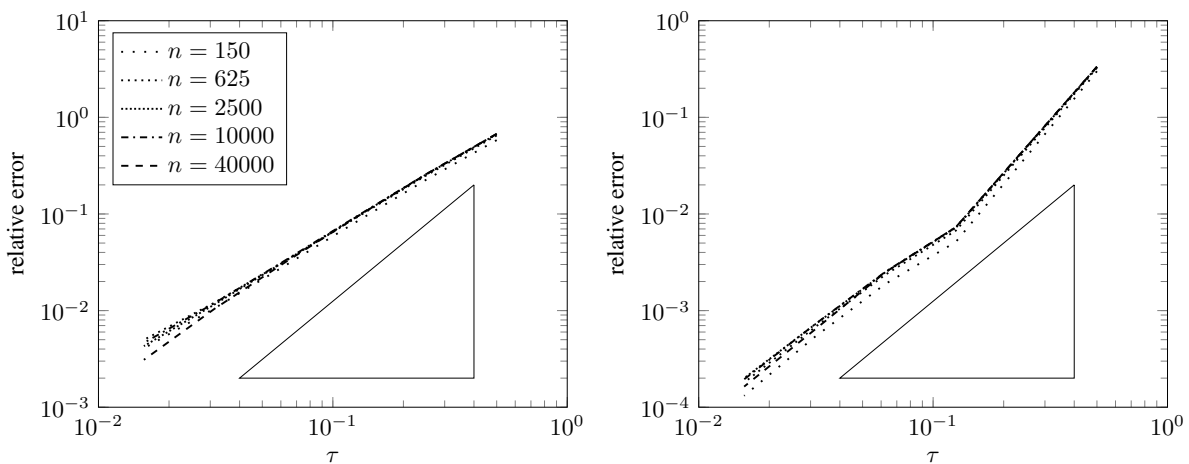
\begin{figure}[H]
\centering
\begin{tikzpicture}[scale=0.9]
\begin{loglogaxis}[xmin=1e-2,xmax=1e-0,ymin=1e-3,ymax=1e+1, legend pos=north west, legend cell align={left},ylabel = relative error,xlabel=$\tau$]
\addplot[color=black, thick, loosely dotted] plot coordinates { ( 0.5, 0.580974 ) ( 0.25, 0.225961 ) ( 0.125, 0.0819605 ) ( 0.0625, 0.0291723 ) ( 0.03125, 0.0105529 ) ( 0.015625, 0.00406301 )   };
\addlegendentry{$n = 150$}
\addplot[color=black, thick, dotted] plot coordinates { ( 0.5, 0.64949 ) ( 0.25, 0.250004 ) ( 0.125, 0.0905021 ) ( 0.0625, 0.0324547 ) ( 0.03125, 0.0120694 ) ( 0.015625, 0.00498563 )  };
\addlegendentry{$n = 625$}
\addplot[color=black, thick, densely dotted] plot coordinates { ( 0.5, 0.66911 ) ( 0.25, 0.256239 ) ( 0.125, 0.0921003 ) ( 0.0625, 0.0324804 ) ( 0.03125, 0.0115667 ) ( 0.015625, 0.00429072 )   };
\addlegendentry{$n = 2500$}
\addplot[color=black, thick, dashdotted] plot coordinates { ( 0.5, 0.674391 ) ( 0.25, 0.258243 ) ( 0.125, 0.092983 ) ( 0.0625, 0.0329849 ) ( 0.03125, 0.0119458 ) ( 0.015625, 0.00464054 )  };
\addlegendentry{$n = 10000$}
\addplot[color=black, thick, dashed] plot coordinates { ( 0.5, 0.675157 ) ( 0.25, 0.257523 ) ( 0.125, 0.0917174 ) ( 0.0625, 0.031528 ) ( 0.03125, 0.0104235 ) ( 0.015625, 0.00309561 )   };
\addlegendentry{$n = 40000$}
\addplot[color=black, thin] plot coordinates { (0.4, 0.2) (0.04, 0.002) (0.4, 0.002) (0.4, 0.2) };
\end{loglogaxis}
\end{tikzpicture}
\begin{tikzpicture}[scale=0.9]
\begin{loglogaxis}[xmin=1e-2,xmax=1e-0,ymin=1e-4,ymax=1e-0, legend pos=north west,ylabel = relative error,xlabel=$\tau$]
\addplot[color=black, thick, loosely dotted] plot coordinates { ( 0.5, 0.30027 )  ( 0.125, 0.00513938 ) ( 0.0625, 0.00188461 ) ( 0.03125, 0.000533602 ) ( 0.015625, 0.000131361 )  };
\addplot[color=black, thick, dotted] plot coordinates { ( 0.5, 0.325127 )  ( 0.125, 0.00672624 ) ( 0.0625, 0.00234951 ) ( 0.03125, 0.000670019 ) ( 0.015625, 0.000180093 )  };
\addplot[color=black, thick, densely dotted] plot coordinates { ( 0.5, 0.33226 )  ( 0.125, 0.00719349 ) ( 0.0625, 0.00248957 ) ( 0.03125, 0.000714752 ) ( 0.015625, 0.000199649 )  };
\addplot[color=black, thick, dashdotted] plot coordinates { ( 0.5, 0.334127 )  ( 0.125, 0.00730559 ) ( 0.0625, 0.00251646 ) ( 0.03125, 0.000718043 ) ( 0.015625, 0.000194521 )  };
\addplot[color=black, thick, dashed] plot coordinates { ( 0.5, 0.334636 )  ( 0.125, 0.0073022 ) ( 0.0625, 0.00249226 ) ( 0.03125, 0.000633795 ) ( 0.015625, 0.000163492 )  };
\addplot[color=black, thin] plot coordinates { (0.4, 0.02) (0.04, 0.0002) (0.4, 0.0002) (0.4, 0.02) };
\end{loglogaxis}
\end{tikzpicture}
\caption{Convergence of the discrete sectional curvature $\kappa_\shell^\tau(v,w)$ (left) and $\kappa_\shell^{\pm\tau}(v,w)$ (right) for a shape $\shell$ represented by the vertex vector in $\R^{3n}$ 
of a triangulation of  $S^2$. 
We plot the relative error $\left|\frac{\kappa - \kappa^\tau}{\kappa}\right|$ (left) as well as $\left|\frac{\kappa - \kappa^{\pm\tau}}{\kappa}\right|$ (right) as a function of $\tau$. 
The thin solid triangles have slope $2$. }
\label{fig:secCurv_DiscreteShells_sphere_convergence}
\end{figure}

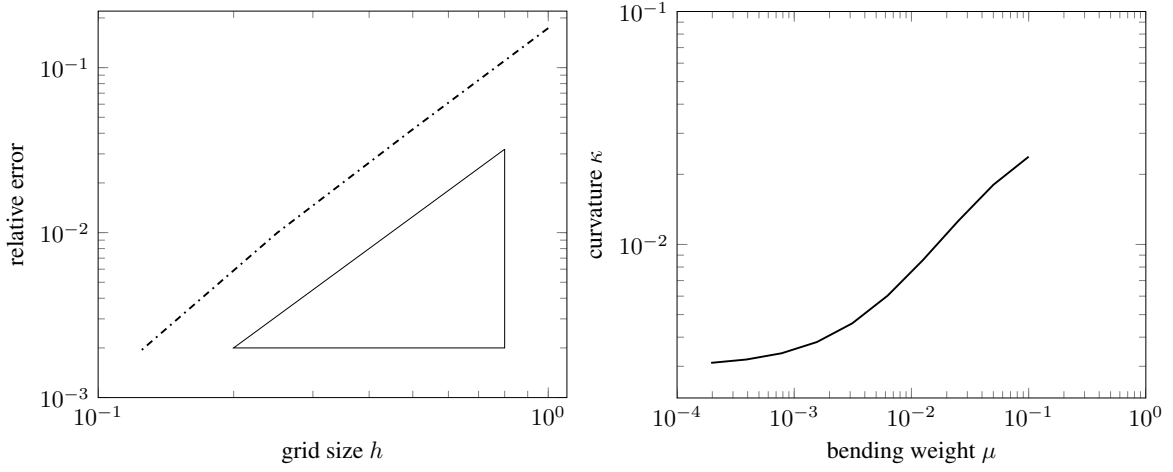
\begin{figure}[H]
\centering
\begin{tikzpicture}[scale=0.9]
\begin{loglogaxis}[xmin=0.1,xmax=1.1,ymin=0.001,ymax=0.22, legend pos=north west,xlabel = grid size $h$, ylabel = relative error]
\addplot[color=black, thick, dashdotted] plot coordinates { ( 1.0, 0.173603 ) ( 0.5, 0.042268 ) ( 0.25, 0.00997576 ) ( 0.125, 0.00194468 )  };
\addplot[color=black, thin] plot coordinates { (0.8, 0.032) (0.2, 0.002) (0.8, 0.002) (0.8, 0.032) };
\end{loglogaxis}
\end{tikzpicture}
\begin{tikzpicture}[scale=0.9]
\begin{loglogaxis}[xmin=0.0001,xmax=1.0, ymax=0.1, xlabel = bending weight $\mu$,ylabel = curvature $\kappa$]
\addplot[color=black, thick] plot coordinates { (0.1, 0.023806) (0.05, 0.0180375) (0.025, 0.0125942)  (0.0125, 0.00855959)  (0.00625, 0.00602473)  (0.003125, 0.00458566)  (0.0015625, 0.00381388)  (0.00078125, 0.00341265)  (0.000390625, 0.00320767)  (0.000195313, 0.00310399)  };
\end{loglogaxis}
\end{tikzpicture}
\caption{Left: convergence of the discrete sectional curvature  $\kappa_\shell^{\pm\tau}(v,w)$ for decreasing spatial grid size $h\to 0$ of the triangular mesh (using fixed $\tau = 0.01$ and $\mu =  10^{-3}$).
Right: dependence of the discrete sectional curvature on  the bending weight $\mu = \delta^2$ (using fixed spatial resolution and $\tau=0.01$).}
\label{fig:secCurv_DiscreteShells_sphere_convergenceInH_differentMus}
\end{figure}

\paragraph{Sectional curvature for a triangulated cactus shape.} As a more complex geometric model, we consider a triangulated cactus having $n = 5261$ vertices and a vector of vertex positions $\shell$.
We investigate different pairs of displacement vectors as tangent vector arguments for the computation of sectional curvatures. 
In Figure~\ref{fig:secCurv_DiscreteShells_cactus_dissipation}, we consider subspaces spanned by the eigenmodes $v_i$ of the Hessian of the elastic energy \eqref{eq:discreteShellsEnergy} and compute the sectional curvature 
for all pairs from the set of eigenmodes of the lowest $i = 1, \ldots, 8$ eigenvalues. All sectional curvatures $\kappa_{ij} := \kappa_\shell(v_i, v_j)$ turn out to be negative. 
Furthermore, one observes varying values of sectional curvature reflecting the specific type of coupling in the underlying pair of eigenmodes.
In particular, $|\kappa_{ij}|$ turns out to be relatively large if $i,j \in \{1,2,7\}$, since all of the corresponding modes represent a strong displacement of the right arm. 
The same is observed with respect to the left arm if $i,j \in \{5,6\}$. Furthermore, all two-dimensional subspaces involving the vibration mode $i=4$ are comparably flat.
In this sense, the second order representation of the shell space manifold reflects the physics of the underlying shell model. 

\begin{figure}[H]
\vspace*{1cm}
\centering
\begin{minipage}[t]{0.35\textwidth}
\begin{tikzpicture}[scale=0.8]
  \begin{axis}[ point meta min=0.0,
            point meta max=2.28,
            xmax=9,
            ymax=9,
            xmin=0,
            ymin=0,
            enlargelimits=false,
            axis on top,
            xtick={1,2,3,4,5,6,7,8},
            ytick={1,2,3,4,5,6,7,8},
            width=8cm,
            height = 8cm]       
     \addplot [matrix plot*,point meta=explicit] file [meta=index 2]{images/shells/secCurv_fullCactus_dissipationModes_0p01mu_tau1_negative_8x8.dat};
  \end{axis}
\end{tikzpicture}\end{minipage}
\begin{minipage}[t]{0.62\textwidth}
 \setlength{\unitlength}{.25\linewidth}%
\begin{picture}(4,2.5)
\put(0,1.4){\resizebox{.95\unitlength}{!}{\includegraphics[trim=350 50 350 20, clip]{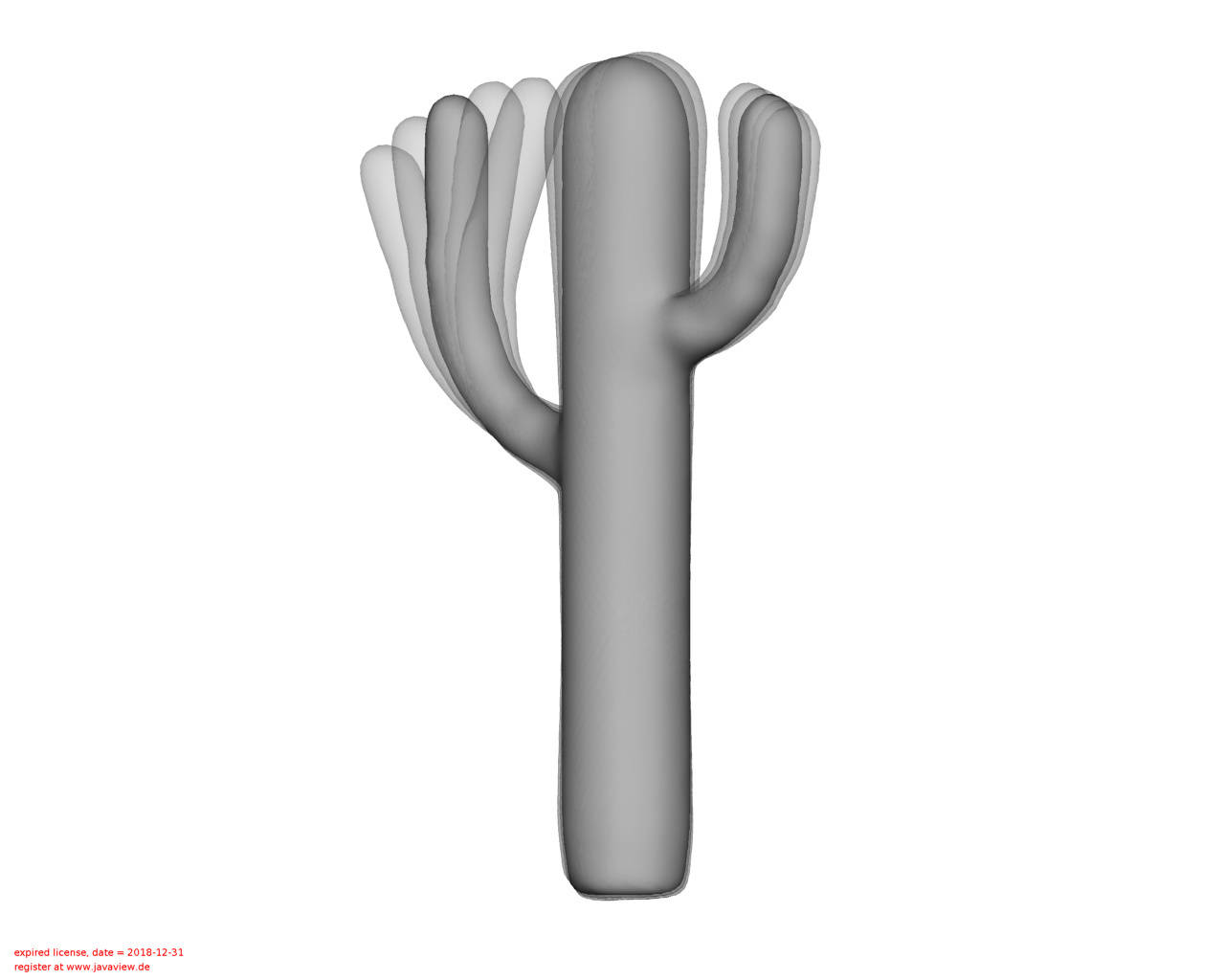}}}
\put(0.15,1.5){$1$}
\put(1,1.4){\resizebox{.95\unitlength}{!}{\includegraphics[trim=350 50 350 20, clip]{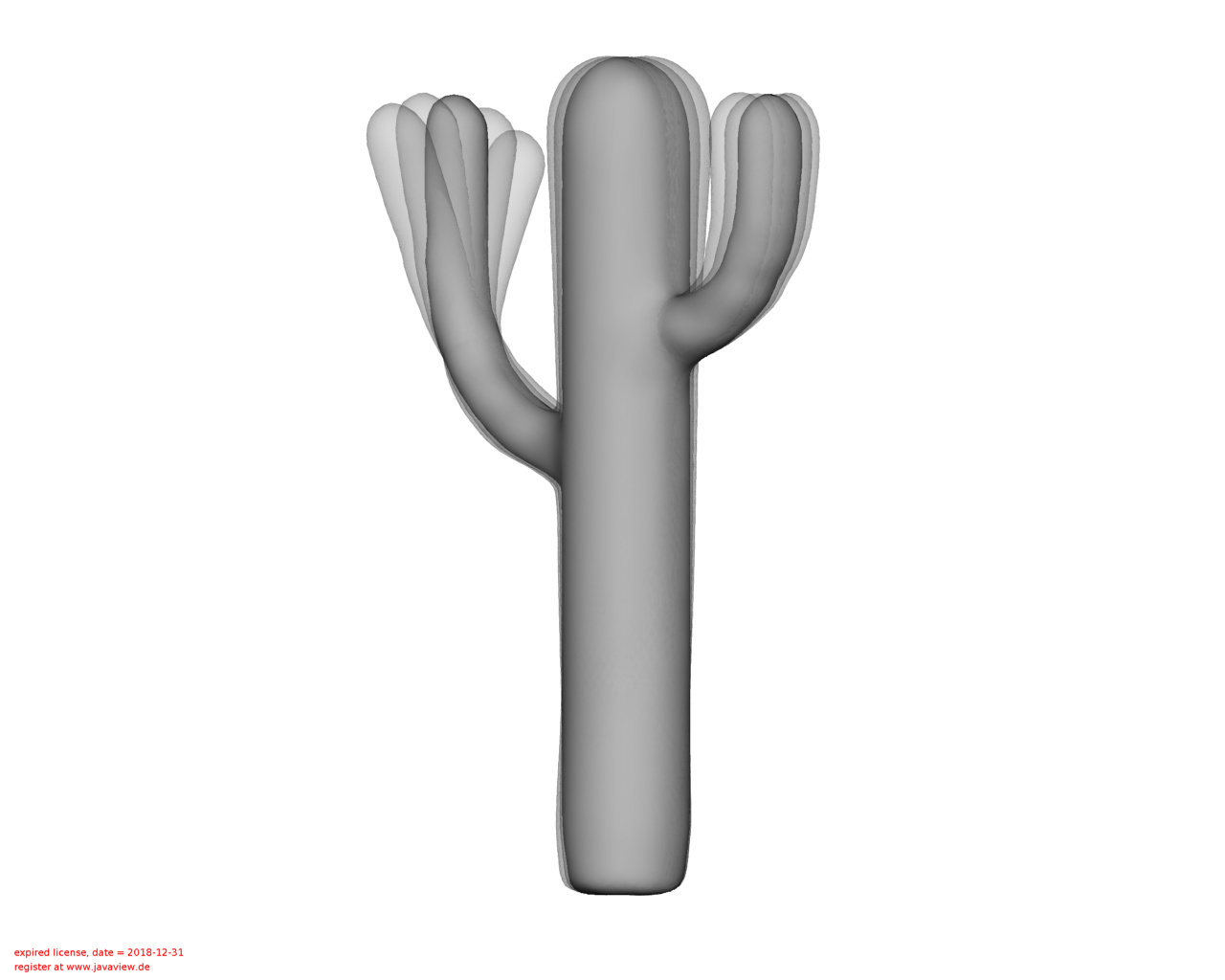}}}
\put(1.15,1.5){$2$}
\put(2,1.4){\resizebox{.95\unitlength}{!}{\includegraphics[trim=350 50 350 20, clip]{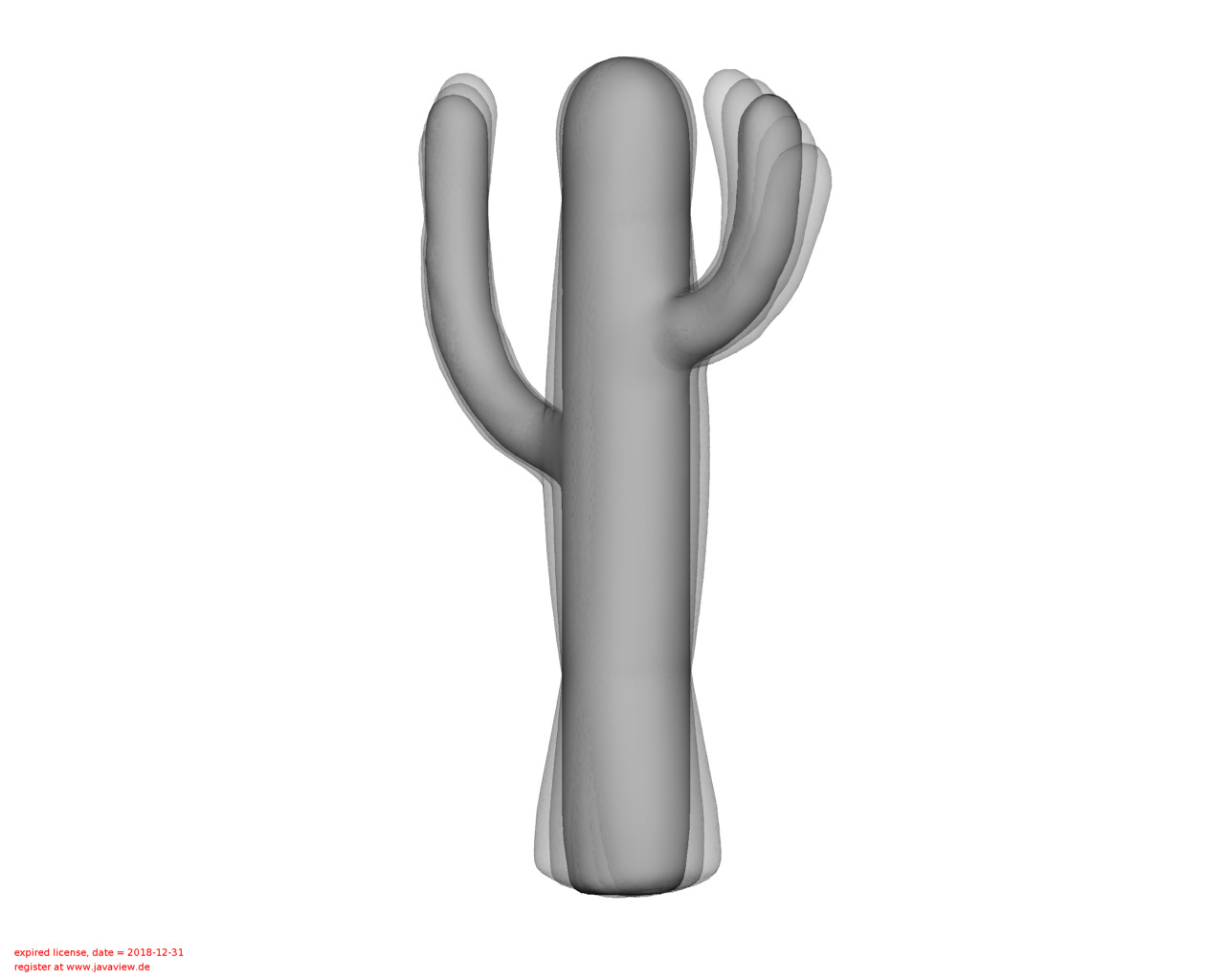}}}
\put(2.15,1.5){$3$}
\put(3,1.4){\resizebox{.95\unitlength}{!}{\includegraphics[trim=350 50 350 20, clip]{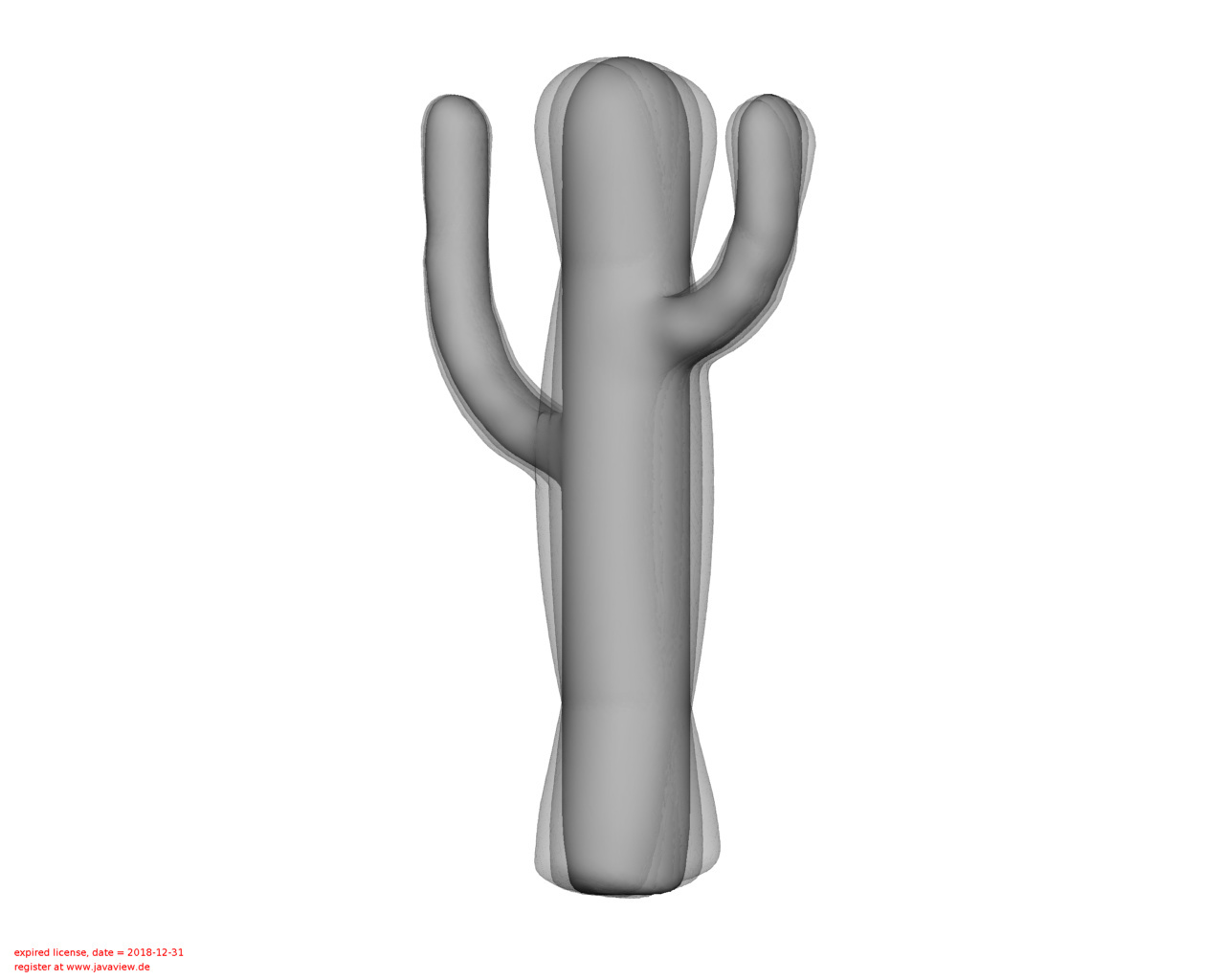}}}
\put(3.15,1.5){$4$}
\put(0,-0.2){\resizebox{.95\unitlength}{!}{\includegraphics[trim=350 50 350 20, clip]{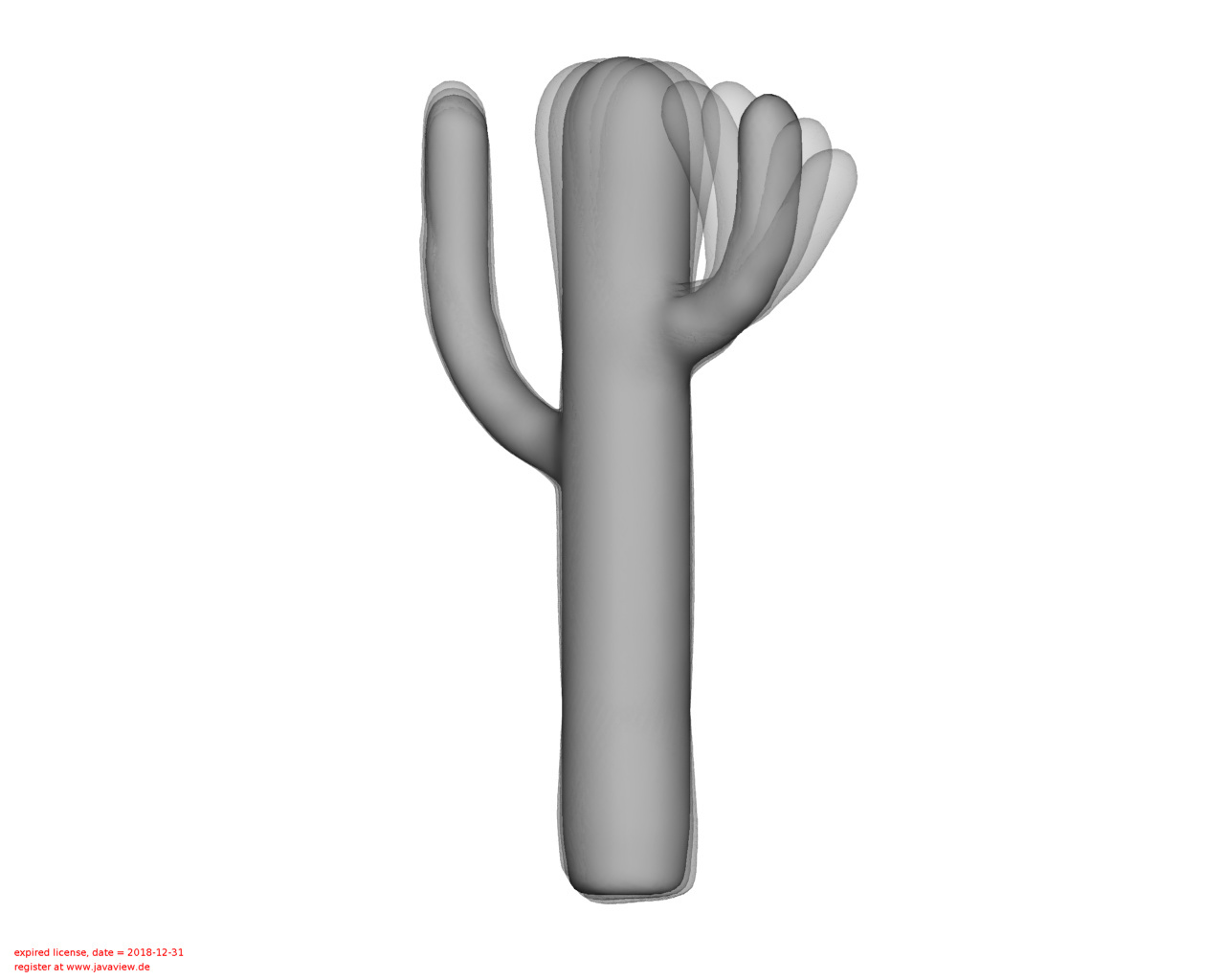}}}
\put(0.15,-0.1){$5$}
\put(1,-0.2){\resizebox{.95\unitlength}{!}{\includegraphics[trim=350 50 350 20, clip]{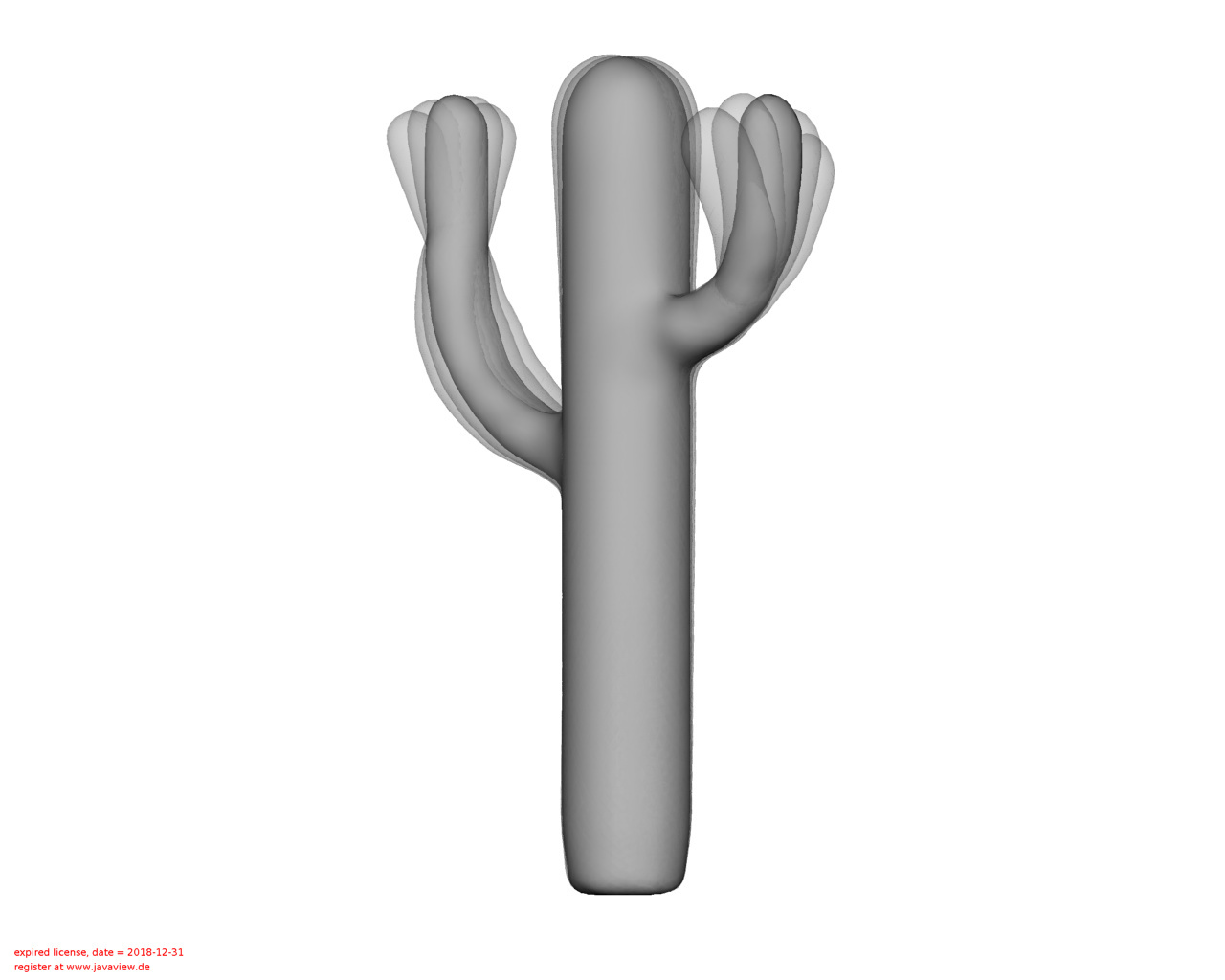}}}
\put(1.15,-0.1){$6$}
\put(2,-0.2){\resizebox{.95\unitlength}{!}{\includegraphics[trim=350 50 350 20, clip]{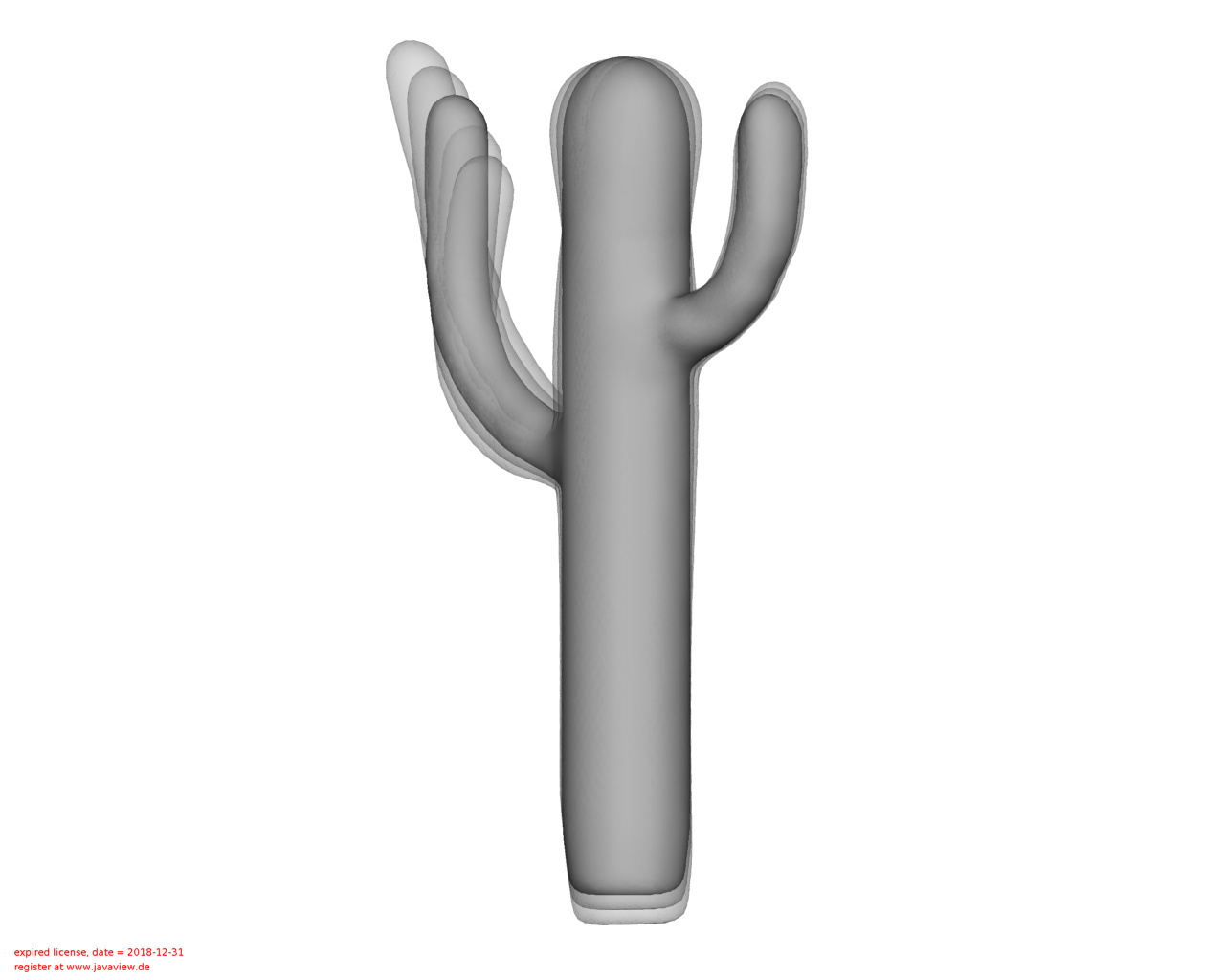}}}
\put(2.15,-0.1){$7$}
\put(3,-0.2){\resizebox{.95\unitlength}{!}{\includegraphics[trim=350 50 350 20, clip]{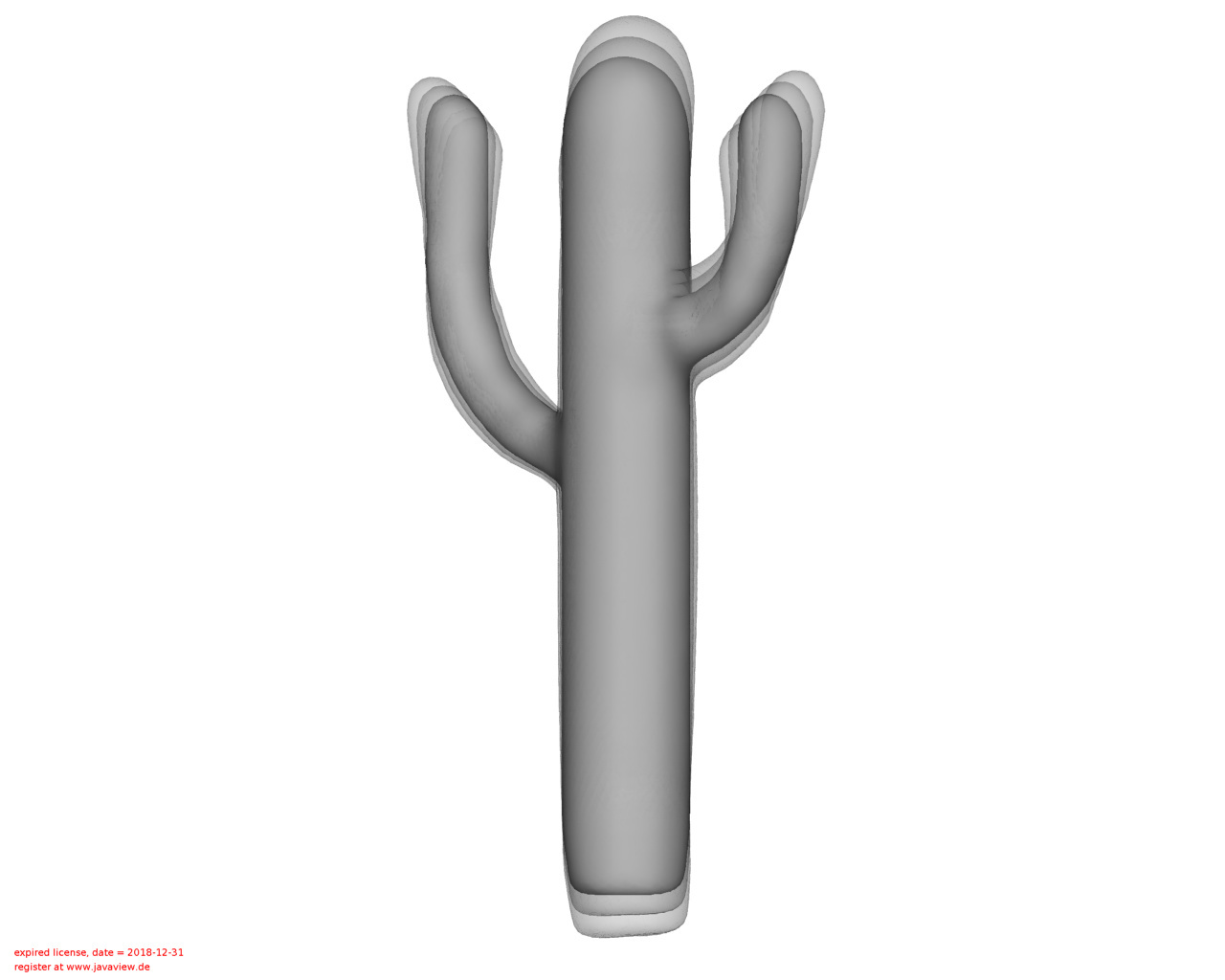}}}
\put(3.15,-0.1){$8$}
\end{picture}
\end{minipage}\\[1ex]
\caption{Sectional curvatures $\kappa_{ij} := \kappa^{\pm\tau}_{\shell}(v_i, v_j)$ for pairs of eigenmodes $v_{i}$, $i = 1, \ldots, 8$, at the discrete shell $\shell$ using $\mu = 10^{-2}$ and $\tau = 10^{-2}$. 
Left: Confusion matrix $\min(\bar\kappa, |\kappa_{ij}|)_{ij}$ for $\bar\kappa = 2.28$ being the absolute value of the $90$th percentile of $|(\kappa_{ij})_{ij}|$, with colour scale $0$ \protect\includegraphics[width=0.175\linewidth]{images/colorbar} $\bar\kappa$. 
Right: Visualization of modes rendered as $\exp_{\shell}(\sigma v_i)$ for $i=1,\ldots,4$ (top) and $i=5,\ldots,8$ (bottom), with $\sigma = 0$ solid and $\sigma \neq 0$ transparent. }
\label{fig:secCurv_DiscreteShells_cactus_dissipation}
\end{figure}
In contrast, in Figure~\ref{fig:secCurv_DiscreteShells_cactus_deform} we consider manually designed deformations $\phi_{1},\ldots,\phi_{16}$ of $\shell$.
These deformations are generated as follows: 
The foot of the cactus (\ie the \emph{fixed handle region}) is fixed while certain \emph{free handle regions} (e.g. the top of the cactus, or tip of the arms) have been deformed smoothly in space. 
In Figure~\ref{fig:secCurv_DiscreteShells_cactus_deform} (top row) we show deformed shapes $\shell_i := \phi_i(\shell)$ for $i = 1, \ldots, 16$.
Finally, we compute the sectional curvatures of the subspaces  spanned by pairs of the discrete logarithms $v_{i}$ of $\shell_i$ at $\shell$ for $i\in \{1,\ldots,16\}$ (computed for $K=8$)
and  show the corresponding confusion matrix with entries $\kappa_{ij} := \kappa_{\shell}^{\pm\tau}(v_i, v_j)$ for $i,j \in \{1,\ldots,16\}$. 
The colour coding of the confusion matrix is as in Figure~\ref{fig:secCurv_DiscreteShells_cactus_dissipation}. Note that we display $|\kappa_{ij}|$ since $\kappa_{ij} < 0$ for almost all choices of $i \neq j$. 
Additionally, we take into account substantially different resolutions and tesselations of the triangulated cactus model (grey, orange and green). 
However, the resulting pattern of the confusion matrices for the different triangulations are still similar, although the scaling of the sectional curvatures varies. 
This appears to be related to the fact that the discrete elastic energy $\energy_\shell(\phi_i)$ already differs significantly when comparing the grey, green and orange model, respectively, 
as shown in the histogram in Figure~\ref{fig:secCurv_DiscreteShells_cactus_deform}.
Compared to Figure~\ref{fig:secCurv_DiscreteShells_cactus_dissipation}, for this set of tangent vectors (based on manual generation) a clearer separation of certain subspaces can be identified in the confusion matrix. 
For example, the subset $\{v_1, v_2, v_3\}$ represents deformations which are solely supported on the \emph{left} arm, 
whereas the subset $\{v_4, \ldots, v_{10}\}$ represents deformations which are mainly supported on the \emph{right} arm. 
\begin{figure}
\centering
\setlength{\unitlength}{.06\linewidth}%
\begin{picture}(15,.5)
\put(-0.25,-0.){\resizebox{\unitlength}{!}{\includegraphics{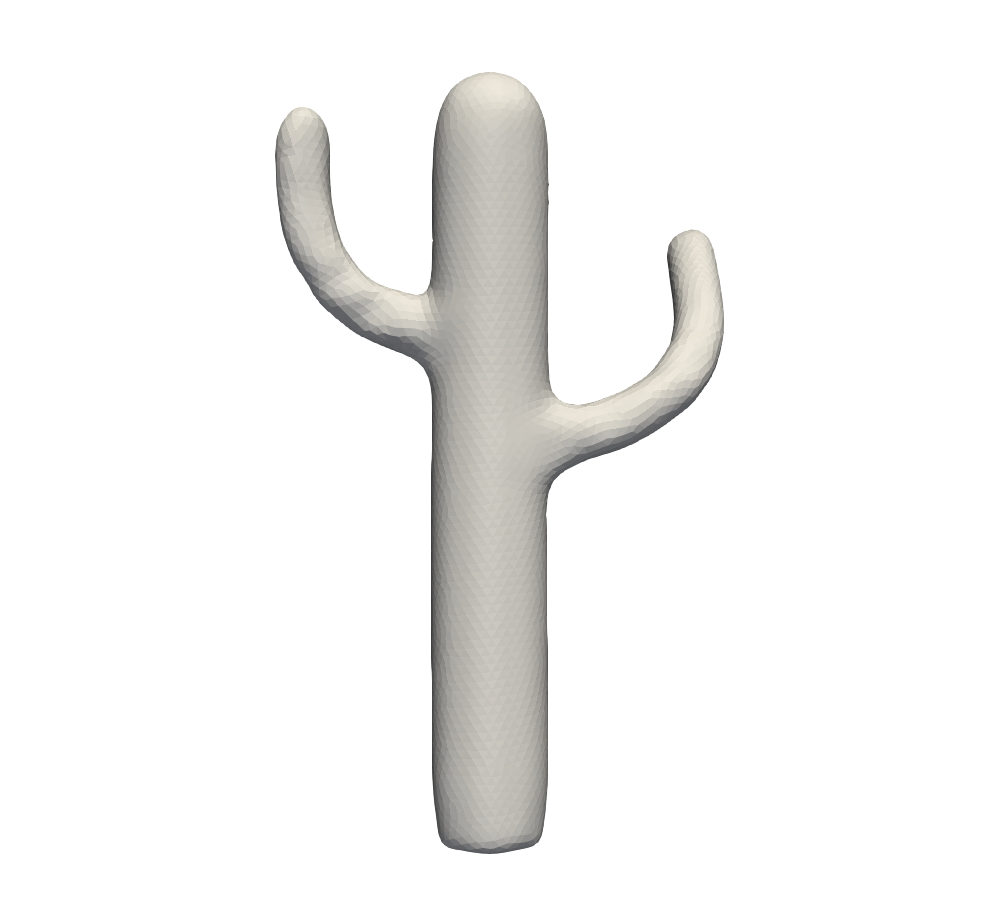}}}
\put(0.75,-0.){\resizebox{\unitlength}{!}{\includegraphics{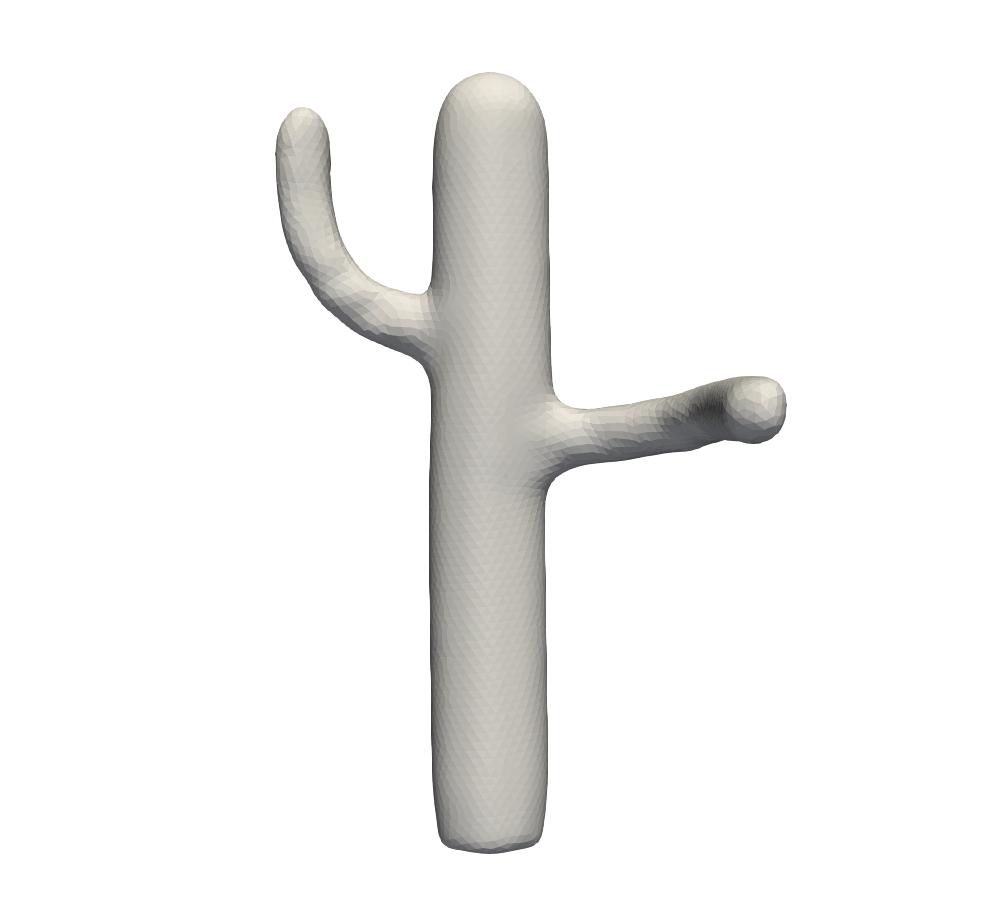}}}
\put(1.75,-0.){\resizebox{\unitlength}{!}{\includegraphics{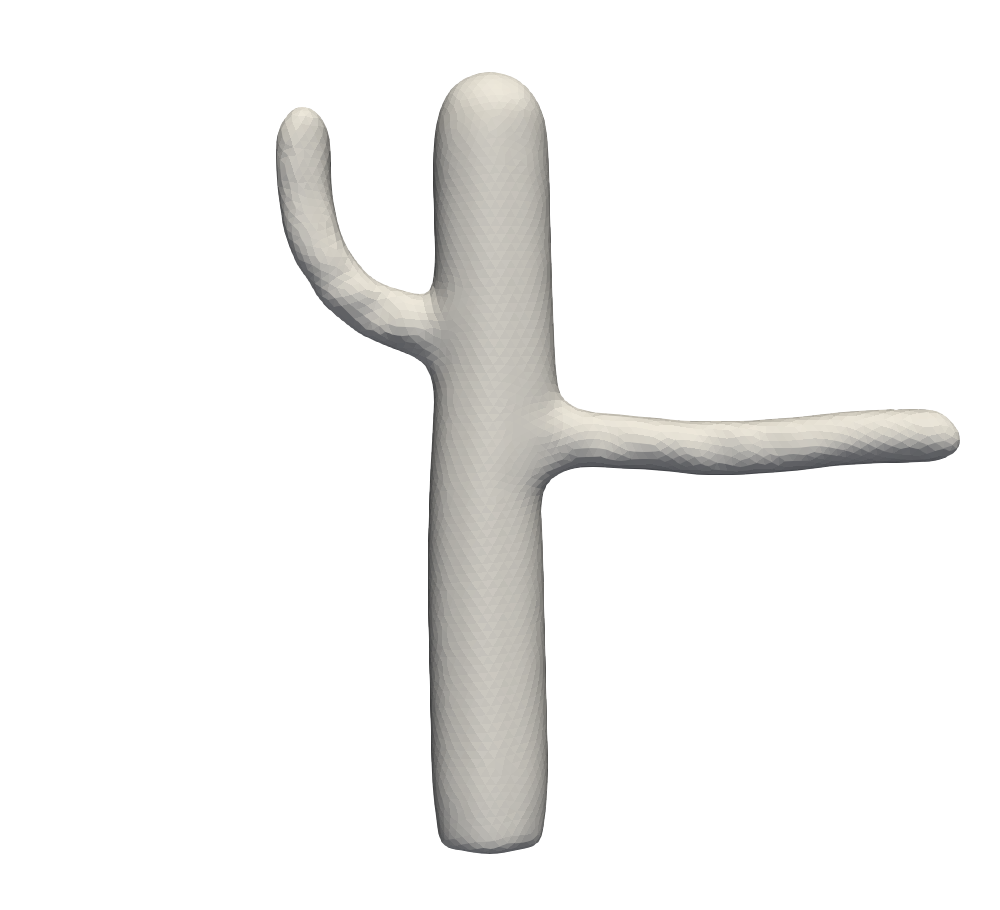}}}
\put(2.75,-0.){\resizebox{\unitlength}{!}{\includegraphics{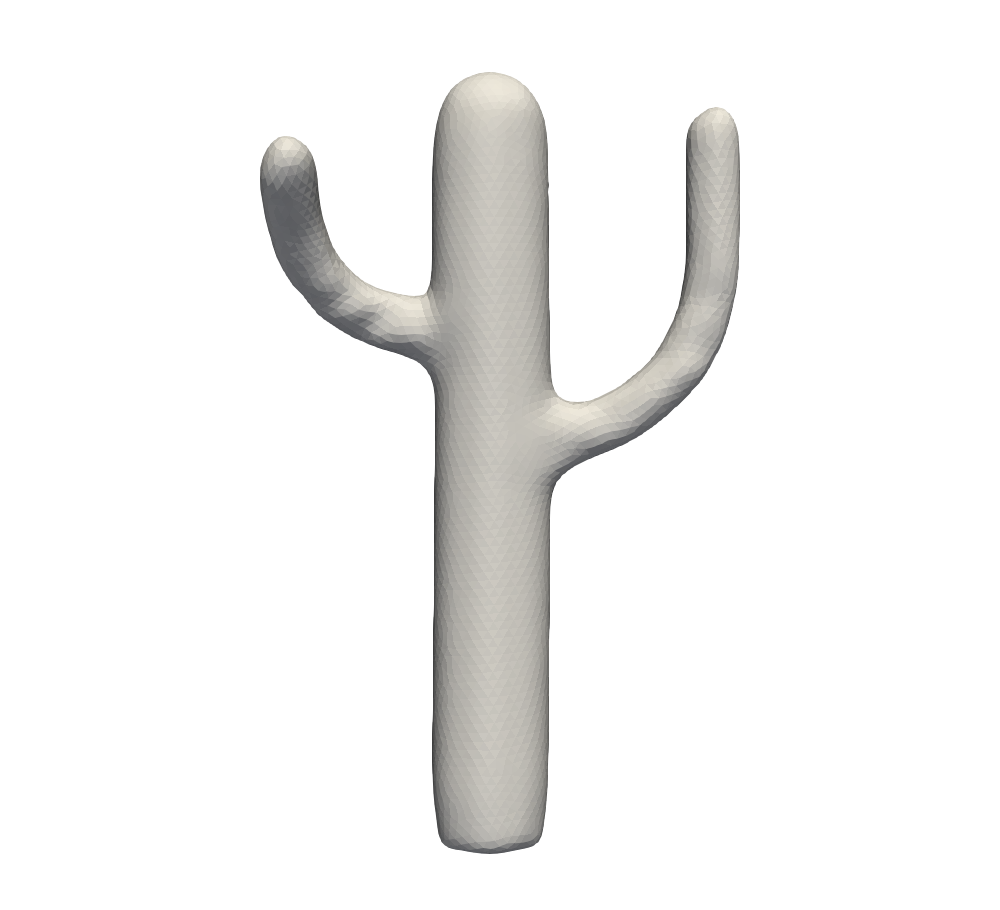}}}
\put(3.75,-0.){\resizebox{\unitlength}{!}{\includegraphics{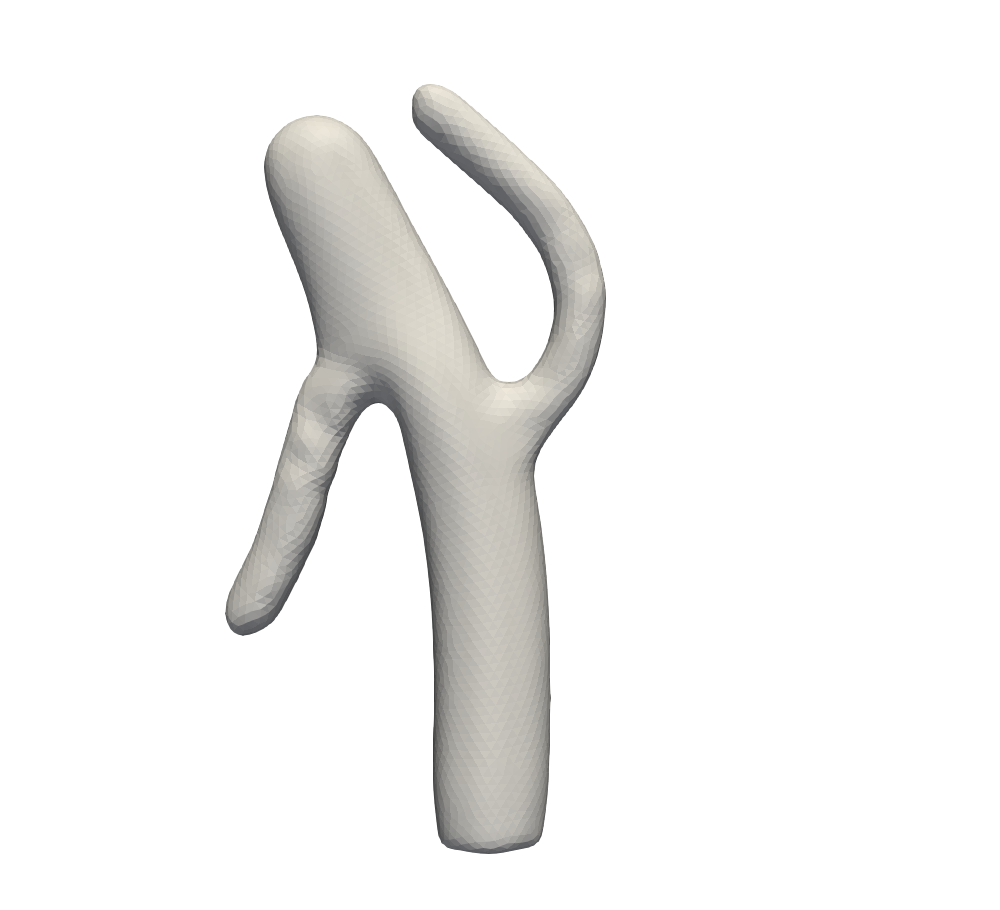}}}
\put(4.75,-0.){\resizebox{\unitlength}{!}{\includegraphics{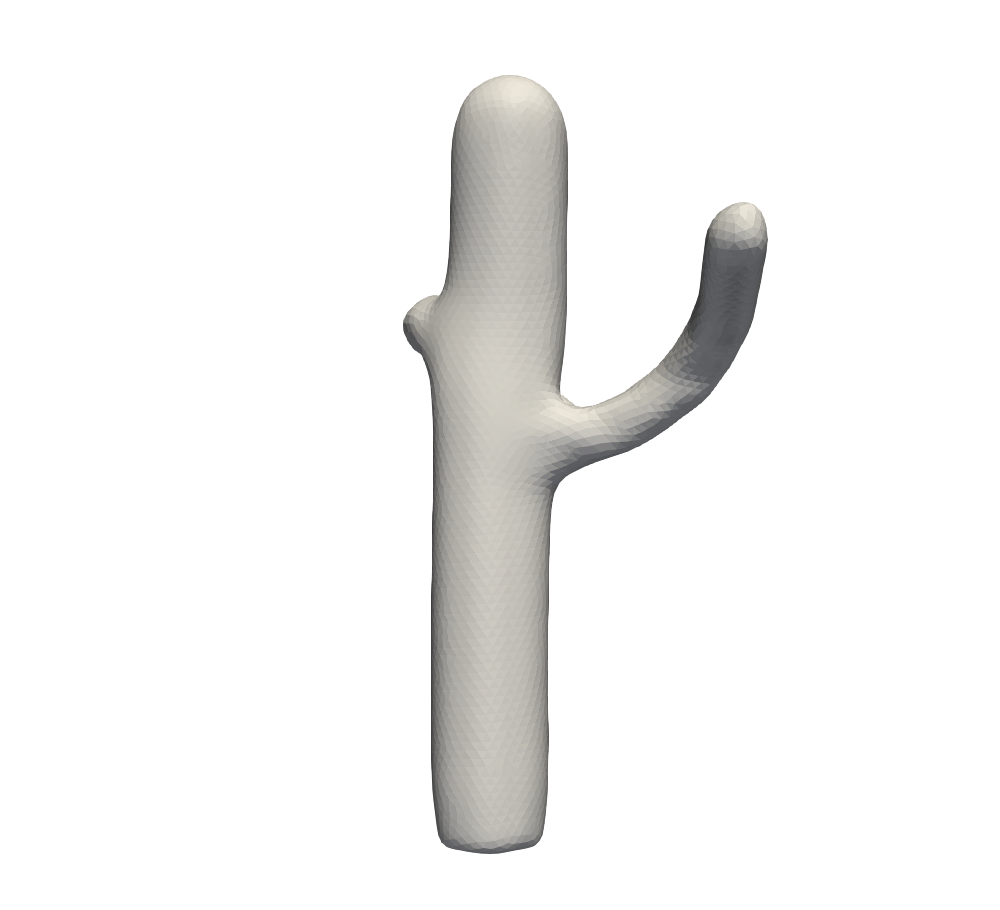}}}
\put(5.75,-0.){\resizebox{\unitlength}{!}{\includegraphics{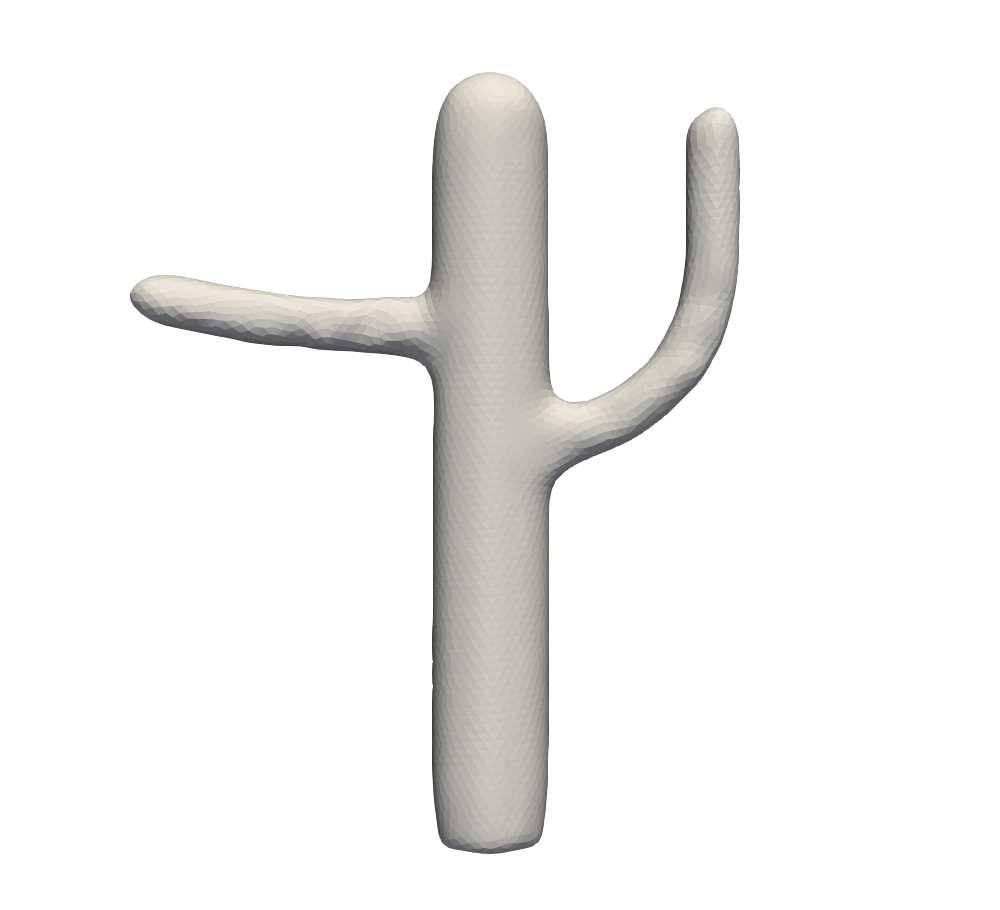}}}
\put(6.7,-0.){\resizebox{\unitlength}{!}{\includegraphics{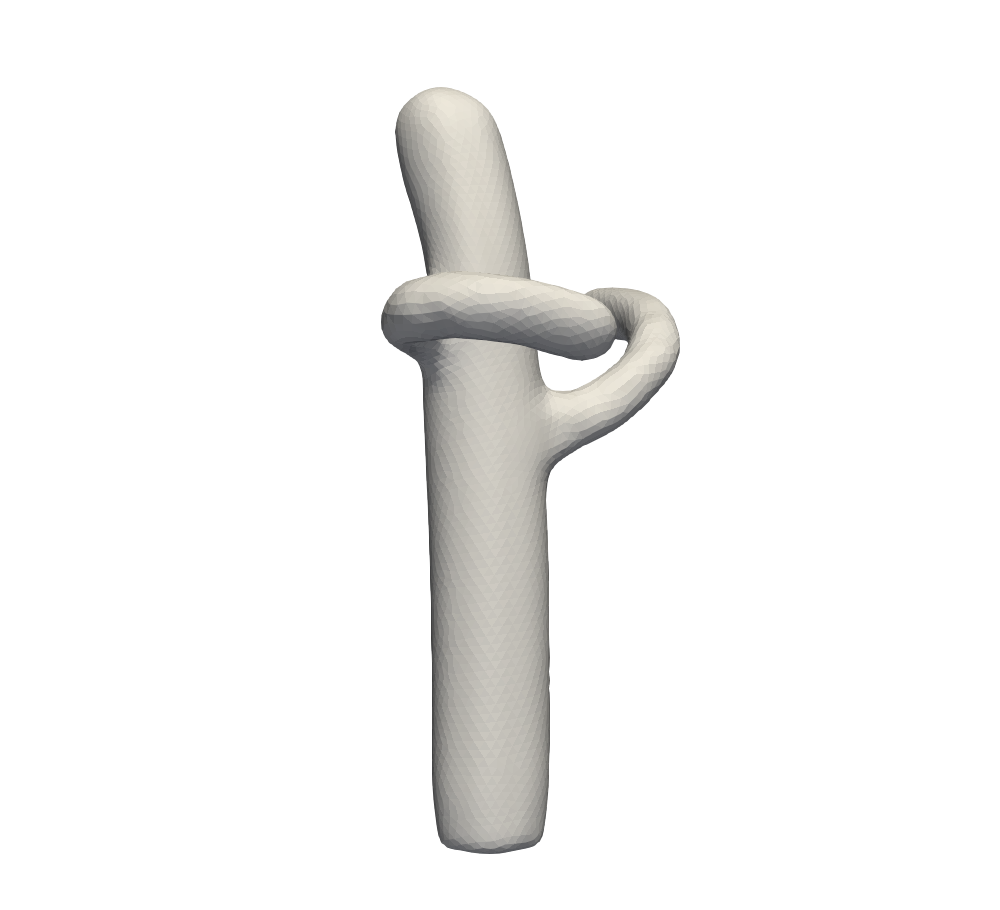}}}
\put(7.7,-0.){\resizebox{\unitlength}{!}{\includegraphics{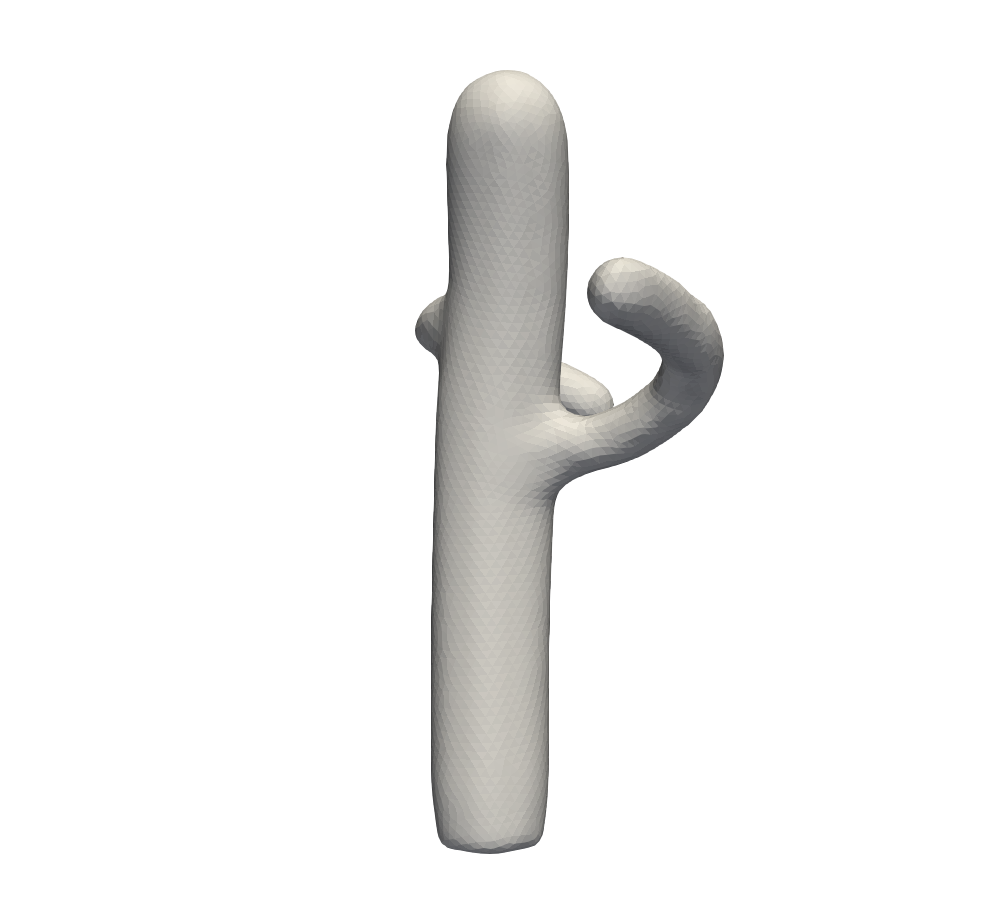}}}
\put(8.7,-0.){\resizebox{\unitlength}{!}{\includegraphics{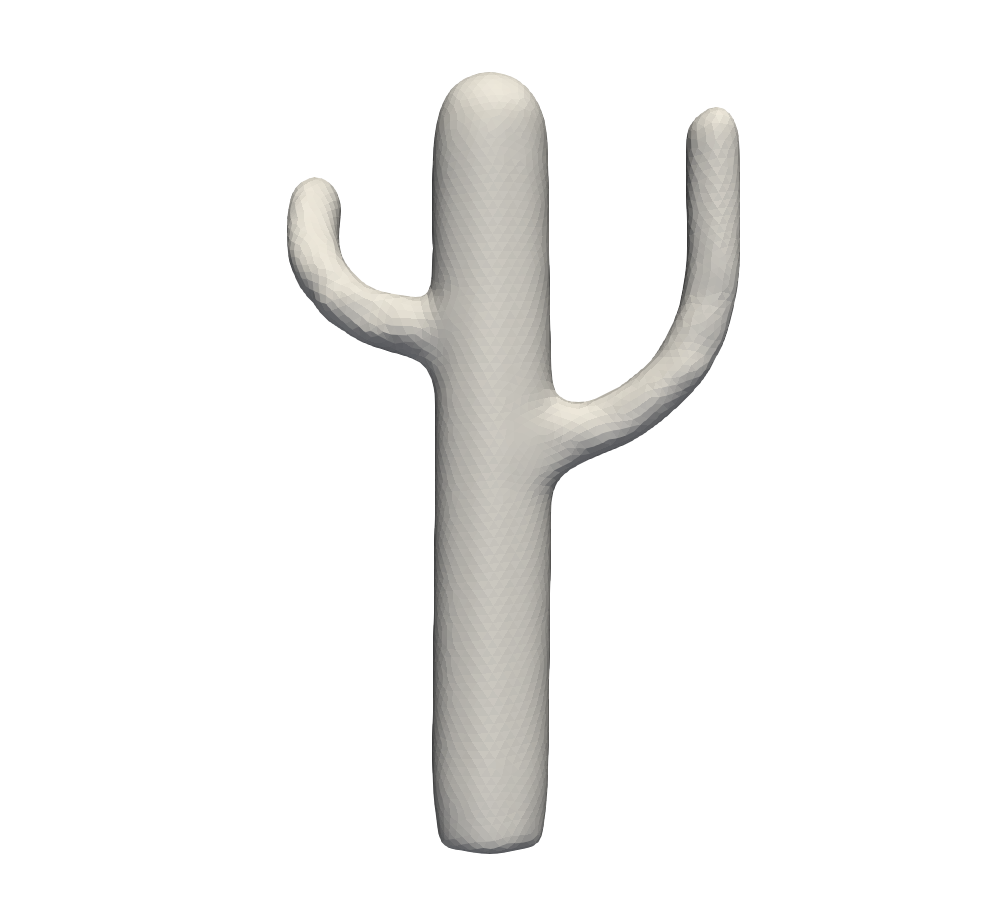}}}
\put(9.75,-0.){\resizebox{\unitlength}{!}{\includegraphics{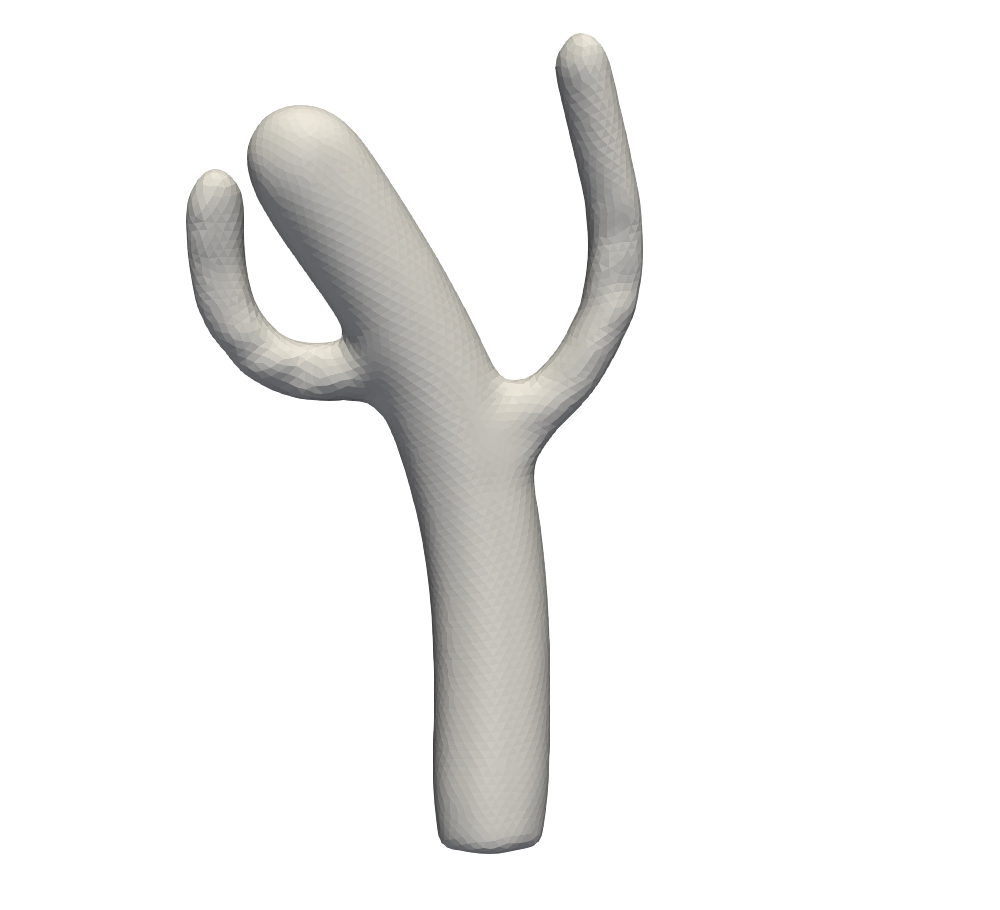}}}
\put(10.75,-0.){\resizebox{\unitlength}{!}{\includegraphics{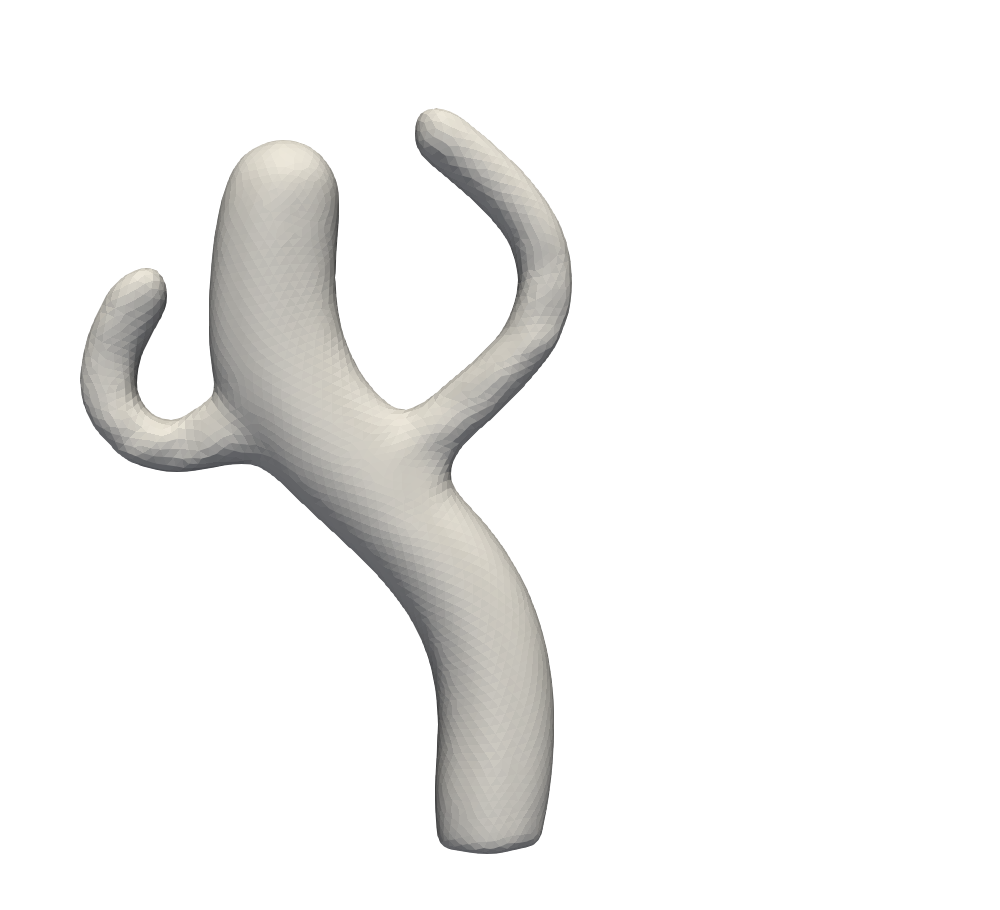}}}
\put(11.5,-0.){\resizebox{\unitlength}{!}{\includegraphics{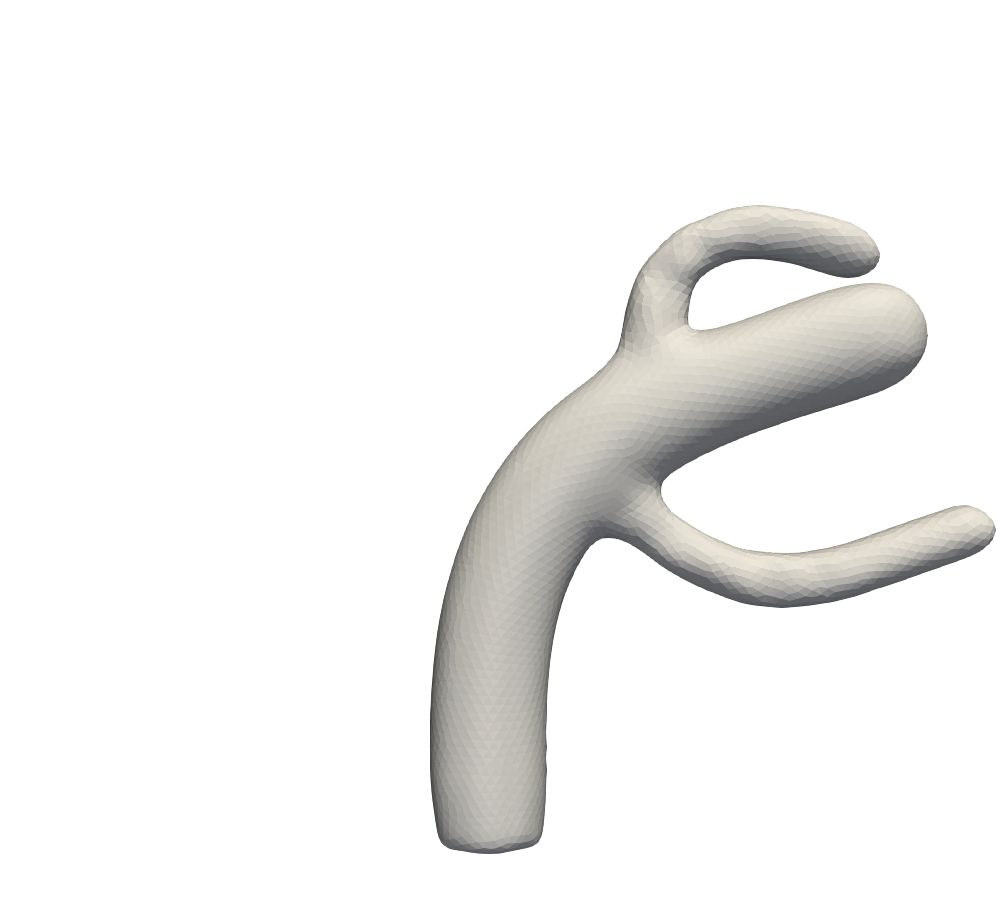}}}
\put(12.75,-0.){\resizebox{\unitlength}{!}{\includegraphics{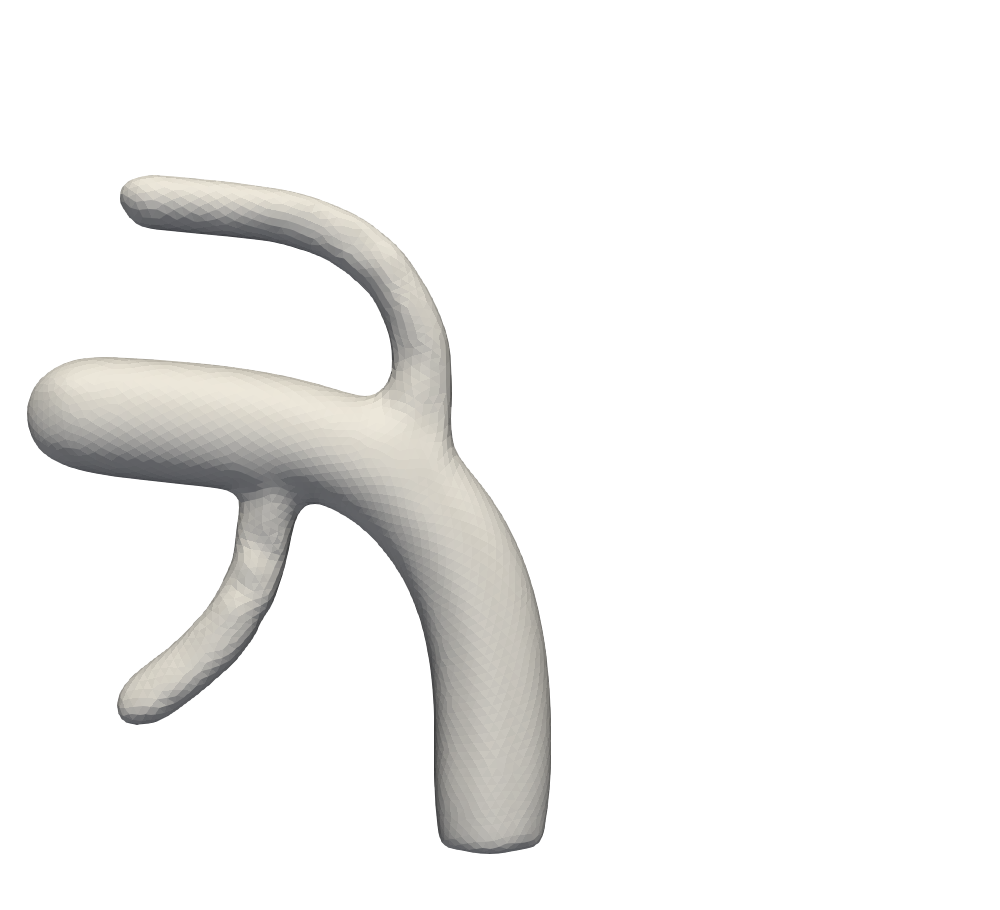}}}
\put(13.6,-0.){\resizebox{\unitlength}{!}{\includegraphics{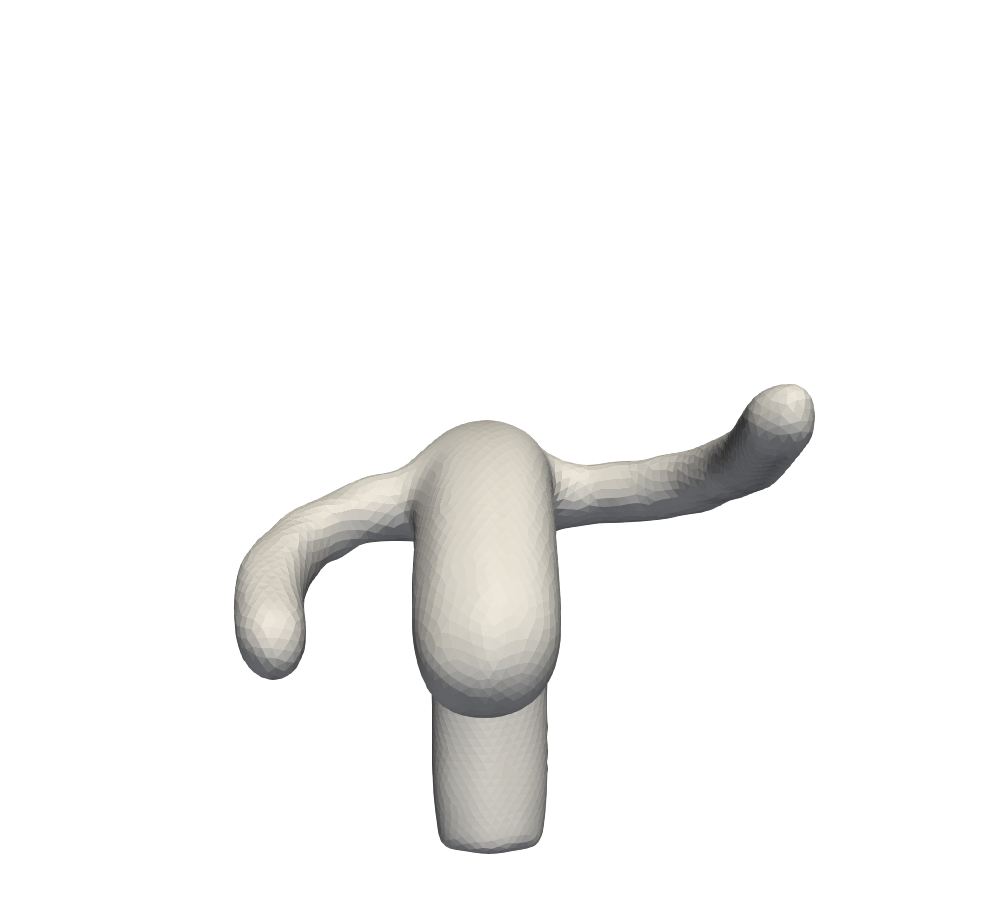}}}
\put(14.6,-0.){\resizebox{\unitlength}{!}{\includegraphics{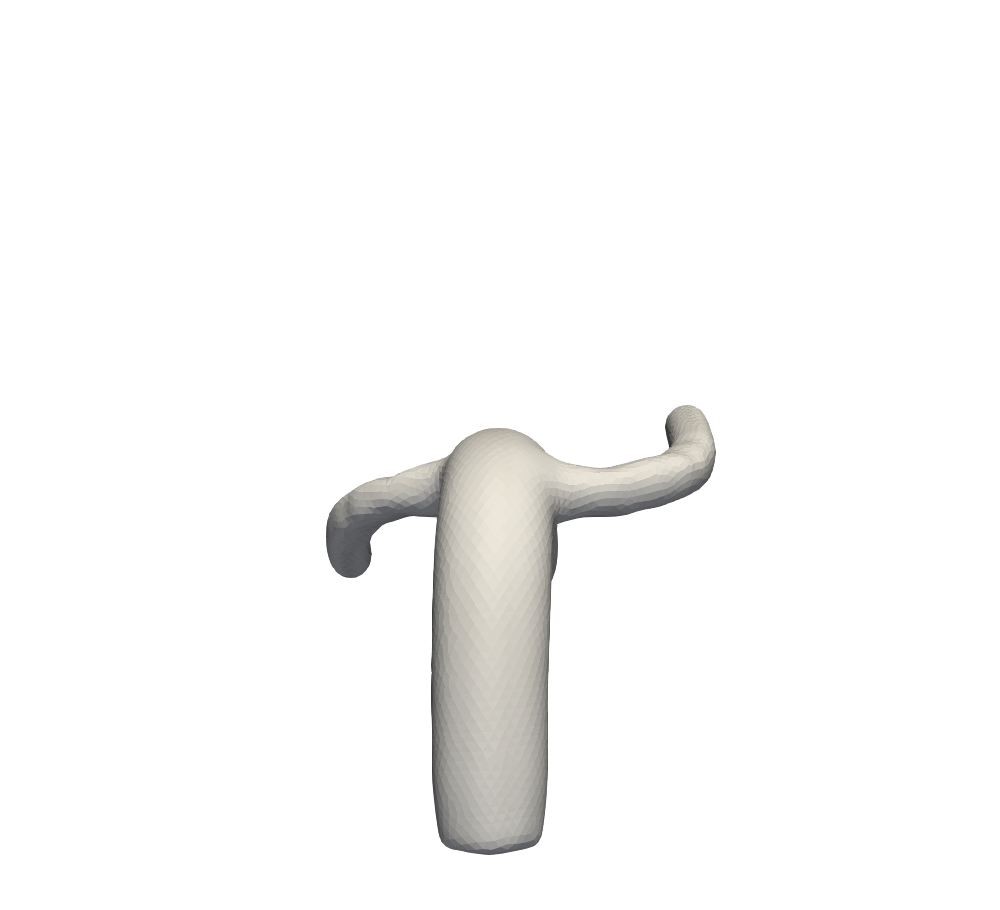}}}
\end{picture}\\[1ex]
\resizebox{0.97\linewidth}{!}{
\begin{tikzpicture}
{\tiny
\begin{axis}[ybar, x=0.2cm, bar width=0.15cm, bar shift=0pt, height=3.5cm, xmin = -1.0, xmax = 63, ymax=0.8, ymin=0, ytick style={draw=none}, yticklabels=none, xtick = { 1,  5,  9,  13,  17,  21,  25,  29,  33,  37,  41,  45,  49,  53,  57,  61}, xticklabels = { 1,  2,  3,  4,  5,  6,  7,  8,  9,  10,  11,  12,  13,  14,  15,  16}]
\addplot[fill=myGrey] coordinates {( 0, 0.0114879) ( 12, 0.0108148) ( 8, 0.0187202) ( 4, 0.0419011) ( 16, 0.126912) ( 20, 0.22293) ( 40, 0.0169677) ( 28, 0.0940827) ( 32, 0.0944281) ( 36, 0.0107281) ( 24, 0.0241642) ( 44, 0.111368) ( 48, 0.0267573) ( 52, 0.0373313) ( 56, 0.0558104) ( 60, 0.0556381) };
\addplot[fill=myOrange] coordinates {( 1, 0.0462706) ( 13, 0.026992) ( 9, 0.0910916) ( 5, 0.154544) ( 17, 0.227931) ( 21, 0.45041) ( 41, 0.0443111) ( 29, 0.314651) ( 33, 0.327007) ( 37, 0.0269065) ( 25, 0.0574385) ( 45, 0.342593) ( 49, 0.0942108) ( 53, 0.132074) ( 57, 0.226599) ( 61, 0.202293) };
\addplot[fill=myGreen] coordinates {( 2, 0.0412061) ( 14, 0.037661) ( 10, 0.0553848) ( 6, 0.149572) ( 18, 0.390864) ( 22, 0.769145) ( 42, 0.0439783) ( 30, 0.617378) ( 34, 0.605802) ( 38, 0.040412) ( 26, 0.0859906) ( 46, 0.317283) ( 50, 0.0917931) ( 54, 0.12603) ( 58, 0.134519) ( 62, 0.133563) };
\end{axis}
}
\end{tikzpicture}}\\[1ex]
{\footnotesize
\begin{minipage}[t]{0.45\textwidth}
\begin{tikzpicture}[scale=0.8]
  \begin{axis}[ point meta min=0.0,
            point meta max=1.0994,
            xmax=17,
            ymax=17,
            xmin=0,
            ymin=0,
            enlargelimits=false,
            axis on top,
            xtick={1, 2, 3, 4,5,  6, 7, 8,9, 10,11, 12, 13, 14,15, 16},
            ytick={1, 2, 3, 4,5,  6, 7, 8,9, 10,11, 12, 13, 14,15, 16},
            width=8cm,
            height = 8cm]        
     \addplot [matrix plot*,point meta=explicit] file [meta=index 2]{images/shells/secCurv_fullCactus_0p01mu_tau0_fixedBdry_FINAL_negative_reordered.dat};
  \end{axis}
\end{tikzpicture}
\end{minipage}
}
\hspace*{5mm}
\begin{minipage}[t]{0.45\textwidth}
 \includegraphics[width = 0.85\linewidth]{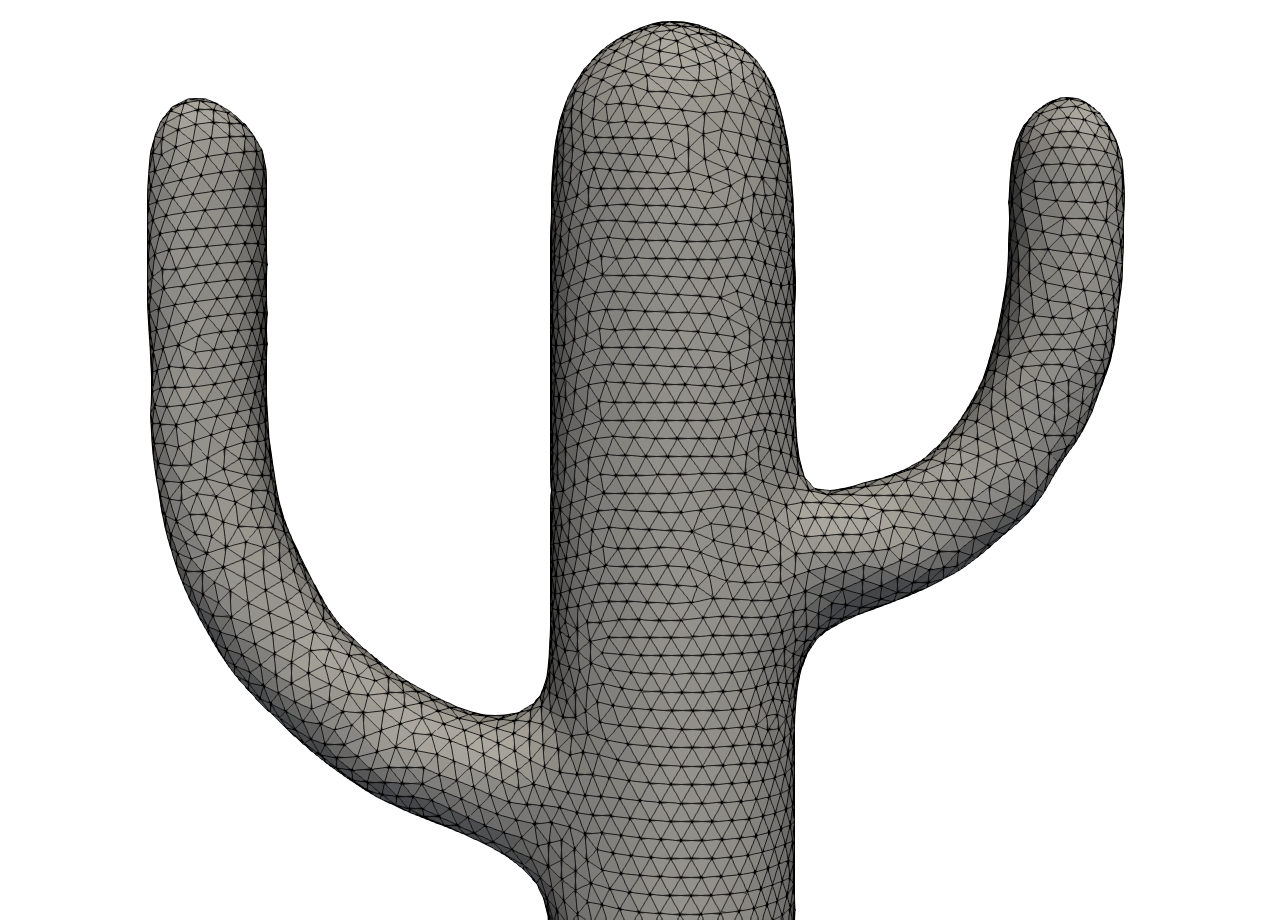}
\end{minipage}\\[1ex]
{\footnotesize
\begin{minipage}[t]{0.45\textwidth}
\begin{tikzpicture} [scale=0.8]
  \begin{axis}[ point meta min=0,
            point meta max=1.27492,
            xmax=17,
            ymax=17,
            xmin=0,
            ymin=0,
            enlargelimits=false,
            axis on top,
            xtick={1, 2, 3, 4,5,  6, 7, 8,9, 10,11, 12, 13, 14,15, 16},
            ytick={1, 2, 3, 4,5,  6, 7, 8,9, 10,11, 12, 13, 14,15, 16},
             width=8cm,
            height = 8cm]        
     \addplot [matrix plot*,point meta=explicit] file{images/shells/secCurv_coarseCactus_0p01mu_tau0_fixedBdry_FINAL_negative_reordered.dat};
  \end{axis}
\end{tikzpicture}
\end{minipage}
}
\hspace*{5mm}
\begin{minipage}[t]{0.45\textwidth}
 \includegraphics[width = 0.85\linewidth]{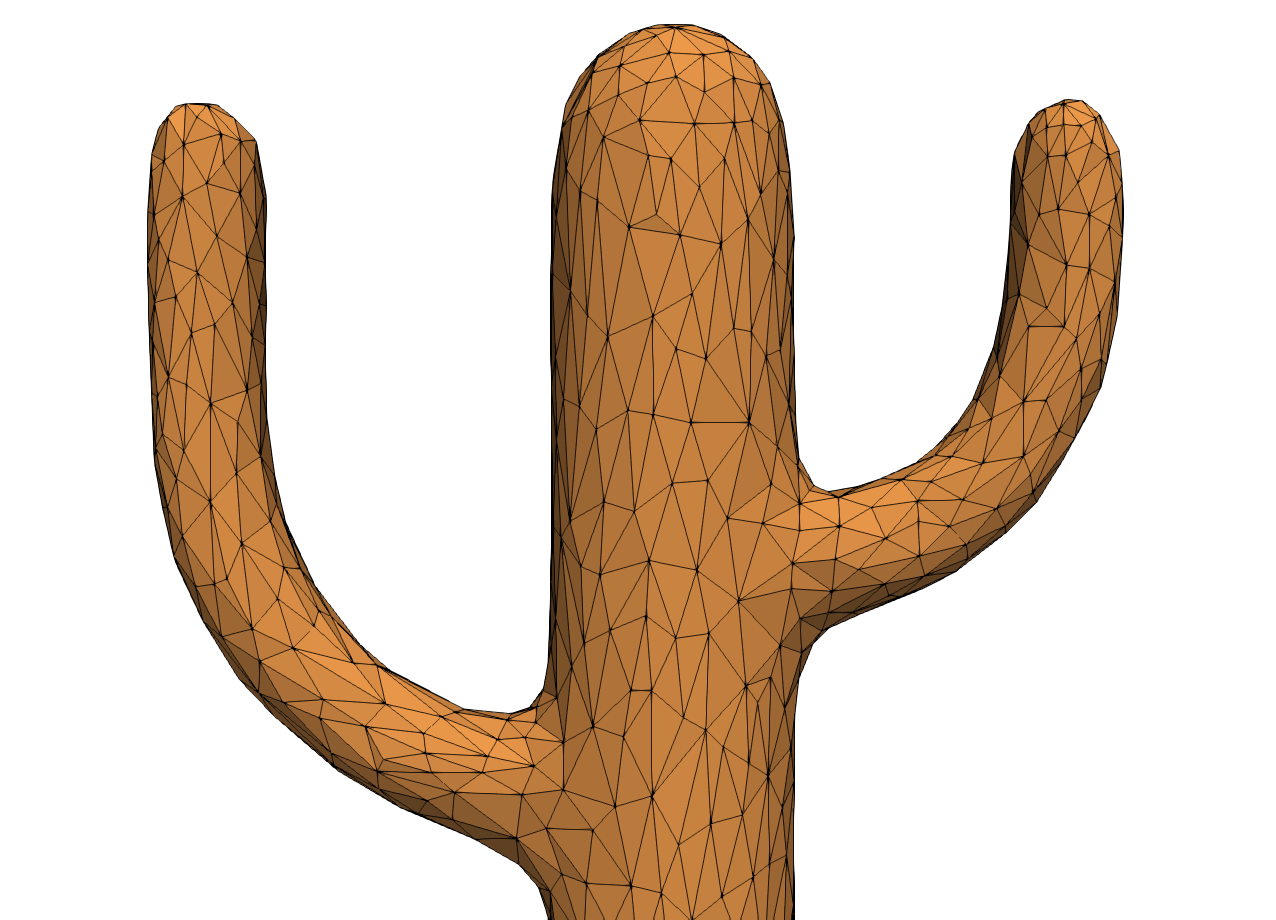}
\end{minipage}\\[1ex]
{\footnotesize
\begin{minipage}[t]{0.45\textwidth}
\begin{tikzpicture} [scale=0.8]
  \begin{axis}[ point meta min=0,
            point meta max=0.249574,
            xmax=17,
            ymax=17,
            xmin=0,
            ymin=0,
            enlargelimits=false,
            axis on top,
          xtick={1, 2, 3, 4,5,  6, 7, 8,9, 10,11, 12, 13, 14,15, 16},
            ytick={1, 2, 3, 4,5,  6, 7, 8,9, 10,11, 12, 13, 14,15, 16},
             width=8cm,
            height = 8cm]        
     \addplot [matrix plot*,point meta=explicit] file{images/shells/secCurv_fullCactus_20pcFlippedEdges_0p01mu_tau0_fixedBdry_FINAL_negative_reordered.dat};
  \end{axis}
\end{tikzpicture}
\end{minipage}
}
\hspace*{5mm}
\begin{minipage}[t]{0.45\textwidth}
 \includegraphics[width = 0.85\linewidth]{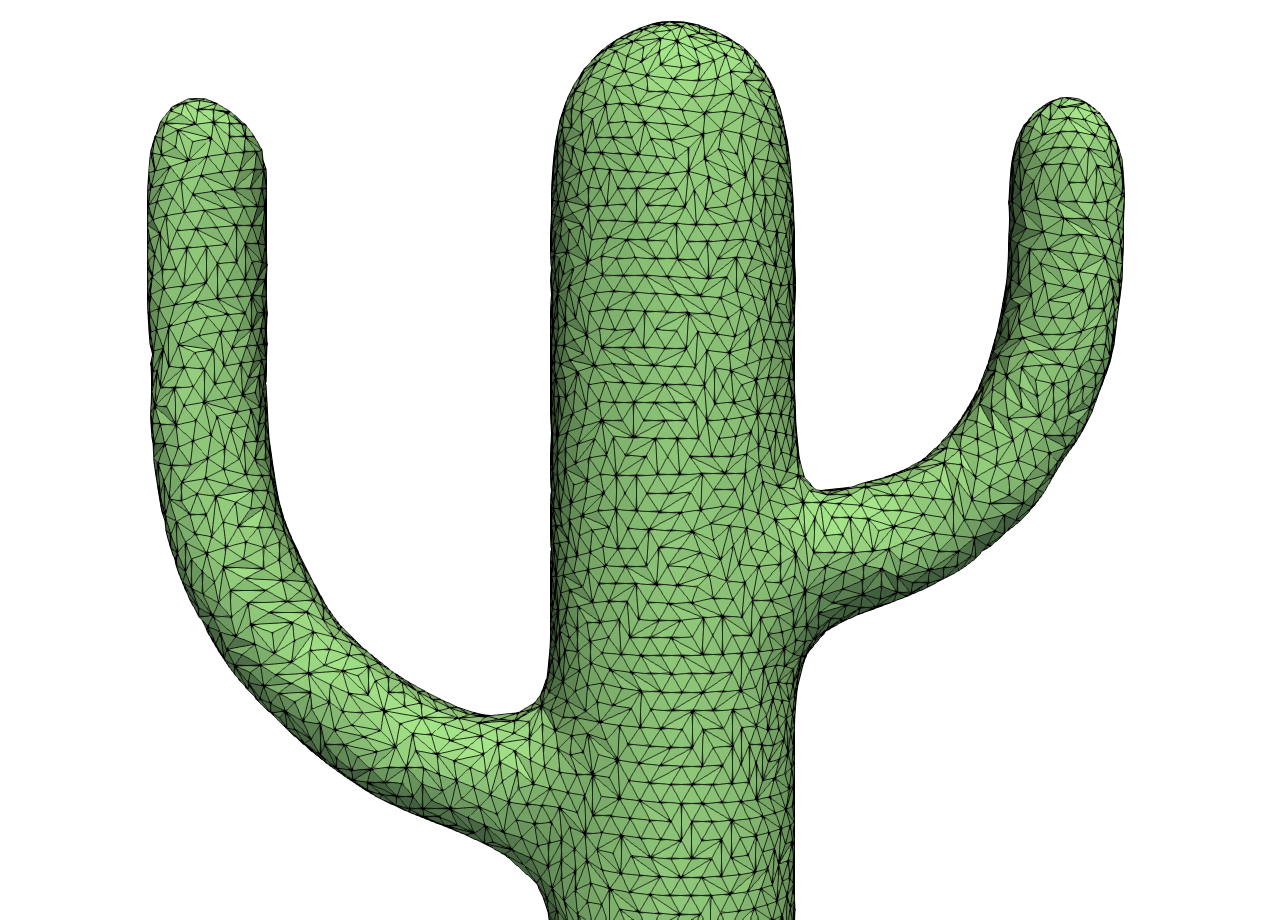}
\end{minipage}
\caption{Top row: the cactus poses $\shell_i = \phi_{i}(\shell)$; second row:  the associated discrete deformation energies $\energy_\shell(\phi_{i})$; 
below: confusion matrices for sectional curvatures $\kappa_{ij}^{\pm\tau} := \kappa_{\shell}(v_i,v_j)$ for $i,j =1,\ldots, 16$ and the underlying triangulation in grey, orange, and green (using $\mu = 10^{-2}$ and $\tau = 10^{-2}$).
For each row we show the confusion matrix $\min(\bar\kappa, |\kappa_{ij}|)_{ij}$ with colour scale $0$ \protect\includegraphics[width=2.5\unitlength]{images/colorbar} $\bar\kappa$. 
Again, $\bar\kappa$ is the absolute value of the $90$th percentile of $|(\kappa_{ij})_{ij}|$, where $\bar\kappa = 1.10$ (grey), $\bar\kappa = 1.27$ (orange) and $\bar\kappa = 0.25$ (green).
} \label{fig:secCurv_DiscreteShells_cactus_deform}
\end{figure}

\paragraph{Positive sectional curvature and nonuniqueness of geodesics.}
It is well-known that estimates of the injectivity radius of the exponential map on a manifold depend on sectional curvature bounds
and in particular that positive sectional curvature may be associated with nonuniqueness of geodesics.
In Figure\,\ref{fig:positiveCurvature} we provide a corresponding example in the space of discrete shells:
We consider a flat rectangle with two bumps pointing in opposite directions,
where the height of the bumps is indicated by the parameters $\zeta$ and $\eta$.
This example was first given in \cite[Fig.\,5]{HeRuWa12} to show nonuniqueness of shortest geodesics.
Indeed, the extremal grey shells in Figure\,\ref{fig:positiveCurvature} correspond to $(\zeta>0,\eta<0)$ and $(\zeta<0,\eta>0)$, respectively.
Thus, to get from one shell to the other, both bumps need to be flipped through.
If the bumps are flipped through simultaneously, this induces a strong distortion energy since then both bumps compete for in-plane area.
Therefore, it is energetically more favourable to flip both bumps through one after another,
and indeed there are two shortest (discrete) geodesics connecting both shells (green and orange) that only differ in the order of the flips.
Intuitively this means that somewhere in between these two geodesics there is a region in shell space
that is avoided by geodesics since (thinking in terms of a chart) there the Riemannian metric becomes large.
At that point in the chart where the Riemannian metric is locally maximal the sectional curvature should be positive.
This intuition is confirmed by computing the sectional curvature at the central flat rectangle in the directions of $\zeta$ and $\eta$,
which turns out to almost reach the value $1000$.
As for the previous example, we confirmed (but do not show here) robustness of this behaviour with respect to triangulation modifications and convergence in $\tau$ and in the triangulation width.

\begin{figure}
\centering
\resizebox{0.6\linewidth}{!}{
\begin{tikzpicture}
\node[inner sep=0pt] (first) at (0,0) {\includegraphics[trim=10 700 10 700, clip,width=.4\textwidth]{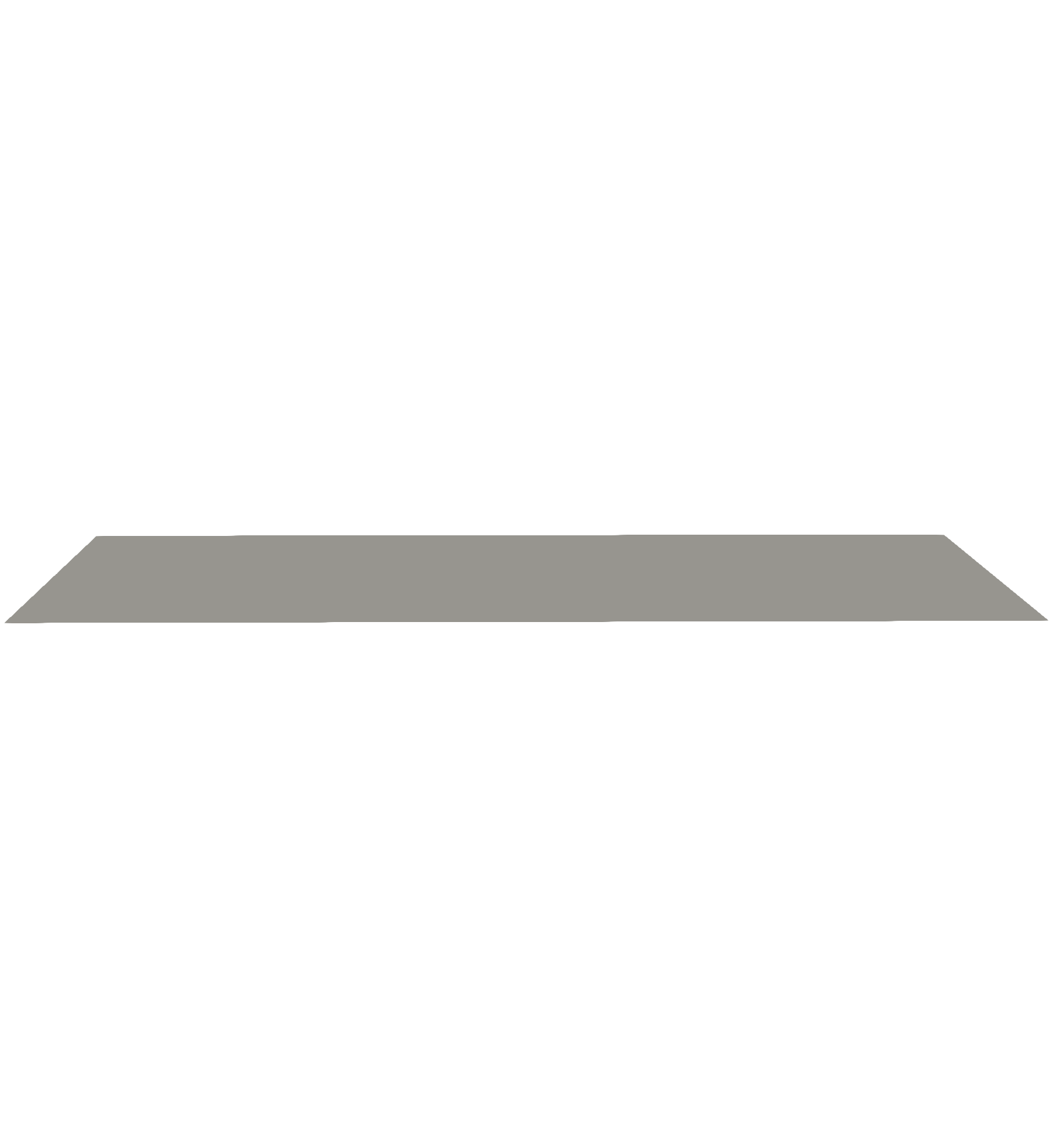}};
\node[inner sep=0pt] (first) at (-6,6.6) {\includegraphics[trim=10 700 10 700, clip,width=.4\textwidth]{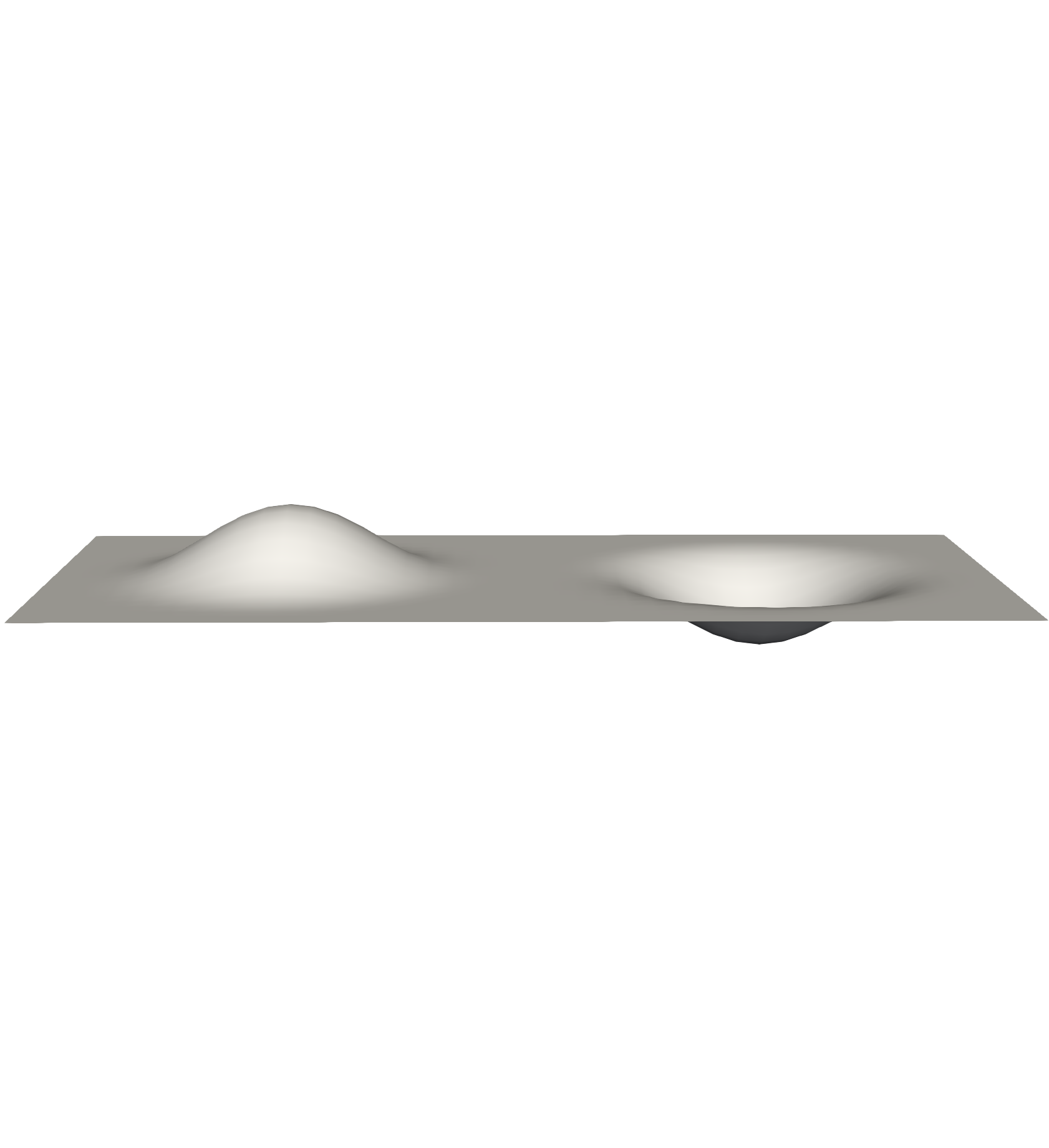}};
\node[inner sep=0pt] (first) at (6, -6.5) {\includegraphics[trim=10 700 10 700, clip,width=.4\textwidth]{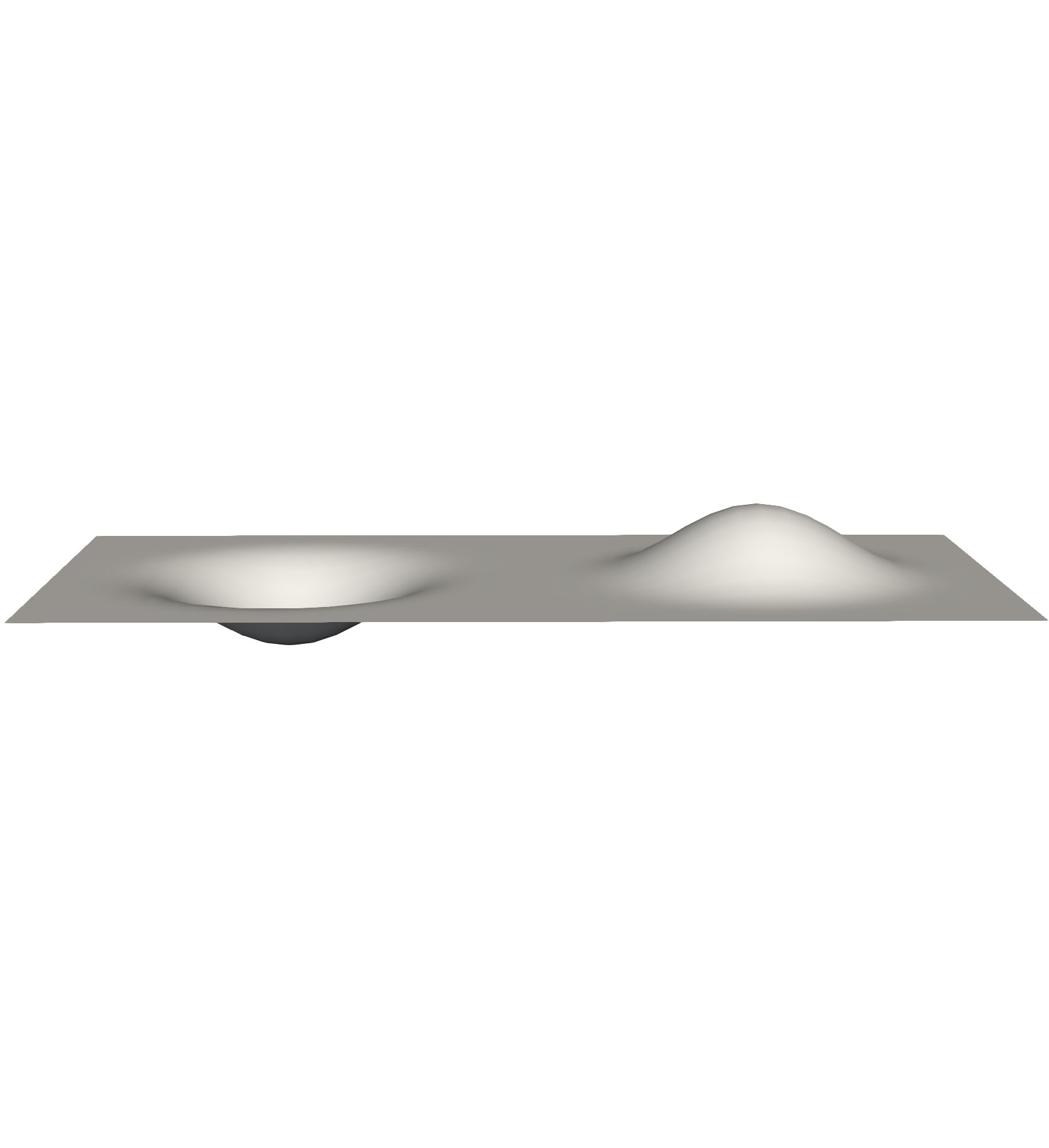}};
\draw [myGreen, line width=3pt, opacity = 0.3, domain=0:90] plot ({-6 + 12*cos(\x)}, {-6 + 12*sin(\x)});
\draw [myOrange, line width=3pt, opacity = 0.5, domain=180:270] plot ({6 + 12*cos(\x)}, {6 + 12*sin(\x)});
\node[inner sep=0pt] (second) at (-1.5,4.3) {\includegraphics[trim=10 700 10 700, clip,width=.4\textwidth]{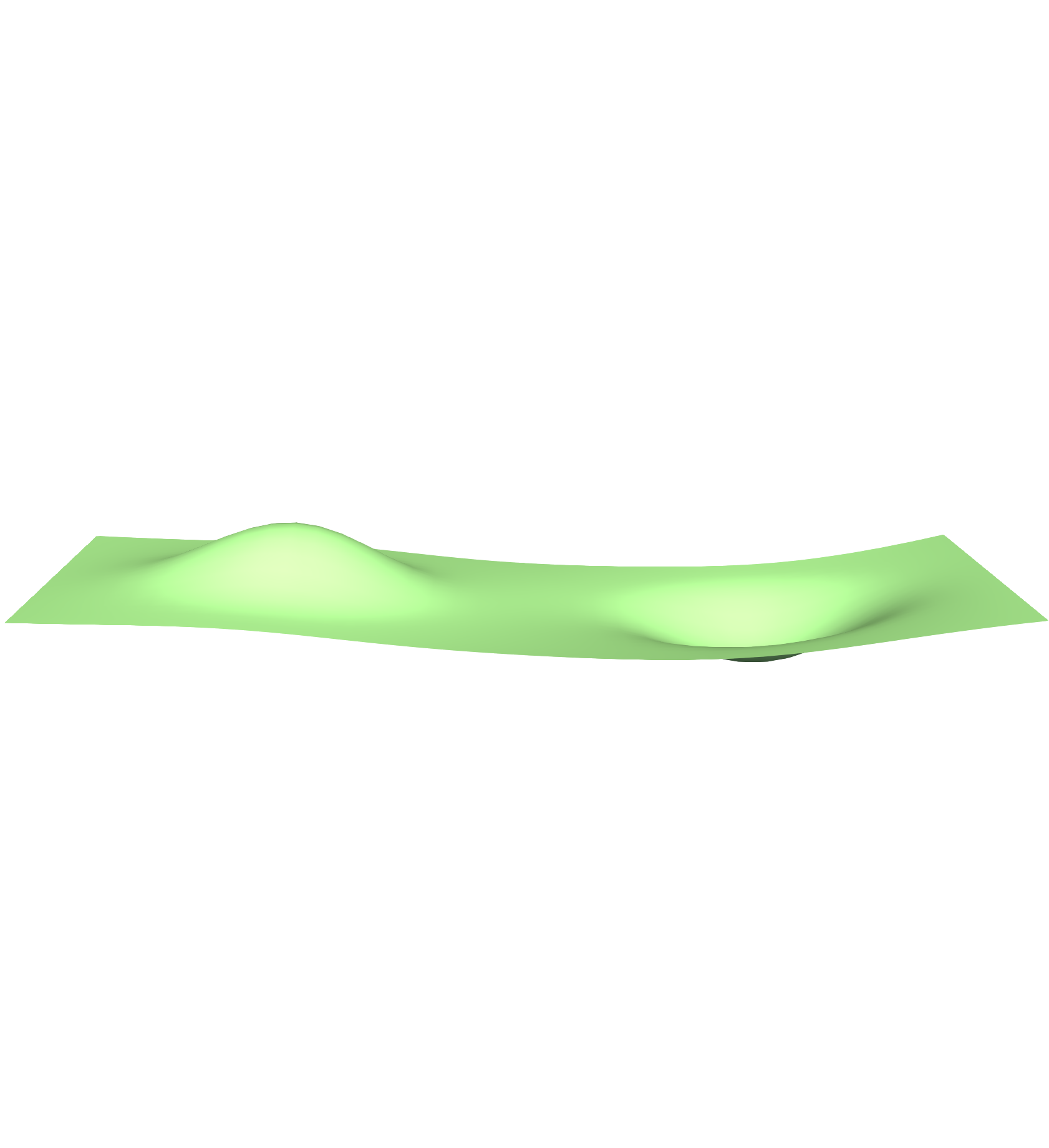}};
\node[inner sep=0pt] (third) at (3,1.75) {\includegraphics[trim=10 700 10 700, clip,width=.4\textwidth]{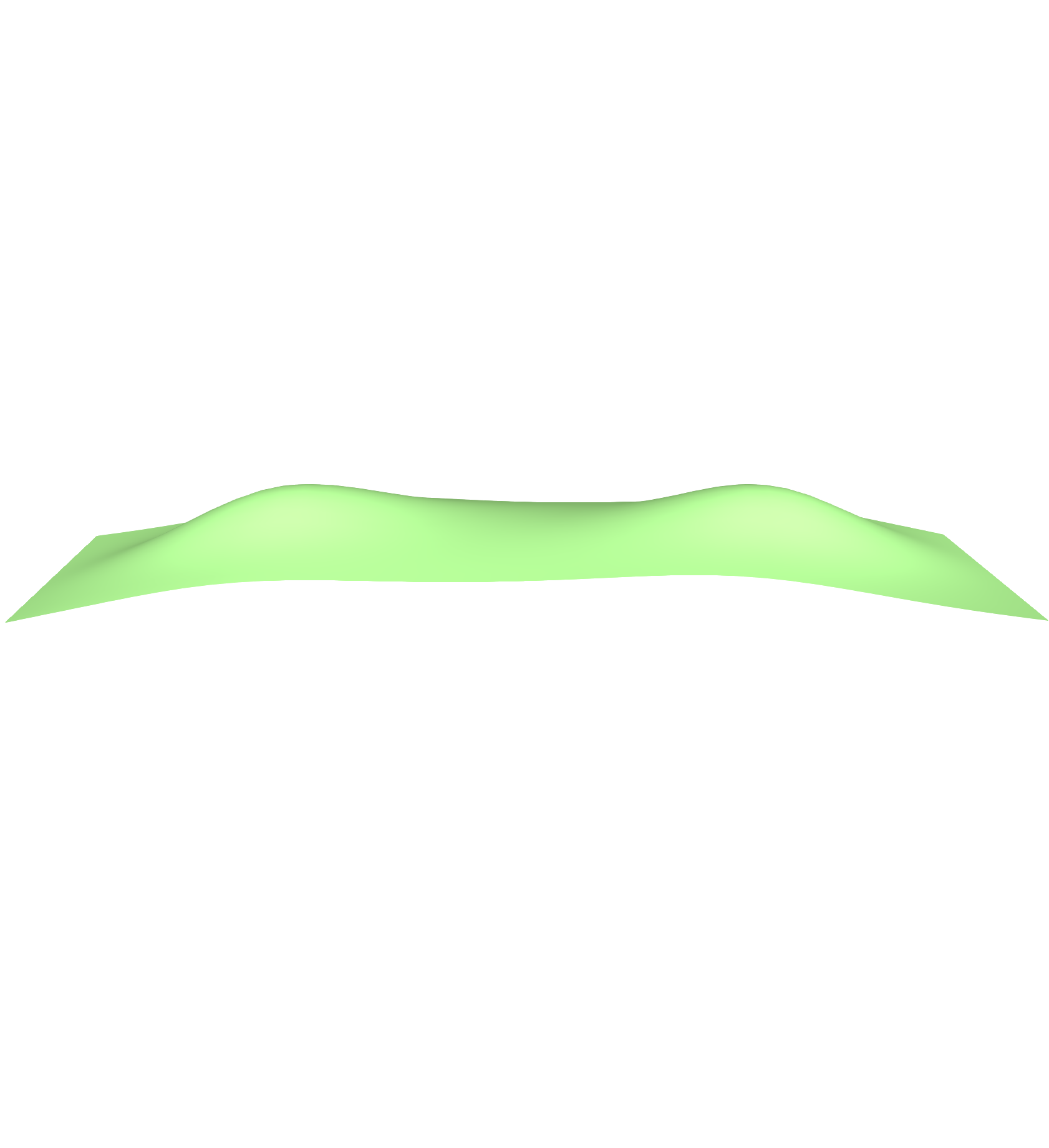}};
\node[inner sep=0pt] (fourth) at (4.9,-2.35) {\includegraphics[trim=10 700 10 700, clip,width=.4\textwidth]{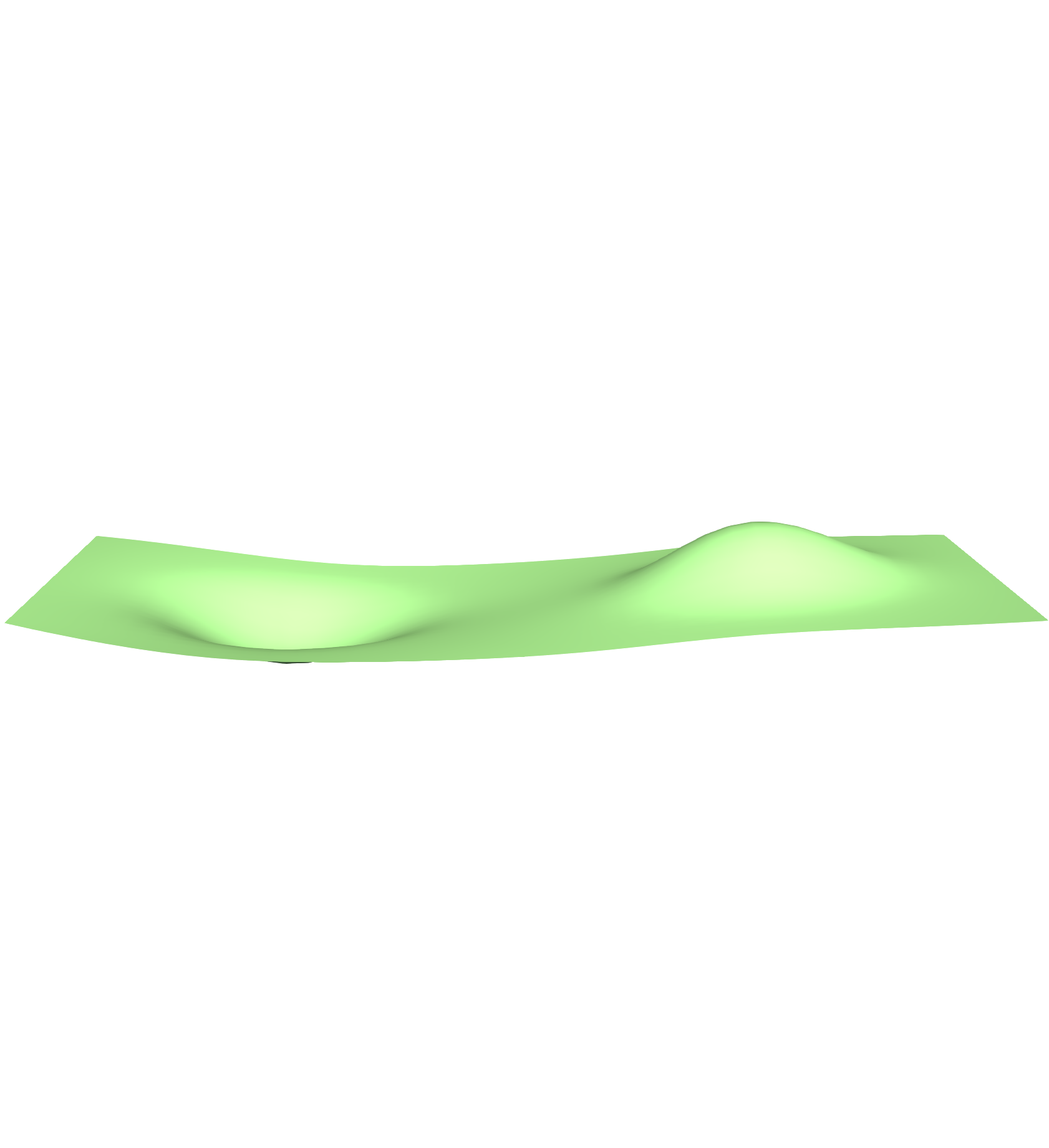}};
\node[inner sep=0pt] (second) at (-5.2,2.05) {\includegraphics[trim=10 700 10 700, clip,width=.4\textwidth]{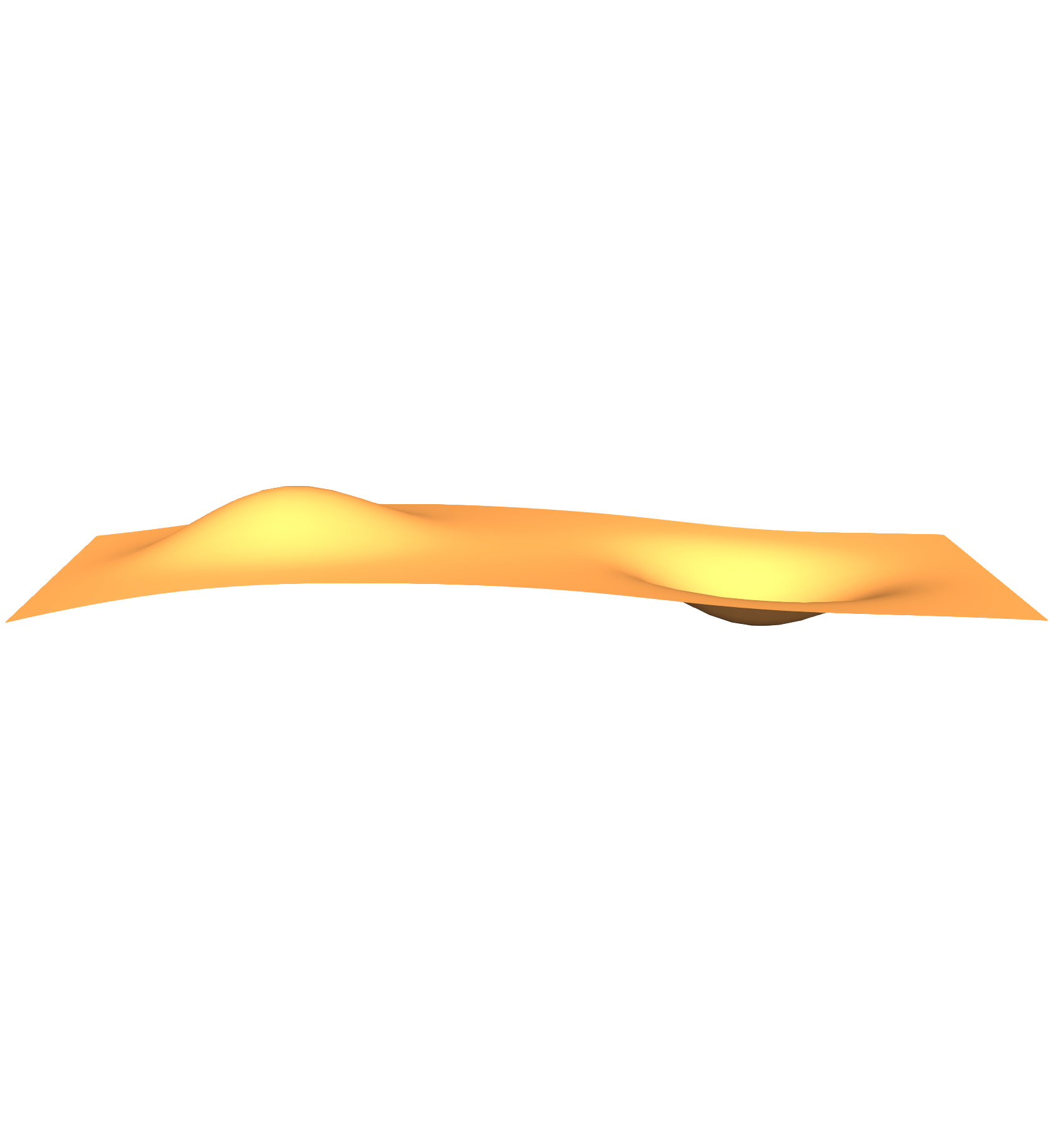}};
\node[inner sep=0pt] (third) at (-3, -1.65) {\includegraphics[trim=10 700 10 700, clip,width=.4\textwidth]{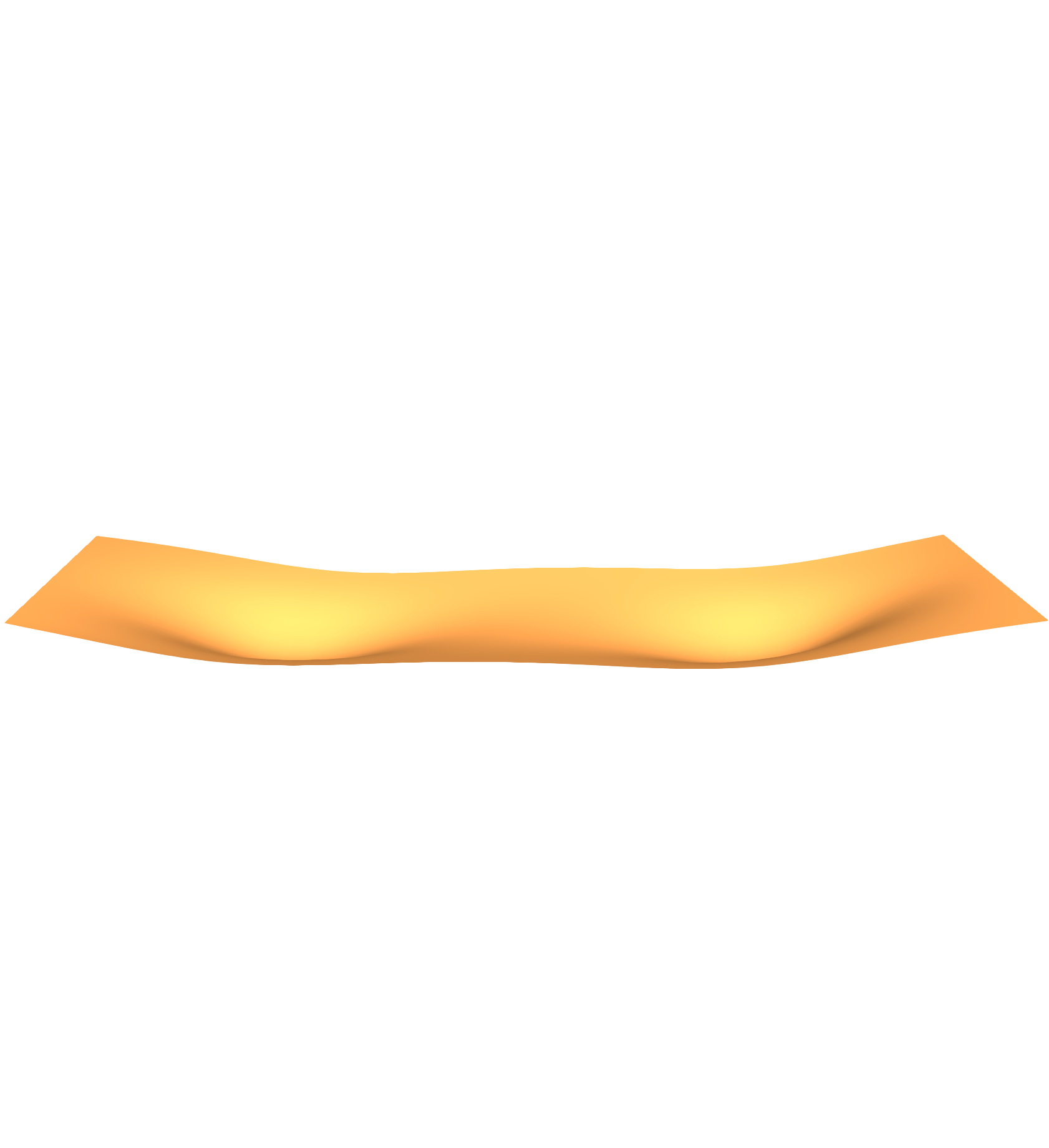}};
\node[inner sep=0pt] (fourth) at (1,-4.5) {\includegraphics[trim=10 700 10 700, clip,width=.4\textwidth]{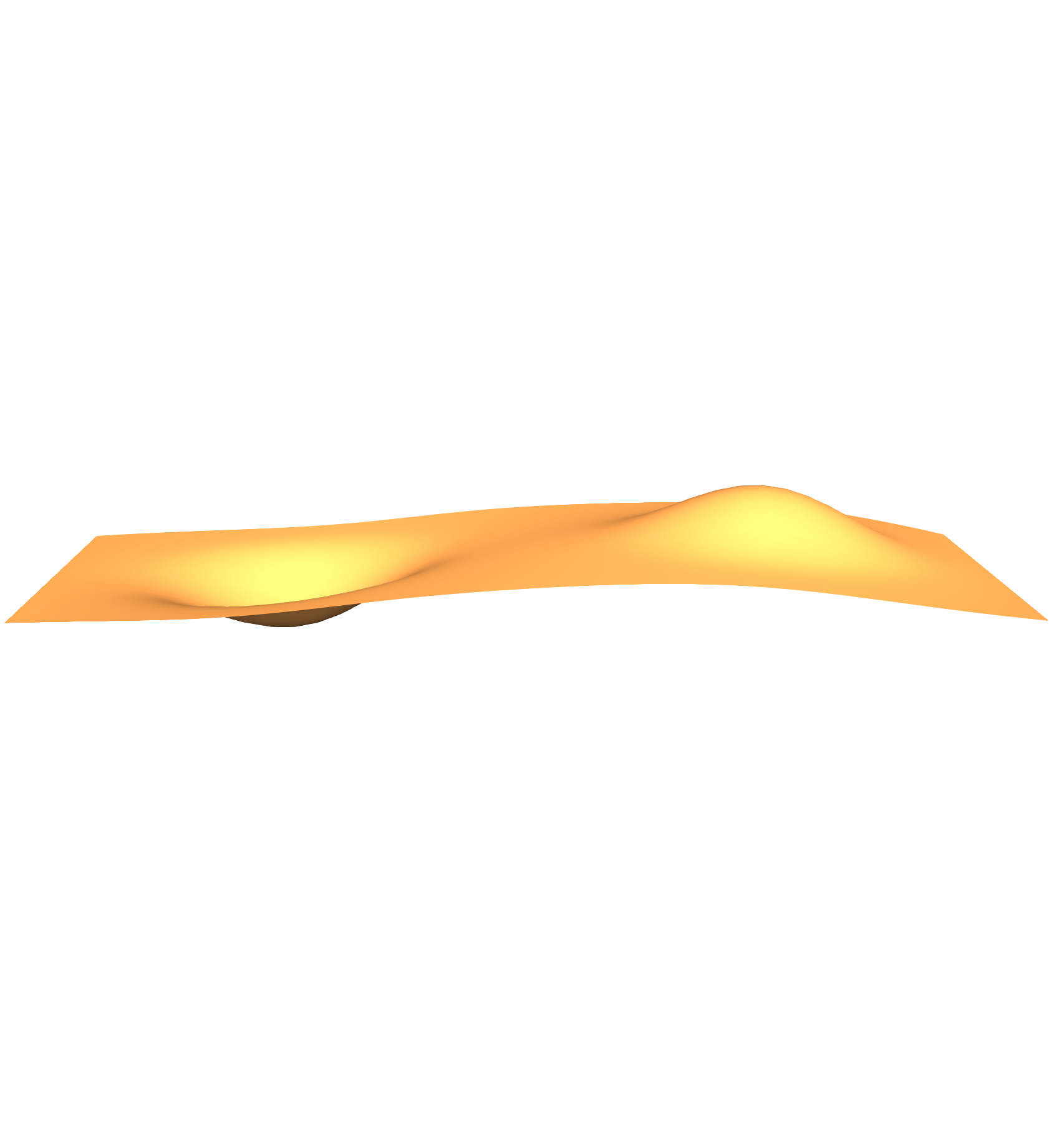}};
\node[inner sep=0pt] (fourth) at (7,4.5) {\includegraphics[trim=10 700 10 700, clip,width=.4\textwidth]{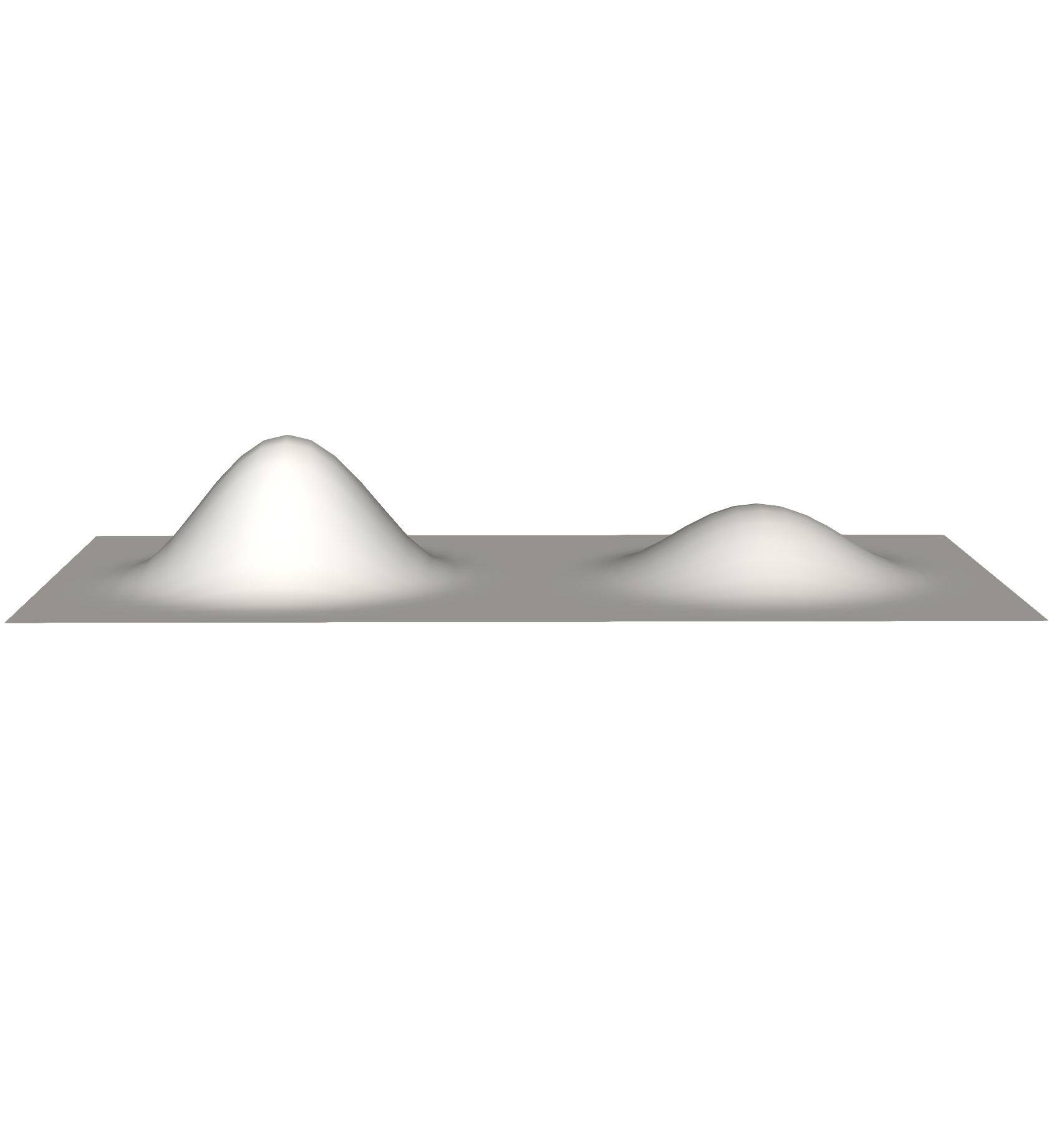}};
\draw[myGrey,fill=myGrey] ({-6 + 12*cos(0)}, {-6 + 12*sin(0)}) circle (.75ex);
\draw[myGrey,fill=myGrey] ({-6 + 12*cos(90)}, {-6 + 12*sin(90)}) circle (.75ex);
\draw[myGreen,fill=myGreen] ({-6 + 12*cos(22.5)}, {-6 + 12*sin(22.5)}) circle (.75ex);
\draw[myGreen,fill=myGreen] ({-6 + 12*cos(45)}, {-6 + 12*sin(45)}) circle (.75ex);
\draw[myGreen,fill=myGreen] ({-6 + 12*cos(67.5)}, {-6 + 12*sin(67.5)}) circle (.75ex);
\draw[myOrange,fill=myOrange] ({6 + 12*cos(202.5)}, {6 + 12*sin(202.5)}) circle (.75ex);
\draw[myOrange,fill=myOrange] ({6 + 12*cos(225)}, {6 + 12*sin(225)}) circle (.75ex);
\draw[myOrange,fill=myOrange] ({6 + 12*cos(247.5)}, {6 + 12*sin(247.5)}) circle (.75ex);
\draw [->,myDarkGrey, line width=2.5pt,domain=0:1] plot ({0}, {6 + \x});
\draw [->,myDarkGrey, line width=2.5pt,domain=0:1] plot ({0}, {-6 - \x});
\draw [->,myDarkGrey, line width=2.5pt,domain=0:1] plot ({7 + \x}, {0});
\draw [->,myDarkGrey, line width=2.5pt,domain=0:1] plot ({-7 - \x}, {0});
\node[inner sep=0pt, text=myDarkGrey] at (-0.5,7.0) {\huge $\zeta$};
\node[inner sep=0pt, text=myDarkGrey] at (-0.7,-7.0) {\huge $-\zeta$};
\node[inner sep=0pt, text=myDarkGrey] at (8,-0.5) {\huge $\eta$};
\node[inner sep=0pt, text=myDarkGrey] at (-8,-0.5) {\huge $-\eta$};
\draw [<->,red, line width=2.0pt,domain=0:1] plot ({5.5}, {4.3 + 0.96*\x});
\draw [<->,red, line width=2.0pt,domain=0:1] plot ({8.5}, {4.3 + 0.6*\x});
\node[inner sep=0pt, text=red] at (4.85,5.2) {\huge $\zeta$};
\node[inner sep=0pt, text=red] at (8.95,5.0) {\huge $\eta$};
\draw [myGrey, line width=2.0pt ] (3.5,3.5) rectangle (10.5,6.0);
\end{tikzpicture}
}
\caption{The two extremal grey shells are connected by two distinct shortest geodesics, the green and the orange one.
This behaviour is associated with a strongly positive sectional curvature $\kappa^{\pm\tau}=871.34$ (for $\tau=0.01$) at the central shell in the directions of $\zeta$ and $\eta$.}
\label{fig:positiveCurvature}
\end{figure}

\subsection*{Acknowledgements}
Alexander Effland acknowledges support from the European Research Council under the Horizon 2020 program, ERC starting grant HOMOVIS (No. 640156).
Behrend Heeren and Martin Rumpf acknowledge support of the Collaborative Research Center 1060 fun\-ded by 
the Deutsche Forschungsgemeinschaft (DFG, German Research Foundation)  
and the Hausdorff Center for Mathematics, funded by the DFG
under Germany's Excellence Strategy - GZ 2047/1, Project-ID 390685813.
Benedikt Wirth was supported by the Alfried Krupp Prize for Young University Teachers awarded by the Alfried Krupp von Bohlen und Halbach-Stiftung 
and by the Deutsche Forschungsgemeinschaft (DFG, German Research Foundation) via Germany's Excellence Strategy through the Cluster of Excellence ``Mathematics M\"unster: Dynamics -- Geometry -- Structure'' (EXC 2044) at the University of M\"unster.

\bibliographystyle{alpha}
\bibliography{bibtex/all,bibtex/own}

\end{document}